\documentclass[12pt,leqno,oneside]{amsart}

\usepackage{amsmath,amsthm,amssymb,amsfonts,ifpdf,xcomment,array,longtable}

\ifpdf
\usepackage[pdftex,pdfstartview=FitH,pdfborderstyle={/S/B/W 1},%
colorlinks=true, linkcolor=blue, urlcolor=blue, citecolor=blue,%
pagebackref=true,pdfpagelayout=OneColumn]{hyperref}
\else
\usepackage[hypertex]{hyperref}
\fi

\usepackage[abbrev,msc-links,nobysame]{amsrefs}

\usepackage{todonotes}

\newtheorem {thm}   {Theorem}
\newtheorem* {thm*}   {Theorem}
\newtheorem* {prp*}   {Proposition}
\newtheorem {lem}      [thm]    {Lemma}
\newtheorem {cnj}      [thm]    {Conjecture}
\newtheorem {que}      [thm]    {Question}
\newtheorem {cor}  [thm] {Corollary}
\newtheorem* {cor*} {Corollary}
\newtheorem {prp}[thm]  {Proposition}
\newtheorem {dfn} [thm]    {Definition}
\newtheorem {rmk} [thm]    {Remark}
\newtheorem* {rmk*} {Remark}
\newcounter{AbcT}

\theoremstyle{definition}

\numberwithin{equation}{section}
\numberwithin{thm}{section}

\renewcommand{\a}{\alpha}
\renewcommand{\b}{\beta}
\renewcommand{\d}{\delta}

\newcommand{\e}{\varepsilon}
\newcommand{\f}{\varphi}
\newcommand{\g}{\gamma}

\renewcommand{\l}{\lambda}
\renewcommand{\o}{\omega}
\renewcommand{\O}{\Omega}
\newcommand{\s}{\sigma}

\newcommand{\R}{{\bf R}}

\newcommand{\Q}{{\bf Q}}

\newcommand{\Z}{{\bf Z}}
\newcommand{\C}{{\bf C}}

\newcommand{\F}{{\bf F}}
\renewcommand{\P}{{\bf P}}

\newcommand {\cF} {{\mathcal F}}

\newcommand {\cH} {{\mathcal H}}

\newcommand {\cA} {{\mathcal A}}
\newcommand {\cB} {{\mathcal B}}

\newcommand {\cP} {{\mathcal P}}
\newcommand {\cC} {{\mathcal C}}
\newcommand {\cD} {{\mathcal D}}
\newcommand {\cQ} {{\mathcal Q}}
\newcommand {\cE} {{\mathcal E}}
\newcommand {\cX} {{\mathcal X}}

\newcommand {\cR} {{\mathcal R}}
\newcommand {\cT} {{\mathcal T}}

\newcommand{\wt}{\widetilde}

\DeclareMathOperator{\Id}{Id}
\DeclareMathOperator{\dist}{dist}

\title[Self-similar measures]%
{Self-similar measures associated to a homogeneous system of three maps}
\author{Ariel Rapaport}
\address{Ariel Rapaport, Department of Mathematics, Technion, Haifa, Israel}
\email{arapaport@technion.ac.il}

\author{P\'eter P. Varj\'u}
\address{P\'eter P. Varj\'u, Centre for Mathematical Sciences, Wilberforce Road, Cambridge CB3 0WB, UK}
\email{pv270@dpmms.cam.ac.uk}

\thanks{
AR was supported by the Herchel Smith Fund at the university of Cambridge. PV has received funding from the European Research Council (ERC) under the European Union’s Horizon 2020 research and innovation programme (grant agreement No. 803711). PV was supported by the Royal Society.}


\begin{document}

\begin{abstract}
We study the dimension of self-similar measures associated to a homogeneous iterated function system of three contracting
similarities on $\R$ and other more general IFS's.
We extend some of the theory recently developed for Bernoulli convolutions to this setting.
In the setting of three maps a new phenomenon occurs, which has been highlighted by recent examples
of Baker, and B\'ar\'any, K\"aenm\"aki.
To overcome the difficulties stemming from this phenomenon, we develop novel techniques, including an extension of
Hochman's entropy increase method to a function field setup.
\end{abstract}

\maketitle

\section{Introduction}\label{sc:intro}

By an iterated function system (IFS), we mean a finite collection
\[
\Phi=\{f_j(x)=\l_j x+t_j:j=1,\ldots,m\}
\]
of contracting similarities on $\R$.
It is well known (see \cite{Hut}) that there
exists a unique nonempty compact $K_{\Phi}\subset\R$
so that
\[
K_{\Phi}=\cup_{j=1}^m f_{j}(K_{\Phi}).
\]
It is called
the attractor or self-similar set corresponding to $\Phi$.

Certain natural measures are associated to $\Phi$. Indeed, given a probability vector $p=(p_1,\ldots,p_m)$ there exists a unique Borel probability measure $\mu$ on $\R$ that satisfies
\[
\mu=\sum_{j=1}^m p_j \cdot f_j \mu
\]
(again, see \cite{Hut}), where $f_j \mu$ is the pushforward of $\mu$ under $f_j$.
The measure $\mu$ is called the self-similar measure associated to $\Phi$ and $p$. It is always supported on $K_\Phi$, and $\mathrm{supp}(\mu)=K_\Phi$ when $p$ is strictly positive.
We refer the reader to the surveys \cites{PSS-60y, Hoc-ICM, Var-ICM} for more background on self-similar
sets and measures.

In this paper, we are chiefly concerned with homogeneous IFS's, which means
$\l_1=\ldots=\l_m=\l$ for some $\l\in(0,1)$.
In this case, we have an alternative characterization of the measure $\mu$.
It is the law of the random variable
\[
\sum_{n=0}^{\infty} t_{\xi_n}\l^n,
\]
where $\xi_n$ is a sequence of independent random variables with law
\[
\P(\xi_n=j)=p_j \:\text{ for }\:j=1,\ldots,m.
\]

Determining the dimension of self-similar measures is a fundamental problem in fractal geometry.
There are several notions of dimension for measures, but many of them coincide for self-similar measures, since they
are exact dimensional, which was proved by Feng and Hu \cite{FH-dimension}.
We say that a measure $\mu$ is exact dimensional, if the limit
\[
\lim_{r\to 0}\frac{\log\mu([x-r,x+r])^{-1}}{\log r^{-1}}
\]
exists and is equal to a constant for $\mu$-almost all $x$.
In this case, we write $\dim \mu$ for this constant and call it the dimension of $\mu$.

If the IFS satisfies the open set condition, which roughly speaking says that the sets $\{f_j(K_\Phi)\}_{j=1}^m$ are nearly disjoint, then
the dimension of the associated self-similar measures can be computed as
\begin{equation}\label{eq:simdim}
\dim\mu = \frac{\sum p_j\log p_j^{-1}}{\sum p_j\log \l_j^{-1}}.
\end{equation}
A closely related analogous result for self-similar sets was first obtained by Moran \cite{Mor-self-similar}.
For a precise statement and proof in the setting of measures, we refer to \cite{Edg-integral}*{Theorem 5.2.5}.
It is also known that the right hand side of \eqref{eq:simdim} is always an upper bound for $\dim\mu$, see
\cite{Edg-integral}*{Corollary 5.2.3} for a proof.

It is natural to speculate if the formula \eqref{eq:simdim} may hold more generally without assuming the open
set condition.
However, there are two obvious obstructions.
First, $\dim\mu$ cannot exceed $1$, the dimension of the underlying space.
Second, the right hand side of \eqref{eq:simdim} depends on the IFS and it may yield different values if we realize $\mu$
as the self-similar measure associated to different IFS's.
For example, one may artificially increase the entropy of the probability vector appearing on the right hand side of \eqref{eq:simdim}
by repeating the same map more than once in the IFS.
A more subtle instance of the same phenomenon is formalized in the following definition.

\begin{dfn}\label{df:EO}
We say that an IFS $\Phi=\{f_1,\ldots,f_m\}$ contains exact overlaps if there are
$\o\neq\o'\in\{1,\ldots,m\}^n$ for some $n\in\Z_{>0}$ such that
\[
f_{\o_1}\circ\ldots\circ f_{\o_n}=f_{\o'_1}\circ\ldots\circ f_{\o'_n}.
\]
In other words, the IFS does not contain exact overlaps if and only if the maps generate a free semi-group.
\end{dfn}

It is a folklore conjecture, going back to a question of K\'aroly Simon in the setting of self-similar sets (see \cite{Simon_EO_conjecture}), that there are no
obstructions to \eqref{eq:simdim} other than the two we discussed above.
More formally, this can be stated as follows.

\begin{cnj}\label{cn:EO}
Let $\mu$ be the self-similar measure associated to the IFS $\Phi$ and the probability vector
$(p_1,\ldots,p_m)$.
If $\Phi$ does not contain exact overlaps then
\[
\dim\mu=\min\Big\{\frac{\sum p_j\log p_j^{-1}}{\sum p_j\log \l_j^{-1}},1\Big\}.
\]
\end{cnj}

This conjecture has been proved recently in several cases starting with the breakthrough of Hochman \cite{Hoc-self-similar}.
Hochman proved the conjecture if the IFS satisfies the exponential separation condition, in particular if
all the parameters $\l_j$ and $t_j$ are algebraic numbers, see \cite{Hoc-self-similar}*{Theorem 1.1}
for the precise statement.
This has been extended by the first author \cite{Rap-EO} to the case when only the contraction parameters $\l_j$ are assumed to be algebraic.
The second author \cite{Var-Bernoulli} proved the conjecture for Bernoulli convolutions with transcendental parameters, and together with
Hochman's results, this proves the conjecture when $m=2$ and $\l_1=\l_2$.

Our aim in this paper is to extend these results to the case $m=3$ and $\l_1=\l_2=\l_3$.
While we are unable to achieve this goal, we extend the results of Breuillard and the second author \cite{BV-transcendent},
which is one of the main ingredients in \cite{Var-Bernoulli}.
As an application, we improve Hochman's bound for the dimension of the set of exceptional parameters for which Conjecture
\ref{cn:EO} fails.
We also prove the conjecture conditionally on a suitable generalization of another result of Breuillard and the second author.
We prove the conjecture unconditionally if the common contraction factor of the maps is at least $2^{-2/3}$.

We move on to state the results of our paper in the case $m=3$, $\l_1=\l_2=\l_3$ and $p_1=p_2=p_3=1/3$.
In fact, we prove them in greater generality, but for the sake of simplicity, we postpone introducing
the general setting to Section \ref{sc:general-setting}.
By an affine change of coordinates, we can conjugate the IFS so that $t_1=0$, $t_2=1$ and $t_3=\tau$ for
some $\tau\in\R$.
Therefore, in this section and the next one, we work with the IFS
\begin{equation}\label{eq:3maps}
f_1(x)=\lambda x,\quad f_2(x)=\lambda x+1, \quad f_3(x)=\lambda x +\tau
\end{equation}
and with the probability vector $p_1=p_2=p_3=1/3$.

Our first result states that if Conjecture \ref{cn:EO} fails for the IFS \eqref{eq:3maps} for some
parameters $(\lambda,\tau)$, then they are approximated by algebraic parameters with very small error
such that the IFS contains a lot of exact overlaps for these approximating parameters.

Before we can state this, we need to introduce a quantity to measure the amount of exact overlaps.
For a set $U\subset(0,1)\times\R$ and $n\in\Z_{\ge 0}$, we consider the random function
$A_{U}^{(n)}:U\to\R$ defined by
\[
A_{U}^{(n)}(\l,\tau)=\sum_{k=0}^{n-1}T_{\xi_k}(1,\tau)\l^k,
\]
where $\xi_k$ are independent random variables taking the values $1$, $2$ and $3$ with equal probability, and
$T_1(Y_1,Y_2)$, $T_2(Y_1,Y_2)$ and $T_3(Y_1,Y_2)$ are the linear forms $0$, $Y_1$ and $Y_2$, respectively.
(The reason for this apparently strange notation will become clear later.)
The entropy rate (also known as Garsia entropy in the literature) is defined as
\[
h(U)=\lim_{n\to\infty}\frac{1}{n}H(A_{U}^{(n)})=\inf\frac{1}{n}H(A_{U}^{(n)}),
\]
where $H(\cdot)$ denotes Shannon entropy.
We note that the sequence $n\mapsto H(A_{U}^{(n)})$ is subadditive, hence the above limit exists and it equals the
infimum.

The case when $U=\{(\l,\tau)\}$ is a single point is of special interest, and we will write $A_{\l,\tau}^{(n)}$
and $h(\l,\tau)$ to simplify notation.
It is immediate from the definition that
\[
h(\l,\tau)\leq\log 3
\]
and equality holds if and only if the IFS \eqref{eq:3maps} does not contain exact overlaps for $\l$ and $\tau$.
The entropy drop $\log 3-h(\l,\tau)$ provides a way to quantify the amount of exact overlaps.

We will see that exact overlaps occur in the IFS \eqref{eq:3maps} if and only if $(\l,\tau)$ satisfy a certain kind
of polynomial equation.
These equations vanish along certain curves in the parameter space, which play an important role in the theory,
so we introduce notation for them.

We write $\cR$ for the set of meromorphic functions on the open unit disc that are ratios of two power series with coefficients
$-1$, $0$ and $1$.
We denote by $\Gamma$ the set of curves $\gamma\subset (0,1)\times\R$
that are of one of the following two forms
\begin{enumerate}
\item $\gamma=\{(\l,\tau)\in(0,1)\times\R:\tau=R(\l)\}$ for some $R\in\cR$,
\item $\gamma=\{(\l_0,\tau):\tau\in\R\}$ for some fixed $\l_{0}\in(0,1)$. 
\end{enumerate} 
If $\gamma$ is of the first form, we call it non-degenerate, otherwise it is degenerate.
We will see that the zero set of the polynomial equations mentioned above are always finite unions
of curves in $\Gamma$, but not all elements of $\Gamma$ occur in this way.

Given $n,l\ge1$ we denote by $\cP_l^{(n)}$ the set of polynomials in $\Z[X]$ with degree strictly less than $n$ and with coefficients bounded in absolute value by $l$. Now we are ready to state our first result.

\begin{thm} \label{th:BV2}
Let $(\lambda,\tau)\in(0,1)\times\R$ and let $\mu$ be the self-similar measure associated
to the IFS \eqref{eq:3maps} with equal probability weights.
Suppose the IFS does not contain exact overlaps and
\[
\dim\mu < \min\Big\{\frac{\log 3}{\log\l^{-1}},1\Big\}.
\]
Then for every $\e>0$ and $N\ge1$, there
exist $n\ge N$ and $(\eta,\sigma)\in (0,1)\times\R$ such
that,
\begin{enumerate}
\item $|\lambda-\eta|,|\tau-\sigma|\le\exp(-n^{1/\e})$;
\item $\frac{1}{n\log\eta^{-1}} H(A_{\eta,\s}^{(n)})\le \dim\mu+\e$;
\item $h(\gamma)\ge \min\{\log3,\log\l^{-1}\}-\e$ for all $\g\in\Gamma$ with $(\eta,\s)\in\gamma$.
\end{enumerate}
In particular, $\eta$ is a root of a nonzero polynomial in $\cP_{2n}^{(2n)}$ and $\s$ can be written in the form
$\sigma=P_{1}(\eta)/P_{2}(\eta)$ for some $P_{1},P_{2}\in\cP_{1}^{(n)}$ such that $P_{2}(\eta)\ne0$.
\end{thm}

We will see later, that items (2) and (3)
in the conclusion imply that $(\eta,\s)$ satisfy at least two independent polynomial equations
associated to exact overlaps provided $\e$ is small enough.
From this, we can deduce the algebraicity of $\eta$ and $\s$ and this is how the last claim of the theorem will follow.

Theorem \ref{th:BV2} is a direct analogue of the main result in \cite{BV-transcendent}, which is one of the key ingredients
of the proof of Conjecture \ref{cn:EO} for Bernoulli convolutions in \cite{Var-Bernoulli}.
To be precise, our approximation is slightly weaker than the analogous result in \cite{BV-transcendent} 
in that the exponent $1/\e$ is replaced by an explicit function of $n$ slowly tending to infinity in \cite{BV-transcendent}.
This change is not likely to be of any significance in applications, and it allows for a slightly simpler proof.

The following corollary is an easy consequence of Theorem \ref{th:BV2}, as we will see in Section \ref{sc:prf of apx thm}.

\begin{cor}\label{cr:Mike}
The set of parameters $(\l,\tau)\in(0,1)\times\R$ for which Conjecture \ref{cn:EO} fails for the IFS \eqref{eq:3maps} with equal probability weights is of Hausdorff dimension $0$.
\end{cor}

This improves the result of Hochman \cite{Hoc-self-similar}, where the bound on the dimension of the exceptional parameters is $1$.
However, in \cite{Hoc-self-similar} the result is proved for packing dimension instead of Hausdorff dimension, and so our result is
weaker in this respect.

In addition, Theorem \ref{th:BV2} provides an abundance of explicit examples of transcendental parameters $(\l,\tau)$
for which Conjecture \ref{cn:EO} holds.
Indeed there are many results in the literature about transcendence measures of classical constants that excludes an
approximation as in the conclusion of Theorem \ref{th:BV2}.
In particular, Conjecture \ref{cn:EO} holds if $\l\in\{\ln 2, e^{-1/4},\pi/4\}$ and $\tau$ is arbitrary.
For more examples and discussion, we refer to the comments after Corollary 5 in \cite{BV-transcendent} and the references therein.

Another key ingredient in \cite{Var-Bernoulli} is the main result of \cite{BV-entropy}, which provides an estimate for the Mahler measure
of a parameter $\l$ if the corresponding Bernoulli convolution has lots of exact overlaps.
Mahler measure is a widely used quantity in number theory to measure the complexity of an algebraic number.
We recall the definition.
Let $\l$ be an algebraic number with minimal polynomial
\[
P(x)=a_n(x-\l_1)\ldots(x-\l_n)\in\Z[x]
\]
so that $a_n$ is the leading coefficient, and $\l_1,\ldots,\l_n$ are the roots (including $\l$).
Then we define the Mahler measure of $\l$ as
\[
M(\l)=|a_n|\prod_{j:|\l_j|>1}|\l_j|.
\]
For convenience, we define $M(\lambda)=\infty$ if $\lambda$ is transcendental.

An analogue of the main result of \cite{BV-entropy} would give an affirmative answer to the following question.

\begin{que}\label{q:BV1}
Is it true that for all $\e>0$, there is $M$ such that the following holds?
Let $(\l,\tau)\in(\e,1-\e)\times\R$ be such that $h(\l,\tau)\le \min\{\log 3,\log\lambda^{-1}\} -\e$ and $h(\g)\ge\min\{\log 3,\log\lambda^{-1}\}-M^{-1}$
for all $\g\in\Gamma$ with $(\l,\tau)\in\gamma$.
Then $M(\l)\le M$.
\end{que}

We note that a condition about the entropy rate of curves passing through $(\l,\tau)$ is necessary.
Indeed, we have, for example, $h(\gamma)=\log3-(2/3)\log 2$ for the non-degenerate curve associated to the function
$R(X)=1$, and hence $h(\lambda,1)\le \log 3-(2/3)\log 2$ for any $\l\in(0,1)$ even if $\l$
is transcendental.
We need to rule out examples of this type.

The second main result of the paper is the following.

\begin{thm}\label{th:EO}
Suppose that the answer to Question \ref{q:BV1} is affirmative.
Then Conjecture \ref{cn:EO} holds for the IFS \eqref{eq:3maps} for any
$(\l,\tau)\in(0,1)\times\R$ with equal probability weights.
\end{thm}

We do not know the answer to Question \ref{q:BV1}, but we can easily deduce the following weaker statement from the results of
\cite{BV-entropy}.

\begin{thm}\label{th:BV1}
For all $\e>0$, there is $M$ such that the following holds.
Let $(\l,\tau)\in(0,1)\times\R$ be such that $h(\l,\tau)< (2/3)\log 2-\e$.
Then $M(\l)\le M$.
\end{thm}

The reader may be alarmed by the absence of any condition on the entropy rates of curves passing through
$(\l,\tau)$.
The explanation for this is the fact, which is a consequence of the above theorem, that there is no non-degenerate curve $\gamma\in\Gamma$
with $h(\gamma)<(2\log 2)/3$.
Therefore, the type of counterexamples that we discussed above will not arise.

As a consequence of this result and the techniques developed for the proof of Theorem \ref{th:EO},
we will prove the following theorem.

\begin{thm}\label{th:large-lambda}
Conjecture \ref{cn:EO} holds for the IFS \eqref{eq:3maps} for all $(\l,\tau)\in(2^{-2/3},1)\times \R$
with equal probability weights.
\end{thm}

We have finished stating the main results of this paper in the setting of IFS's consisting of three maps.
With the exceptions of Theorems \ref{th:BV1} and \ref{th:large-lambda}, we prove our results in greater
generality.
We introduce this more general framework in Section \ref{sc:general-setting}.

Theorems \ref{th:BV2}, \ref{th:EO} and \ref{th:large-lambda} are not routine generalizations of the
corresponding results in \cite{BV-transcendent} and \cite{Var-Bernoulli}.
In the setting of three maps, new phenomena arise leading to very significant new difficulties in the arguments.
Some of these have been highlighted in recent examples of Baker \cite{Bak-super-exponential} (see also \cite{Baker_sup_exp_II}) and
B\'ar\'any, K\"aenm\"aki \cite{BK-super-exponential} of
IFS's without exact overlap that have very bad separation properties.
Other similar examples have later been obtained by Chen \cite{Chen_sup_exp}.
We believe that the main contributions of this paper are not the above results themselves but the new concepts and
techniques we introduce to address the new difficulties.
We discuss the outline of our arguments and highlight the new techniques in the next section.

\subsection*{Acknowledgment}

We are indebted to Vesselin Dimitrov for suggesting the use of Lemmata \ref{lm:Dimitrov} and \ref{lm:Dimitrov-precise}
to us and for his permission to include them with their proofs in our paper.
We thank the anonymous referees for very helpful comments and suggestions.

\section{An overview of the paper}\label{sc:outline}

We outline the proofs of the results stated in the previous section and highlight the main new difficulties
that arise when we extend the theory from the case of Bernoulli convolutions.
We end the section by setting out the organization of the paper.

In this section, we do not give full details, and some of our discussions are imprecise.
Our aim is to guide the reader in understanding the roles played by the various parts of the proofs.
Everything will be repeated in a rigorous manner in later parts of the paper.

\subsection{Ideas from the proof of Theorem \ref{th:BV2}}\label{sc:ideas-th3}

We begin by discussing the proof of Theorem \ref{th:BV2} comparing it with the case of Bernoulli convolutions.

Bernoulli convolutions are parametrized by a single parameter $\l\in(0,1)$.
They are defined as the laws of the random variables
\[
\sum_{n=0}^{\infty}\pm\l^n,
\]
where the $\pm$ are independent unbiased random variables.

For Bernoulli convolutions, exact overlaps occur if and only if $\l$ is a root of a polynomial with coefficients $\pm1,0$.
Mahler \cite{Mah-discriminant} proved that two such numbers of degree at most $n$  are separated by at least
$n^{-Cn}$ for an absolute constant $C$.
Provided no exact overlaps occur for $\l$,
this fact can be used to find values of $n$ such that any two numbers of the form $\sum_{j=0}^{n-1}\pm\l^j$
are separated by at least $n^{-Cn}$, where $C$ is a (possibly different) constant depending on the distances of $\l$ to $0$ and $1$.

In brief, this argument goes as follows.
We pick some value of $n$.
If there are two  numbers of the form $\sum_{j=0}^{n-1}\pm\l^j$ of distance smaller than $n^{-Cn}$, then
a simple argument gives us a number $\eta_n$ of distance at most $n^{-Cn}$ to $\l$ (for some other constant $C$) that
is a root of a polynomial with coefficients $\pm1$.
Now we can use the above quoted result of Mahler to conclude that $\eta_n=\eta_{n+1}=\ldots$, provided
there are two numbers of the form  $\sum_{j=0}^{k-1}\pm\l^j$ at distance less than $k^{-Ck}$ for each $k=n,n+1,\ldots$.
This chain must break at some point, as otherwise we would have $\l=\eta_n$, which means exact overlaps would occur.
This argument, or more precisely a variant of it, is a critically important ingredient in the paper \cite{BV-transcendent}, where
the analogue of Theorem \ref{th:BV2} is proved for Bernoulli convolutions.

The above considerations cannot be adapted to the setting of the IFS \eqref{eq:3maps}.
To explain this, we first take a closer look at the polynomial equations describing exact overlaps.
In the setting of the IFS \eqref{eq:3maps}, the analogue of the numbers of the form $\sum_{j=0}^{n-1}\pm\l^n$
are the numbers of the form
\begin{equation}\label{eq:finite-sum-3maps}
\sum_{j=0}^{n-1}T_{\o_j}(1,\tau)\l^j,
\end{equation}
where $T_1,T_2,T_3$ are the same linear forms as in the previous section and $\o_j\in\{1,2,3\}$ for each $j$.
It can be seen that exact overlaps occur in the IFS \eqref{eq:3maps} if and only if $(\l,\tau)$ satisfies a non-trivial
polynomial equation of the form $Q(\l,1,\tau)=0$ for some $Q\in\cQ^{(n)}\subset\Z[X,Y_1,Y_2]$ and $n\ge1$, where
\[
\cQ^{(n)}=\Big\{\sum_{j=0}^{n-1}(T_{\o_j}(Y_1,Y_2)-T_{\o'_j}(Y_1,Y_2))X^j:\o,\o'\in\{1,2,3\}^n\Big\}.
\]
The set of solutions of an equation of the form $Q(\l,1,\tau)=0$ for $0\neq Q\in\cQ^{(n)}$ is a finite union of curves
in $\Gamma$ (defined in the previous section), at most one of which is non-degenerate.
(However, not all curves in $\Gamma$ occur in this way.)

It can also be seen that the existence of two numbers of small distance of the form \eqref{eq:finite-sum-3maps}
is equivalent to the existence of a polynomial $0\neq Q\in\cQ^{(n)}$ such that $Q(\l,1,\tau)$ is small.
Furthermore, this is equivalent to the existence of a curve $\gamma_n\in\Gamma$ passing near $(\l,\tau)$ on which $Q$ vanishes.
Unfortunately, some of the curves that arise in this way intersect each other, and for this reason there is no
lower bound on their distances.
This means that we are not able to conclude $\gamma_n=\gamma_{n+1}$ analogously to the argument for Bernoulli convolutions that we sketched above no matter how fast distances between points of the
form \eqref{eq:finite-sum-3maps} go to $0$.
In fact, examples given by Baker \cite{Bak-super-exponential} and B\'ar\'any, K\"aenm\"aki \cite{BK-super-exponential}
show that these distances may go to $0$ arbitrarily fast.

Luckily, the proof in \cite{BV-transcendent} does not require that none of the distances
between points of the form $\sum_{j=0}^{n-1}\pm\l^n$
is small.
It is enough to establish that only few of these distances are small.

To quantify this properly, we introduce the notion of entropy of a random variable $A$ at a scale $r\in\R_{>0}$ defined as
\begin{equation}\label{avg ent scl r}
H(A;r)=\int_{0}^1H(\lfloor r^{-1}A+t\rfloor)dt,
\end{equation}
where $H(\cdot)$ on the right stands for Shannon entropy. If $\mu$ is the distribution of $A$ we write $H(\mu;r)$ in place of $H(A;r)$.
This quantity expresses the entropy of $A$ with respect to a partition of $\R$ into intervals of length $r$ averaged
over translated copies of $A$.
Hochman \cite{Hoc-self-similar} used the same quantity without averaging.

We write
\[
A_{\l,\tau}^{(n)}=\sum_{j=0}^{n-1}T_{\xi_j}(1,\tau)\l^{j},
\]
where $\xi_0,\xi_1,\ldots$ is a sequence of independent random variables taking the values $1,2,3$ with equal probabilities.
Note that the distribution of $A_{\l,\tau}^{(n)}$ converges weakly to the self-similar measure.
To extend the argument from \cite{BV-transcendent} to the setting of the IFS \eqref{eq:3maps}, we need to find several values of $n$
such that
\begin{equation}\label{eq:initial-bound}
H(A_{\l,\tau}^{(n)};\exp(-Cn\log^2n))\ge hn,
\end{equation}
for a suitable $h$ that is only slightly below $\min\{\log\l^{-1},\log 3\}$.

Suppose that \eqref{eq:initial-bound} fails for some particular value of $n$.
In this case, there are many pairs of points in the range of $A_{\l,\tau}^{(n)}$
that are of distance less than $\exp(-Cn\log^2n)$.
To each such pair, their is a corresponding polynomial $Q\in\cQ^{(n)}$
such that $Q(\l,1,\tau)<\exp(-Cn\log^2n)$.
We will show that this collection of polynomials must satisfy one of two possible scenarios.
First, it may happen that all of these polynomials vanish on a curve $\gamma\in\Gamma$ that may be degenerate
or non-degenerate.
This curve passes in an arbitrarily small neighborhood of $(\l,\tau)$ if $n$ is sufficiently large.
Moreover, we also have $h(\gamma)\le h$.
The second alternative is that there is a point $(\eta,\sigma)$ near $(\l,\tau)$ where all of the polynomials in the collection
vanish, but there is no curve $\gamma\in\Gamma$ containing $(\eta,\sigma)$ where all the polynomials vanish.
In this case, $\eta$ and $\sigma$ can be shown to be roots of polynomials of degree at most $2n$ whose coefficients
can be controlled, and they satisfy a separation property similar to parameters of Bernoulli convolutions with exact overlaps.
(See Section \ref{sc:common-root-Q} for more details.)

If we can show that the second of the above two alternatives always holds whenever $n$ is sufficiently large, then we can carry out an argument analogous
to that we described in the beginning of the section in the setting of Bernoulli convolutions.
To achieve this, we prove the following result, which is where the main novelty lies in our proof of Theorem \ref{th:BV2}.

\begin{prp}\label{pr:no-bad-curves-3maps}
Let $(\l,\tau)\in(0,1)\times\R$ be such that the IFS \eqref{eq:3maps} contains no exact overlaps.
Then for all $h<\min\{\log \l^{-1},\log 3\}$, there is a neighborhood of $(\l,\tau)$ that
is not intersected by a curve $\gamma\in\Gamma$ with $h(\gamma)\le h$.
\end{prp}

Our strategy to prove this result considers the limiting distribution of $A_{\gamma}^{(n)}$ on a suitably chosen
space, and introduces a notion of dimension for it.
We will show that the dimension depends lower semi-continuously on $\gamma$ with respect to a suitably chosen
metric, and we will also show that the dimension equals $h(\gamma)$.
These two properties together imply that $h(\gamma)$ depends lower semi-continuously on $\gamma$.

Equipped with these tools we can then argue as follows.
Suppose there is a sequence $\gamma_n$ of curves that passes in smaller and smaller neighbourhoods of a point $(\l,\tau)$
and $h(\gamma_n)\le h$ for some $h$ and for all $n$.
By a compactness argument, we will be able to pass to a subsequence converging to some limit curve $\gamma$.
We will conclude that $h(\gamma)\le h$ and $(\l,\tau)\in\gamma$, hence $h(\l,\tau)\le h$.
This shows that the IFS \eqref{eq:3maps} has exact overlaps for this $(\l,\tau)$ if $h<\log 3$.
In fact, because of some additional difficulties, which we will point out below, the argument works only if we have $h<\log\l^{-1}$
in addition.

We introduce our setup.
We denote by $\F[[X]]$ the ring of formal power series over a field $\F$.
We assume throughout that $\F$ is countable and of characteristic $0$
(For the purposes of the applications in this paper, one may take $\F=\Q$).
We endow $\F[[X]]$ with the valuation $|R|=2^{-n}$, where $n$ is the index of the first non-zero coefficient of $R\in \F[[X]]$.
Then $|R_1-R_2|$ defines a metric on $\F[[X]]$, but the resulting topology is not locally compact.

Given $R\in \F[[X]]$, we define the $\F[[X]]$-valued random element
\[
A_R=\sum_{j=0}^{\infty} T_{\xi_j}(1,R(X))X^j,
\]
where $\xi_j$ are independent random variables taking the values $1,2,3$ with equal probabilities.
It is easy to see that the right hand side converges for every realisation of $\{\xi_j\}_{j\ge0}$.
One may consider the distribution of $A_R$ a function field analogue of a self-similar measure.

We turn to the above mentioned notion of dimension for measures on $\F[[X]]$.
Let $A=\a_0+\a_1X+\a_2X^2+\ldots$ be an $\F[[X]]$-valued random element, so $\a_0,\a_1,\ldots$
are $\F$-valued random variables.
We write
\[
H(A;n)=H\Big(\sum_{j=0}^{n-1}\a_j X^j\Big),
\]
where $H(\cdot)$ denotes Shannon entropy on the right hand side.
Now we define
\[
\dim A = \lim_{n\to \infty} \frac{H(A;n)}{n}
\]
provided the limit exists.
If $\mu$ is the law of $A$, we write $\dim \mu = \dim A$.
One may consider this notion an analogue of entropy dimension on $\R$.

It can be shown that $H(A_R;n)/n$ monotone increases to $\dim A_R$.
Since $H(A_R;n)$ is continuous in $R$, $\dim A_R$ is lower semi-continuous.
 
A major ingredient in our proof of Proposition \ref{pr:no-bad-curves-3maps} is the following result
which we prove by adapting Hochman's full machinery from \cite{Hoc-self-similar} to our function field setting.

\begin{prp}\label{pr:dim=h-3maps}
Let $\gamma\in\Gamma$ be a non-degenerate curve corresponding to a function $R\in\cR\cap\Q[[X]]$.
Then
\[
\dim A_R = h(\gamma).
\]
\end{prp}

Now we return to our proof strategy of Proposition \ref{pr:no-bad-curves-3maps}.
Let $\gamma_n$ be a sequence of curves that pass in smaller and smaller neighbourhoods of a point $(\l,\tau)$
and let $h(\gamma_n)\le h$ for some $h$ and for all $n$.
The most significant case is when the curves are non-degenerate, so we only consider this now.
Let $R_n$ be the function in $\cR$ that defines $\gamma_n$.
For simplicity, we assume $R_n\in\Q[[X]]$.

Recall that functions in $\cR$ are ratios of two power-series
with coefficients $\pm1, 0$.
It can be seen that the $j$'th coefficient of a function in $\cR\cap\Q[[X]]$ is an integer of absolute value at most
$2^{j}$.
This means that there are only finitely many possibilities for each coefficients,
and so $\cR\cap\Q[[X]]$ is compact in $\Q[[X]]$.
This allows us to pass to a convergent subsequence of $R_n$, which we do without changing notation.
We denote the limit function by $R$ and the corresponding non-degenerate curve by $\gamma$.

The last hurdle in our proof of Proposition \ref{pr:no-bad-curves-3maps} is that functions in
$\cR\cap\Q[[X]]$ converging coefficientwise (i.e. in the $|\cdot|$ metric) may fail to converge uniformly in any neighbourhood of $\l$.
That is to say, it may happen that $(\l,\tau)\notin\gamma$.
We overcome this problem by showing that in such a situation, there is a non-trivial interval $[a,b]\subset\R$ such that
$\dim\mu_{\l,\tau'}\le h/\log \l^{-1}$ for all $\tau'\in[a,b]$.
Now using the case of Conjecture \ref{cn:EO} for the IFS \eqref{eq:3maps} when $\tau$ is rational, we can prove that
$\l$ is algebraic.
This case of the conjecture can be proved using the methods of \cite{Var-Bernoulli} without any significant changes.
We carry this out in Appendix \ref{sec:int trans case}.
Then using the results of Hochman \cite{Hoc-self-similar} for algebraic parameters, we can conclude
$h(\{\l\}\times \R)\le h$ provided $h\le\log \l^{-1}$.
Therefore, the degenerate curve $\{\l\}\times \R$ can be used in place of $\gamma$.

\subsection{Ideas from the proofs of Theorems \ref{th:EO} and \ref{th:large-lambda}}\label{sc:Ideas2}

Conjecture \ref{cn:EO} is proved in \cite{Var-Bernoulli} for Bernoulli convolutions for transcendental parameters.
One of the key ingredients of the proof is the transversality property of polynomials with coefficients
$\pm1, 0$ established by Solomyak \cite{Sol-Bernoulli}, see also \cite{PS-Bernoulli} and \cite{PS-transversality}.
Loosely speaking, this property is a quantitative refinement of the fact that such polynomials cannot have
a double root in the interval $[1/2,2/3]$.

It was pointed out to us by Dimitrov \cite{Dim-personal} that a weaker version of transversality,
first used by Hochman \cite{Hoc-self-similar}
in the context of self-similar measures, would suffice for the purposes of the argument in \cite{Var-Bernoulli}.
This weaker property stipulates that there is a number $K$ for each $\e>0$ such that no polynomial with coefficients
$\pm1,0$ can have a root of multiplicity higher than $K$ in $(0,1-\e)$.
We give more details below in Lemma \ref{lm:Dimitrov}.

In the contexts of Theorems \ref{th:EO} and \ref{th:large-lambda}, we would require this multiplicity bound for polynomials of the form
\begin{equation}\label{eq:P1234}
P_1P_2-P_3P_4,
\end{equation}
where
$P_1,P_2,P_3,P_4$ are all polynomials with coefficients $\pm1,0$.
Unfortunately, it may happen that the first $N$ coefficients of \eqref{eq:P1234} vanish for some large $N$ and then we cannot
bound the absolute values of the first few non-zero coefficients of \eqref{eq:P1234} independently of $N$.
This feature rules out the usual proofs for bounding the multiplicities of roots of \eqref{eq:P1234}.
In fact, we are unable to rule out the possibility that polynomials of the form \eqref{eq:P1234} may have roots of arbitrarily high multiplicity
in $(\e,1-\e)$.
However, we are able to show that this does not happen often, and this is the main novelty in the proofs of Theorems \ref{th:EO} and \ref{th:large-lambda}.

We outline the proofs of  Theorems \ref{th:EO} and \ref{th:large-lambda}.
The overall strategy follows \cite{Var-Bernoulli}; however, there are significant new difficulties as mentioned above,
which we explain further.
We suppose to the contrary that there is $(\l_0,\tau_0)\in(0,1)\times\R$ that violates Conjecture \ref{cn:EO}.
That is, the IFS \eqref{eq:3maps} contains no exact overlaps for these parameters, yet
$\dim\mu_{\l_0,\tau_0}<\min\{1,\log 3/\log \l^{-1}_{0}\}$, where $\mu_{\l_0,\tau_0}$ is the corresponding self-similar measure. Let $\e>0$ be small with respect to $(\l_{0},\tau_{0})$, and let $M>1$ be large with respect to $\e>0$.
We apply Theorem \ref{th:BV2}, and conclude that there exist arbitrarily large $n$ and $(\eta,\sigma)\in(0,1)\times\R$
such that
\begin{align*}
|\eta-\l_0|,|\tau_0-\s|\le&\exp(-n^{100}),\\ 
h(\eta,\s)\le&\min\{\log\eta^{-1},\log 3\}-\e,\\
h(\gamma)\ge&\min\{\log\eta^{-1},\log 3\}-M^{-1}
\end{align*}
for all $\g\in\Gamma$ with $(\eta,\s)\in\gamma$.
An affirmative answer to Question \ref{q:BV1} now yields $M(\eta)\le M$.

We also apply the results of Hochman \cite{Hoc-self-similar}.
Since $(\l_0,\tau_0)$ violates Conjecture \ref{cn:EO}, we can conclude that for all $C_0>1$
and for all sufficiently large $n'$
(depending on $C_0$, $\l_0$ and $\tau_0$), we have
\[
H(A_{\l_0,\tau_0}^{(n')};C_0^{-n'}) \le (\min\{\log\l_0^{-1},\log 3\}-\e)n'.
\]
We will apply this theorem with a carefully chosen value of $n'$ so that $|\l_0-\eta|$
is exponentially small in $n'$.
This implies, as we discussed in Section \ref{sc:ideas-th3}, that there is a large family of pairs of points
in the support of $A_{\l_0,\tau_0}^{(n')}$ that are at distance less than $C_0^{-n'}$.
For each such pair of points, there corresponds a polynomial $0\neq Q\in\cQ^{(n')}$
such that $|Q(\l_0,1,\tau_0)|\le C_0^{-n'}$.
We write $\wt\cQ^{(n')}$ for the collection of polynomials that satisfy this property.

We pick an arbitrary $\wt Q\in\wt\cQ^{(n')}$.
We can write it in the form
\[
\wt Q(X,Y_1,Y_2)=\wt P_1(X)Y_1+\wt P_2(X)Y_2
\]
for some polynomials $\wt P_1$,
$\wt P_2$ of degree at most $n'$ with coefficients $\pm1,0$.
We write $\wt R=-\wt P_1/\wt P_2$ and write $\wt\gamma$ for the associated non-degenerate curve.
It can be proved that $\wt\gamma$ passes in an arbitrarily small neighbourhood of $(\l_0,\tau_0)$ if $n$ is sufficiently large
and the value of $n'$ is set appropriately.
If we could also show that $Q(X,1,\wt R(X))=0$ for all $Q\in\wt\cQ^{(n')}$, then we could conclude
\[
h(\wt\gamma)\le \frac{H(A_{\wt\gamma}^{(n')})}{n'}\le \frac{H(A_{\l_0,\tau_0}^{(n')};C_0^{-n'})}{n'}
\le \min\{\log\l_0^{-1},\log 3\}-\e
\]
using the definition of $\wt\cQ^{(n')}$.
This would contradict Proposition \ref{pr:no-bad-curves-3maps} completing the proof of Theorem \ref{th:EO}.

Let $Q\in\wt\cQ^{(n')}$ and write $Q(X,Y_1,Y_2)=P_1(X)Y_1+P_2(X)Y_2$. 
The equation $Q(X,1,\wt R(X))=0$ can now be rewritten as $P_1(X)\wt P_2(X)-P_2(X)\wt P_1(X)=0$.
Using the definition of $\wt\cQ^{(n')}$, we can write
\[
|P_1(\l_0)+P_2(\l_0)\tau_0|,|\wt P_1(\l_0)+\wt P_2(\l_0)\tau_0|\le C_0^{-n'}
\]
from which we conclude
\[
|P_1(\l_0)\wt P_2(\l_0)-\wt P_1(\l_0)P_2(\l_0)| = O(C_0^{-n'}),
\]
where the implied constant depends only on $\l_0$.

Now we invoke the following lemma that was suggested to us by Dimitrov \cite{Dim-personal}.
(We give a more precise version in Lemma \ref{lm:Dimitrov-precise}.)

\begin{lem}\label{lm:Dimitrov}
For every $k\in\Z_{\ge2}$, there is a constant $C$ such that the following holds.
Let $\l\neq\eta\in[0,1]$.
Assume that $\eta$ is algebraic of degree at most n and let $n'\ge Cn\log (n+1)$ be an integer.
Let $P\in\Z[X]$ be a polynomial of degree at most $n'$ with coefficients bounded by $n^k$ in absolute value.
Assume
\[
(2M(\eta))^{n'/k}|P(\l)|^{1/k}\le |\l-\eta|\le (2M(\eta))^{-n'}.
\]
Then $\eta$ is a zero of $P$ of order at least $k$.
\end{lem}

We apply this for the polynomial $P=P_1\wt P_2-\wt P_1 P_2$ with $\l=\l_0$.
If we choose the value of $n'$ appropriately and $C_0$ is large enough, we can conclude
that $\eta$ is a root of $P_1\wt P_2-\wt P_1 P_2$ of multiplicity $k$ for an arbitrarily large
but fixed (independently of $n'$) $k$.
If we knew that this always implies $P_1\wt P_2-\wt P_1 P_2=0$, we could conclude
$Q(X,1,\wt R(X))=0$ for all $Q\in\cQ^{(n')}$, and we could complete the proof of Theorem \ref{th:EO}
as we discussed above.

However, as we noted above, we are not able to rule out the possibility that a non-trivial polynomial
of the above form could have roots of arbitrarily high multiplicity.
To overcome this problem, we introduce another fractal object.

Let $\xi_{0},\xi_{1},\ldots$ be a sequence of random variables taking the values $\{1,2,3\}$
with equal probabilities.
Recall the definition of $T_1,T_2,T_3$ from Section \ref{sc:intro}.
Given $R\in\cR\cap\Q[[X]]$ and $n\ge 0$, we set
\[
A_{R}^{(n)}=\sum_{j=0}^{n-1}T_{\xi_j}(1,R(X)) X^j
\]
as before, but now we think about it as a random meromorphic function on the complex unit disk.
Furthermore, given $\l\in(0,1)$ that is not a pole of $R$ and $K\in\Z_{\ge 1}$, we define the $\R^K$-valued
random vector
\[
B_{R,\l,K}^{(n)}=(A_{R}^{(n)}(\l),\frac{d}{dX} A_{R}^{(n)}(\l),\ldots,\frac{d^{K-1}}{dX^{K-1}} A_{R}^{(n)}(\l)).
\]
We note that these random vectors can also be obtained by a random walk of $n$ steps using
a self-affine IFS.
(See Section \ref{sc:lb on h(R,lam,M)} for details.)
Therefore, the entropy rate
\[
h(R,\l,K)=\lim_{n\to\infty}\frac{H(B_{R,\l,K}^{(n)})}{n}=\inf_{n\ge1} \frac{H(B_{R,\l,K}^{(n)})}{n}
\]
exists by subadditivity.

In this notation, the above discussion up to the application of Lemma \ref{lm:Dimitrov} leads to the bound
\[
h(\wt R,\eta,K)\le\frac{H(B_{\wt R,\eta,K}^{(n')})}{n'}\le  \min\{\log\l_0^{-1},\log 3\}-\e.
\]
Here $K$ can be taken arbitrarily large, but have to be fixed at the beginning of the argument,
in particular, it cannot depend on $n'$.

To complete the argument, we need the following result related to Proposition \ref{pr:no-bad-curves-3maps},
which is the second main new contribution of the paper.

\begin{prp}\label{pr:no-bad-bouquet-3maps}
Let $(\l_0,\tau_0)\in(0,1)\times\R$, such that the IFS \eqref{eq:3maps} contains no exact overlaps.
Then for every $h_0<\min\{\log\l_0^{-1},\log 3\}$, there is $K\in\Z_{>0}$ such that
$h(R,\l,K)\ge h_0$ for all $R\in\cR\cap\Q[[X]]$ and $\l\in(0,1)$ such that $|R(\l)-\tau_0|,|\l-\l_0|\le K^{-1}$.
\end{prp}

This result is based on the following intuitive idea.
Let $R_l\in\cR\cap\Q[[X]]$, $(\l_l,\tau_l)\in(0,1)\times\R$ with $\tau_l=R_{l}(\l_{l})$, and $K_l\in\Z_{>0}$ for $l=1,2,\ldots$.
Suppose to the contrary that $h(R_l,\l_l,K_l)\le h_0$ for all $l$, $(\l_l,\tau_l)\to (\l_0,\tau_0)$ and $K_l\to \infty$.
For simplicity, we assume further that $R_l$ converges to some $R_0\in\cR$ in the $|\cdot|$ metric and $R_0(\l_0)=\tau_0$.
This is not always possible to arrange, and we comment further on this below. 
The random variables $B_{R_l,\l_l,K_l}$ converge in distribution
to a self-affine measure $\mu_{R_l,\l_l,K_l}$, whose dimension is bounded above by $h(R_l,\l_l,K_l)/\log \l_l^{-1}$.
Intuitively, it is reasonable to expect that
$\mu_{R_l,\l_l,K_l}$ converges to $\mu_{R_0}$ in some sense, where $\mu_{R_0}$ is the distribution of $A_{R_0}$.
Then we may expect that dimension is lower semi-continuous with respect to this convergence in a way that takes into
account the factor $1/\log\l_l^{-1}$.
This would yield
\[
\liminf h(R_l,\l_l,K_l)\ge \dim\mu_{R_0},
\]
and we could conclude by Proposition \ref{pr:dim=h-3maps}.
Indeed, writing $\gamma_0$ for the non-degenerate curve associated to $R_0$, we have
\[
\log 3=h(\l_0,\tau_0)\le h(\gamma_0)=\dim\mu_{R_0},
\]
since the IFS contains no exact overlaps.

It is not obvious how to define the convergence of the self-affine measures to $\mu_{R_0}$
in a rigorous way, or indeed if it can be done at all.
Nevertheless, this intuition motivates our proof.

Finally, we comment on the case when it is not possible to arrange that the $R_l$'s have a limit function
satisfying $R_0(\l_0)=\tau_0$.
We have to deal with a similar issue in the course of the proof of Theorem \ref{th:BV2}, as well, but
in the present situation, this is significantly more difficult to overcome.
We use an estimate from Tur\'an's theory of power sums \cite{Tur-PWsums} to show that a derivative of
$R_l$ of a certain order must blow up at $\l_l$, which allows us to rescale $\mu_{R_l,\l_l,K_l}$ so that it converges
to a self-similar measure.
We show that we can bound the dimension of this self-similar measure in terms of $h_0$ and that
$\mu_{\l_0,\tau}$ for all $\tau\in\R$ can be realized as a projection of it.
In the end, we are able to deduce that $\dim \mu_{\l_0,\tau}\le h_0/\log\l_0^{-1}$ for all $\tau\in\R$ and we can
conclude as in the proof of Theorem \ref{th:BV2}.

\subsection{The organization of the paper}

We introduce the general setting of the paper and reformulate our main results in Section \ref{sc:general-setting}.
We collect some results from Diophantine approximation in Section \ref{sc:aux res num}, which we will use at various
places in the paper.
These results are standard (with the exception of Lemma \ref{lm:Dimitrov-precise}, which we mentioned above in a simpler form
in Lemma \ref{lm:Dimitrov}).

In Sections \ref{sc:ssm-Cx} and \ref{sc:repel}, we develop the ideas discussed above that are needed for the proof of Theorem \ref{th:BV2}.
Section \ref{sc:ssm-Cx} is devoted to the study of an analogue of self-similar measures in a function field setting.
We adapt the methods of Hochman \cite{Hoc-self-similar} to this setting and prove (a more general version of) Proposition \ref{pr:dim=h-3maps}.
In Section \ref{sc:repel}, we prove (a more general version of) Proposition \ref{pr:no-bad-curves-3maps},
i.e. low entropy curves avoid small neighborhoods of parameter points without exact overlaps.
In Section \ref{sc:prf of apx thm}, we give the proofs of Theorem \ref{th:BV2-general} and Corollary \ref{cr:Mike-general}, which 
are restatements of Theorem \ref{th:BV2} and Corallary \ref{cr:Mike} in our general setting.

In Sections \ref{sc:lb on h(R,lam,M)} and \ref{lb ent ord M}, we develop the ideas discussed above that are needed for the proof of Theorem \ref{th:EO}.
In Section \ref{sc:lb on h(R,lam,M)}, we introduce the self-affine measures that we discussed briefly above, and
we start studying their entropy rates.
This will be completed in Section \ref{lb ent ord M}, where we concentrate on the singularities of the curve in the definition of
the self-affine measure and prove Proposition \ref{pr:no-bad-bouquet-3maps} in a more general form.
The proof of Theorem \ref{th:BV1} is given in Section \ref{sc:BV1}.
In Section \ref{sc:EO-proof}, we complete the proofs of Theorem \ref{th:EO-general} (a general version of Theorem \ref{th:EO})
and Theorem \ref{th:large-lambda}.

We prove a case of Conjecture \ref{cn:EO} for homogeneous IFS's with rational translations in Appendix \ref{sec:int trans case}.
This result is a mild generalization of the case of Bernoulli convolutions treated in \cite{Var-Bernoulli} and its references, and does not
require substantial new ideas.
The result is used as a black box, and the paper can be read without the appendix.
However, the appendix recalls many of the crucial elements of the arguments in the papers \cite{BV-transcendent}
and \cite{Var-Bernoulli}, which will be used in Sections \ref{sc:prf of apx thm} and \ref{sc:EO-proof}, so the reader may
find it beneficial to read (parts of) the appendix for an introduction to these arguments in a simpler setting.

\section{The general setup}\label{sc:general-setting}

In this section, we set out the general setup of the paper.
We fix some notation that we use throughout.
We will also restate the results of the paper in this general setting.

We fix integers $m\ge 2$, $a_1,\ldots,a_m$ and $b_1,\ldots,b_m$ with $(a_i,b_i)\neq(a_j,b_j)$
for all $1\le i<j\le m$.
For each $j=1,\ldots,m$, we introduce the linear form
$T_j(Y_1,Y_2)=a_jY_1+b_jY_2$.
Note that the case of three maps considered in the previous sections corresponds to the tuples $(0,0)$, $(1,0)$ and $(0,1)$.
For parameters $\l\in(0,1)$, $\tau\in\R$ we associate the IFS
\begin{equation}\label{eq:IFS}
\{f_j(x)=\l x+T_j(1,\tau):j=1,\ldots,m\}.
\end{equation}
We also fix a probability vector $p=(p_1,\ldots,p_m)$ with strictly positive coordinates.
We denote by $\mu_{\l,\tau}$ the self-similar measure associated to this IFS and probability vector.

Our motivation to consider this more general framework instead of just the IFS \eqref{eq:3maps} is twofold.
First, it contains other natural examples, such as the IFS
\[
x\mapsto \l x-1-\tau,\quad x\mapsto \l x+1-\tau, \quad x\mapsto \l x-1+\tau,\quad x\mapsto \l x+1+\tau.
\]
The self-similar measures associated to this IFS include the convolution of a Bernoulli convolution with a scaled copy,
and this complements the measures studied by Nazarov, Peres and Shmerkin \cite{NPS-convolutions}.
Second, this more general framework may be useful in a way similar to how we use the case of
rational translations
(to be discussed in Appendix \ref{sec:int trans case}) in this paper.

We write $\cP_{l}^{(n)}\subset\Z[X]$ for the set of polynomials of degree strictly less than $n$ with
integer coefficients bounded in absolute value by $l$.
We write $\cP_l\subset\Z[[X]]$ for the set of formal power series with integer coefficients bounded in absolute value by $l$.
We write $\cR_l$ for the set of ratios of two functions in $\cP_l$, we consider these functions both as formal Laurent series
with rational coefficients and meromorphic functions on the complex unit disk.
We note that we wrote $\cR$ for $\cR_1$ in the first two sections of the paper.

We write $\Omega=\{1,\ldots,m\}^{\Z_{\ge 0}}$.
We equip $\Omega$ with the Bernoulli measure $\b$ corresponding to the probability vector $p$, that is $\beta=p^{\Z_{\ge 0}}$.

We write $\cQ^{(n)}$ for the set of polynomials of the form
\[
Q(X,Y_1,Y_2)=\sum_{j=0}^{n-1} (T_{\o_j}(Y_1,Y_2)-T_{\o'_{j}}(Y_1,Y_2))X^j
\]
for $\o,\o'\in\Omega$.
It is immediate from the definitions that each $Q\in\cQ^{(n)}$ can be written in the form
\[
Q(X,Y_1,Y_2)=P_1(X)Y_1+P_2(X)Y_2
\]
for some $P_1,P_2\in\cP_L^{(n)}$, where
\begin{equation}\label{eq:def of L}
L:=\max\{a_i-a_j,b_i-b_j:i,j=1,\ldots,m\}.
\end{equation}
We observe that
\[
f_{\o_0}\circ\ldots\circ f_{\o_{n-1}}=f_{\o'_0}\circ\ldots\circ f_{\o'_{n-1}}
\]
for some $\o,\o'\in\Omega$, if and only if
\[
\sum_{j=0}^{n-1} (T_{\o_j}(1,\tau)-T_{\o'_{j}}(1,\tau))\l^j=0.
\]
Therefore, the IFS \eqref{eq:IFS} contains exact overlaps for some parameters $\l,\tau$
if and only if $Q(\l,1,\tau)=0$ for some $0\neq Q\in\cQ^{(n)}$.

For a set $U\subset(0,1)\times\R$ and $n\in\Z_{\ge 0}$, we consider the random function
$A_{U}^{(n)}:U\to\R$ defined by
\[
A_{U}^{(n)}(\l,\tau)=\sum_{k=0}^{n-1}T_{\xi_k}(1,\tau)\l^k,
\]
where $\xi_0,\xi_1,\ldots$ are i.i.d. random variables with $\P\{\xi_0=j\}=p_j$ for $1\le j\le m$.
The entropy rate is defined as
\[
h(U)=\lim_{n\to\infty}\frac{1}{n}H(A_{U}^{(n)})=\inf_{n\ge1}\frac{1}{n}H(A_{U}^{(n)}),
\]
where $H(\cdot)$ denotes Shannon entropy.
We note that the sequence $n\mapsto H(A_{U}^{(n)})$ is subadditive, hence the above limit exists and it equals the
infimum.

The case when $U=\{(\l,\tau)\}$ is a single point is of special interest, and we will write $A_{\l,\tau}^{(n)}$
and $h(\l,\tau)$ to simplify notation. We denote by $\mu_{\l,\tau}^{(n)}$ the distribution of $A_{\l,\tau}^{(n)}$.
It is immediate from the definition that
\[
h(\l,\tau)\le H(p):=\sum_{j=1}^m p_j\log p_j^{-1},
\]
and equality holds if and only if the IFS \eqref{eq:IFS} does not contain exact overlaps for $\l$ and $\tau$.

We denote by $\Gamma$ the set of curves $\gamma\subset (0,1)\times\R$
that are of one of the following two forms
\begin{enumerate}
\item $\gamma=\{(\l,\tau)\in(0,1)\times\R:\tau=R(\l)\}$ for some $R\in\cR_L$,
\item $\gamma=\{(\l_0,\tau):\tau\in\R\}$ for some fixed $\l_{0}\in(0,1)$. 
\end{enumerate} 
If $\gamma$ is of the first form, we call it non-degenerate, otherwise it is degenerate.

Now we restate our results from Section \ref{sc:intro} in our general framework.

\begin{thm} \label{th:BV2-general}
Let $(\lambda,\tau)\in(0,1)\times\R$ and let $\mu_{\l,\tau}$ be the self-similar measure associated
to the IFS \eqref{eq:IFS} and the probability vector $p$.
Suppose the IFS does not contain exact overlaps and
\[
\dim\mu_{\l,\tau}<\min\Big\{\frac{H(p)}{\log \l^{-1}},1\Big\}.
\]
Then for every $\e>0$ and $N\ge1$, there
exist $n\ge N$ and $(\eta,\sigma)\in (0,1)\times\R$ such
that,
\begin{enumerate}
\item $|\l-\eta|,|\tau-\s|\le \exp(-n^{1/\e})$;
\item $\frac{1}{n\log\eta^{-1}} H(A_{\eta,\s}^{(n)})\le \dim\mu_{\l,\tau}+\e$;
\item $h(\gamma)\ge \min\{\log\l^{-1},H(p)\}-\e$ for all $\g\in\Gamma$ with $(\eta,\s)\in\gamma$.
\end{enumerate}
In particular, $\eta$ is a root of a nonzero polynomial in $\cP_{2L^2n}^{(2n)}$ and $\s$ can be written in the form
$\sigma=P_{1}(\eta)/P_{2}(\eta)$ for some $P_{1},P_{2}\in\cP_{L}^{(n)}$ such that $P_{2}(\eta)\ne0$.
\end{thm}

\begin{cor}\label{cr:Mike-general}
The set of parameters $(\l,\tau)\in(0,1)\times\R$ for which Conjecture \ref{cn:EO} fails for the IFS \eqref{eq:IFS} and the probability vector $p$
is of Hausdorff dimension $0$.
\end{cor}

\begin{que}\label{q:BV1-general}
Is it true that for all $\e>0$, there is $M$ such that the following holds?
Let $(\l,\tau)\in(\e,1-\e)\times\R$ be such that $h(\l,\tau)\le \min\{H(p),\log \l^{-1}\}-\e$ and $h(\g)\ge \min\{H(p),\log \l^{-1}\}-M^{-1}$
for all $\g\in\Gamma$ with $(\l,\tau)\in\gamma$.
Then $M(\l)\le M$.
\end{que}

\begin{thm}\label{th:EO-general}
Suppose that the answer to Question \ref{q:BV1-general} is affirmative.
Then Conjecture \ref{cn:EO} holds for the IFS \eqref{eq:IFS} for any
$(\l,\tau)\in(0,1)\times\R$.
\end{thm}

We conclude this section with a summary of our main notation:

\begin{longtable}[c]{|c|>{\centering}p{7.1cm}|}
\hline
$T_j(Y_1,Y_2)=a_jY_1+b_jY_2$ & Linear forms, Sec. \ref{sc:general-setting}\tabularnewline
\hline
$\{f_j(x)=\l x+T_j(1,\tau)\}_{j=1}^m$ & IFS, Sec. \ref{sc:general-setting}\tabularnewline
\hline
$p=(p_1,\ldots,p_m)$ & Probability vector, Sec. \ref{sc:general-setting}\tabularnewline
\hline
$\mu_{\l,\tau}$, $\mu_{z,w}$ & Self-similar measures, Sec. \ref{sc:general-setting}, \ref{sc:repel}\tabularnewline
\hline
$\cP_{l}^{(n)}$,$\cQ^{(n)}$ & Sets of polynomials, Sec. \ref{sc:general-setting}\tabularnewline
\hline
$\cP_l$ & Set of power series, Sec. \ref{sc:general-setting}\tabularnewline
\hline
$\cR_l$ & Set of $P_1/P_2$ with $P_1,P_2\in\cP_l$, Sec. \ref{sc:general-setting}\tabularnewline
\hline
$\Omega:=\{1,\ldots,m\}^{\Z_{\ge 0}}$ & Symbolic space, Sec. \ref{sc:general-setting}\tabularnewline
\hline
$\beta:=p^{\Z_{\ge 0}}$ & Bernoulli measure, Sec. \ref{sc:general-setting}\tabularnewline
\hline
$L:=\max\{a_i-a_j,b_i-b_j\}_{i,j=1}^m$ & Global constant, Sec. \ref{sc:general-setting}\tabularnewline
\hline
$A_{U}^{(n)}$, $A_{\l,\tau}^{(n)}$ & Random elements, Sec. \ref{sc:general-setting}\tabularnewline
\hline
$\mu_{\l,\tau}^{(n)}$ & Distribution of $A_{\l,\tau}^{(n)}$, Sec. \ref{sc:general-setting}\tabularnewline
\hline
$h(U)$, $h(\l,\tau)$, $h(R)$, $h(R,\l,K)$ & Entropy rates, Sec. \ref{sc:general-setting}, \ref{sc:ssm-Cx}, \ref{sc:lb on h(R,lam,M)} \tabularnewline
\hline
$\Gamma$ & Set of curves, Sec. \ref{sc:general-setting}\tabularnewline
\hline
$M(P)$, $M(\eta)$ & Mahler measure, Sec. \ref{sc:aux res num}\tabularnewline
\hline
$\F((X))$ & Field of formal Laurent series, Sec. \ref{sc:ssm-Cx}\tabularnewline
\hline
$\F[[X]]$ & Ring of formal power series, Sec. \ref{sc:ssm-Cx}\tabularnewline
\hline
$|R|$ & Absolute value of $R\in\F((X))$, Sec. \ref{sc:ssm-Cx}\tabularnewline
\hline
$A_R$, $A_{R}^{(n)}$ & Random power series, Sec. \ref{sc:ssm-Cx}\tabularnewline
\hline
$\mu_R$ & Distribution of $A_R$, Sec. \ref{sc:ssm-Cx}\tabularnewline
\hline
$H(A;n)$, $H(A;l|n)$ & Entropies of random $A\in \F[[X]]$, Sec. \ref{sc:ssm-Cx}\tabularnewline
\hline
$\dim A_R, \dim\mu_R$ & Limit of $\frac{1}{n}H(A_R;n)$, Sec. \ref{sc:ssm-Cx}\tabularnewline
\hline
$s(\l):=\min\{1,H(p)/\log\lambda^{-1}\}$ & Natural upper bound, Sec. \ref{sc:prf of apx thm}\tabularnewline
\hline
$H(A;r)$, $H(\nu;r)$, $H(\nu;r_1|r_2)$ & Entropies at and between scales, Sec. \ref{sc:prf of apx thm}\tabularnewline
\hline
\end{longtable}

\section{Auxiliary results from number theory}\label{sc:aux res num}

In this section, we gather some results from Diophantine approximation, which will be needed later on.
We are mostly concerned by lower bounds on the distance between two algebraic numbers
and on the values of polynomials with small integer coefficients.
The importance of such estimates has been highlighted in Section \ref{sc:outline}.
All of the results in this section are well known and has been used in the dimension theory of fractal measures before,
with the exception of Lemma \ref{lm:Dimitrov-precise}, which we learned from Dimitrov \cite{Dim-personal}.

The results in this section will be used first in Section \ref{sc:prf of apx thm}, and the reader may safely skip them and
refer back later when needed.

Given a polynomial,
\[
P(x)=\sum_{k=0}^{n}a_kx^k=a_n(x-\l_1)\ldots(x-\l_n)\in\C[x],
\]
its Mahler measure is denoted $M(P)$, and defined by
\[
M(P)=|a_n|\prod_{j:|\l_j|>1}|\l_j|.
\]
The length of $P$ is denoted $\ell_1(P)$, and defined by
\[
\ell_1(P)=|a_0|+\ldots+|a_n|.
\]
It is immediate from the definition that Mahler measures of polynomials are multiplicative.
It is a standard fact that $M(P)\le \ell_1(P)$ for all $P\in\C[x]$, see \cite{BG-heights}*{Lemma 1.6.7}.

For an algebraic $\eta\in\C$, we write $\deg\eta$ for its degree
over $\Q$, and $M(\eta)$ for its Mahler measure, which is defined to be the Mahler measure of its minimal polynomial over $\Z$.

The following lemma is due to Mahler \cite{Mah-discriminant}.
We will use it to bound the distance between parameters with exact overlaps, and it is a key ingredient
in the proof of Theorem \ref{th:BV2-general}, as we discussed in Section \ref{sc:ideas-th3}.

\begin{lem}
\label{lem:dist between distinct roots}Let $n,l\ge1$ and let $\eta,\eta'\in\C$
be two distinct algebraic numbers, each of which is a root of a nonzero
polynomial in $\mathcal{P}_{l}^{(n)}$. Then,
\[
|\eta-\eta'|\ge2^{-n-1}n^{-5n}l^{-4n}.
\]
\end{lem}

\begin{proof}
Let $P\in\Z[X]$ be of degree at most $d\ge1$ and with distinct
roots. By \cite{Mah-discriminant}*{Theorem 2} it follows that the distance between
any two roots of $P$ is at least,
\[
\sqrt{3}d^{-(d+2)/2}M(P)^{-(d-1)}.
\]

If $\eta$ and $\eta'$ are Galois conjugates, take $P$ to be their
minimal polynomial over $\Z$. If they are not Galois conjugates,
take $P$ to be the product of their minimal polynomials. In either
case, $P$ has distinct roots, the degree of $P$ is at most $2n$,
and its Mahler measure is at most the product of the Mahler measures
of two polynomials in $\mathcal{P}_{l}^{(n)}$.

We have $M(E)\le\ell_{1}(E)\le nl$
for each $E\in\mathcal{P}_{l}^{(n)}$, which gives $M(P)\le n^{2}l^{2}$.
Thus,
\[
|\eta-\eta'|\ge\sqrt{3}(2n)^{-n-1}M(P)^{-2n+1}\ge2^{-n-1}n^{-5n}l^{-4n},
\]
which completes the proof of the lemma.
\end{proof}

The following lemma is a standard application of Jensen's formula to bound the number of roots a polynomial in $\mathcal{P}_{l}^{(n)}$ may have away from the unit circle. This will be used in the next lemma to show that a polynomial taking a small value at a point must have a nearby root.

\begin{lem}
\label{lem:bound on num of roots}There is a function $a:\Z_{>0}\to(0,1)$
such that $\lim_{k\rightarrow\infty} a(k)=1$ and the following
holds. Let $l,n\ge1$ and $0\ne P\in\mathcal{P}_{l}^{(n)}$ be given.
Then there are at most $k(1+\frac{\log l}{\log(k+1)})$ nonzero roots
of $P$ of absolute value less than $a(k)$.
\end{lem}

\begin{proof}
For $k\ge1$ set $a(k)=\frac{k}{k+1}\cdot\frac{1}{(k+1)^{1/k}}$.
Let $l,n\ge1$ and $0\ne P\in\mathcal{P}_{l}^{(n)}$ be given. Without
loss of generality we may assume that $|P(0)|\ge1$, otherwise we
can divide $P$ by an appropriate power of $X$ and obtain a new polynomial
with this property. Let $k\ge1$ and let $z_{1},\ldots,z_{K}$ be the
roots of $P$ of absolute value less than $a(k)$. Write $r=k/(k+1)$,
then by Jensen's formula and since $\log|P(0)|\ge0$,
\[
\sum_{j=1}^{K}\log\frac{r}{|z_{j}|}\le\int_{0}^{1}\log|P(re^{2\pi it})|dt.
\]
For each $t\in[0,1]$,
\[
|P(re^{2\pi it})|\le l(1+r+r^{2}+\ldots)=\frac{l}{1-r}=l(k+1).
\]
Since $|z_{j}|\le a(k)$ for each $1\le j\le K$,
\[
\frac{K}{k}\cdot\log(k+1)=K\cdot\log\frac{r}{a(k)}\le\sum_{j=1}^{K}\log\frac{r}{|z_{j}|}.
\]
From all of this, we get
\[
\frac{K}{k}\cdot\log(k+1)\le\log(l(k+1)),
\]
which completes the proof of the lemma.
\end{proof}

The next lemma will be used to find a nearby parameter with exact overlaps when the points in the
$n$'th generation approximation of a self-similar measure are not well separated.

\begin{lem}
\label{lem:close root single poly}For every $\e\in(0,1)$ there exists
$c=c(\e)\in(0,1)$ such that the following holds. Let $n\ge1$,
$l\ge3$, $0\ne P\in\mathcal{P}_{l}^{(n)}$, $0<r<\e^n2^{-n}$ and $\lambda\in\C$
be given. Suppose that $\e\le|\lambda|\le1-\e$ and $|P(\lambda)|\le r$.
Then there exists $\eta\in\C$ such that $P(\eta)=0$ and,
\[
|\lambda-\eta|\le\left(2^{n}\e^{-n}r\right)^{c/\log l}.
\]
\end{lem}

\begin{proof}
Let $\e\in(0,1)$, let $C\ge1$ be large with respect to $\e$,
and let $n,l,P,r$ and $\l$ be as in the statement of the
lemma. By Lemma \ref{lem:bound on num of roots},
$P$ has at most $C\log l$ nonzero roots of modulus at most $1-\e/2$ provided $C$ is sufficiently large depending on $\e$,
which we assumed.
Denote all of these roots (with multiplicity) by $\eta_{1},\ldots,\eta_{m}$.
Then $m\le C\log l$ and,
\[
r\ge|P(\lambda)|\ge(\e/2)^{n-m}\cdot\prod_{j=1}^{m}|\eta_{j}-\lambda|.
\]
Thus, there exists some $1\le j\le m$ with
\[
|\eta_{j}-\lambda|\le\left(r\cdot(\e/2)^{-n}\right)^{1/m}\le\left(r\cdot(\e/2)^{-n}\right)^{1/(C\log l)}.
\]
This completes the proof of the lemma by taking $c=1/C$.
\end{proof}

The next two lemmata can be used to show that two parameters with exact overlaps found with the help of the previous
lemma must either coincide or be far apart.

\begin{lem}
\label{lem:lb on val of poly}Let $l,n\ge1$, let $\lambda\in\C$
be algebraic, and let $P\in\mathcal{P}_{l}^{(n)}$. Then $P(\lambda)=0$
or $|P(\lambda)|\ge(ln)^{-\deg\lambda}M(\lambda)^{-n}$.
\end{lem}

\begin{proof}
This proof uses the notion of the height of an algebraic number, which is defined by
$H(\eta)=M(\eta)^{1/\deg \eta}$ for $\eta\in\overline\Q$.

Write $d$ for $\deg\lambda$, then $H(\lambda)=M(\lambda)^{1/d}$.
By \cite{Masser-auxiliary}*{Proposition 14.7},
\[
H(P(\lambda))\le\ell_{1}(P)H(\lambda)^{n}\le ln\cdot M(\lambda)^{n/d}.
\]
From $P(\lambda)\in\Q[\lambda]$ it follows that $\deg P(\lambda)\le d$,
hence
\[
M(P(\lambda))=H(P(\lambda))^{\deg P(\lambda)}\le(ln)^{d}\cdot M(\lambda)^{n}.
\]
The lemma now follows from \cite{Masser-auxiliary}*{Proposition 14.13}.
\end{proof}

\begin{lem}
\label{lem:lb on value of poly at root}Let $l,n\ge1$, let $\lambda\in\C$
be a root of some nonzero polynomial in $\mathcal{P}_{l}^{(n)}$, and
let $P\in\mathcal{P}_{l}^{(n)}$. Then $P(\lambda)=0$ or $|P(\lambda)|\ge(ln)^{-2n}$.
\end{lem}

\begin{proof}
Let $E\in\mathcal{P}_{l}^{(n)}\setminus\{0\}$ be with $E(\lambda)=0$,
then $M(\lambda)\le M(E)\le\ell_{1}(E)$. Since $E\in\mathcal{P}_{l}^{(n)}$
we have $\ell_{1}(E)\le ln$, hence $M(\lambda)\le ln$. Additionally,
from $E(\lambda)=0$ and $\deg(E)\le n$ we get $\deg\lambda\le n$.
The lemma now follows from Lemma \ref{lem:lb on val of poly}.
\end{proof}

We state and prove a more precise version of Lemma \ref{lm:Dimitrov}.
It provides an alternative approach to the proof of Conjecture \ref{cn:EO} for Bernoulli convolutions
given in \cite{Var-Bernoulli}, and allows for a relaxation of the transversality property.
This will be important in our proof of Theorems \ref{th:EO-general} and \ref{th:large-lambda} in Section \ref{sc:EO-proof}.
This lemma and its proof was suggested to us by Dimitrov \cite{Dim-personal}.

\begin{lem}\label{lm:Dimitrov-precise}
Let $\l,\eta\in[0,1]$ and $n,n',l,k\in\Z_{>0}$.
Let $0\ne P\in\cP_l^{(n')}$.
Let $\a$ be a number that satisfies
\[
\log \a>  \frac{(n(k+1)+(k+2))\log n'+(n+1)\log l+\log 2}{n'}.
\]
Assume that $\eta \ne \l$ and that $\eta$ is algebraic of degree at most $n$.
Assume
\[
(\a M(\eta))^{n'/k}|P(\l)|^{1/k}\le |\l-\eta|\le (\a M(\eta))^{-n'}.
\]
Then $\eta$ is a zero of $P$ of order at least $k$.
\end{lem}

When we use this lemma, we will take $\a=2$.
Then the assumption on $\a$ clearly holds provided $n'\ge Cn\log n$ for a constant $C$ depending on $k$, and $l<\exp(n'/2n)$.
In particular, the version stated in Lemma \ref{lm:Dimitrov} follows.

The case of $P(\l)=0$ in this lemma gives a refinement of Lemma \ref{lem:dist between distinct roots} when the degrees of the
two algebraic numbers are significantly different and there is additional information available about the Mahler measures.
This is related to a result of Mignotte \cite{Mig-larger-degree}.

\begin{proof}
Denote by $o$ the order of vanishing of $P$ at $\eta$.
That is, $\frac{d^j}{dX^j} P(\eta)=0$ for $0\le j<o$ and $\frac{d^o}{dX^o} P(\eta)\ne0$.
A priori, we allow the possibility of $o=0$, and we
suppose to the contrary that $o\le k$.

By Taylor's theorem with Lagrange remainder term, we have
\[
P(\l)=\frac{d^o}{o!dX^o} P(\xi)(\l-\eta)^o
\]
for some $\xi$ in the closed interval between $\eta$ and $\l$.
If $o=0$, this holds with $\xi=\l$.

We write $\wt P=d^o/(o!dX^o) P$.
We clearly have $\wt P\in\cP_{(n')^ol}^{(n')}$, and
\[
\Big|\frac{d}{dX}\wt P(X)\Big|\le (n')^{o+2}l
\]
for all $X\in[0,1]$.
By the mean value theorem, we have
\[
|\wt P(\eta)|\le (n')^{o+2}l |\xi -\eta| + |\wt P(\xi)|,
\]
hence
\[
|\wt P(\eta)|\le (n')^{o+2}l |\l-\eta| + \frac{|P(\l)|}{|\l-\eta|^o}.
\]

By the definition of $o$ and $\wt P$, we have $\wt P(\eta)\neq 0$.
Therefore, Lemma \ref{lem:lb on val of poly} and $o\le k$ gives
\[
((n')^kln')^{-n}M(\eta)^{-n'}\le  (n')^{k+2}l |\l-\eta| + \frac{|P(\l)|}{|\l-\eta|^k}.
\]
By the definition of $\a$, we can write
\[
2(\a M(\eta))^{-n'} < |\l-\eta|+  \frac{|P(\l)|}{|\l-\eta|^k},
\]
which contradicts the assumptions on $|\l-\eta|$.
\end{proof}

\section{Self-similar measures over function fields}\label{sc:ssm-Cx}

As we explained in Section \ref{sc:outline}, one of the main contributions of this paper and the main difficulty
in the proof of Theorem \ref{th:BV2-general} is proving that the entropy rate $h(\gamma)$ of a curve $\gamma\in\Gamma$
passing near a point $(\l,\tau)\in(0,1)\times\R$ without exact overlaps cannot be much smaller than
$\min\{H(p),\log \l^{-1}\}$.

We will achieve this goal in Corollary \ref{cr:no-bad-curves} in the next section.
In this one, we lay the groundwork by studying the entropy rate of non-degenerate curves.
We prove that the entropy rate is lower semi-continuous with respect to the
coefficient-wise convergence of the power series defining the curves.
To this end, we study self-similar measures defined over function fields.
We introduce a notion of dimension in analogy with the notion of entropy dimension of measures on $\R$, which is easily seen
to be lower semi-continuous.
We will show that this dimension equals the entropy rate of the corresponding non-degenerate curve.
This result can be seen as a function field analogue of (a stronger form of) Conjecture \ref{cn:EO}.

We begin by explaining the setup in detail.
As in Section \ref{sc:outline}, let $\F$ be a countable field of characteristic $0$. For the purposes of this paper, we could take $\F=\Q$, but we consider a more general setting in this section for the sake of possible future applications, and because it makes no difference in the proofs. Denote by $\F((X))$ the field of formal Laurent series over $\F$, and by $\F[[X]]$ the ring of formal power series over $\F$. We endow $\F((X))$ with the non-Archimedean absolute value $|R|=2^{-n}$, where $n$ is the index of the first non-zero coefficient of $0\ne R\in \F((X))$. For $R=0$ we set $|R|=0$. Note that $\F[[X]]$ is the closed unit ball of $\F((X))$ with respect to this absolute value.
We stress that the topology induced by $|\cdot|$ is not locally compact.

Let $\xi_0,\xi_1,\ldots$ be i.i.d. random variables with $\P\{\xi_0=j\}=p_j$ for $1\le j\le m$. For $R\in \F[[X]]$ and $n\ge1$ we set,
\[
A_{R}=\sum_{j=0}^{\infty} T_{\xi_j}(1,R(X))X^j\qquad\text{and}\qquad A_{R}^{(n)}=\sum_{j=0}^{n-1} T_{\xi_j}(1,R(X))X^j.
\] 
We denote the distribution of $A_R$ by $\mu_R$. By subadditivity, the limit
\begin{equation}\label{def of h(R)}
h(R)=\lim_{n\to \infty}\frac{1}{n} H(A_{R}^{(n)})
\end{equation}
always exists. Throughout this section, $\{\wt A_{R}^{(n)}\}_{n\ge1}$ denotes an independent copy of the process $\{A_{R}^{(n)}\}_{n\ge1}$.

Let $A=\a_0+\a_1X+\a_2X^2+\ldots$ be an $\F[[X]]$-valued random element, so $\a_0,\a_1,\ldots$
are $\F$-valued random variables.
For $l\ge n\ge0$ we write
\[
H(A;n)=H\Big(\sum_{j=0}^{n-1}\a_j X^j\Big),
\]
where $H(A;0)$ is defined to be $0$, and
\[
H(A;l|n)=H(A;l)-H(A;n).
\]
We define
\[
\dim A = \lim_{n\to \infty} \frac{H(A;n)}{n},
\]
provided the limit exists. If $\mu$ is the law of $A$, we write $\dim\mu$ for $\dim A$.

Note that for $R\in\F[[X]]$ and $l\ge n\ge1$,
\[
H(A_R;n)=H(A_R^{(l)};n),
\]
which implies $\frac{1}{n}H(A_R;n)\le H(p)$. As the following lemma shows, it turns out that $\dim\mu_R$ always exists.

\begin{lem}
For each $R\in\F[[X]]$ and $n\ge1$,
\[
\frac{1}{n}H(A_R;n) \le \frac{1}{n+1}H(A_R;n+1).
\]
In particular $\dim\mu_R$ exists, and the map $R \to \dim\mu_R$ is lower semicontinuous with respect to the metric induced by $|\cdot|$.
\end{lem}

\begin{proof}
By the concavity of conditional entropy it follows that for $k\ge1$,
\begin{align*}
H(A_R;k+1| k) = & H(\wt A_{R}^{(1)}+XA_R;k+1| k) \\
\ge & H(XA_R;k+1| k) = H(A_R;k| k-1).
\end{align*}
Iterating this inequality, we conclude that
\[
H(A_R;n+1|n)\ge H(A_R;k+1|k)
\]
for all $k=0,\ldots,n-1$.
Thus,
\begin{align*}
\frac{1}{n}H(A_R;n)=& \frac{1}{n}\sum_{k=0}^{n-1}H(A_R;k+1|k)\\
\le & \frac{1}{n+1}\sum_{k=0}^{n}H(A_R;k+1| k)= \frac{1}{n+1}H(A_R;n+1).
\end{align*}
The lemma now follows since for each $n\ge1$ the map $R \to H(A_R;n)$ is continuous with respect to $|\cdot|$.
\end{proof}

The main purpose of this section is to prove the following result.

\begin{prp}\label{pr:dim=h}
For all $R\in\F[[X]]$, we have
\[
\dim \mu_R=h(R).
\]
\end{prp}

Since $\dim\mu_R$ is lower semi-continuous in $R$ with respect to the metric $|\cdot|$, we have the following
immediate corollary.

\begin{cor}\label{cr:h-semi-cont}
The function $h(\cdot)$ is lower semi-continuous on $\F[[X]]$ with respect to the $|\cdot|$ metric.
That is to say, let $R,R_{1},R_{2},\ldots\in\F[[X]]$
be with $|R-R_{n}|\overset{n}{\rightarrow}0$.
Let $\e>0$, and suppose that $h(R_{n})\le H(p)-\e$ for all $n\ge1$.
Then $h(R)\le H(p)-\e$.
In particular, there exists $0\ne Q(X,Y_{1},Y_{2})\in\cup_{n\ge1}\mathcal{Q}^{(n)}$ such
that $Q(X,1,R(X))=0$.
\end{cor}

To prove Proposition \ref{pr:dim=h}, we adapt to our function field setting the techniques
introduced by Hochman in \cite{Hoc-self-similar}.
We outline the strategy.
We first note that
\begin{equation}\label{eq:ssimilarity}
A_{R}^{(4n)}=A_{R}^{(n)}+X^n \wt A_{R}^{(3n)}
\end{equation}
for each $n$, where recall that $\wt A_{R}^{(3n)}$ is independent of $A_{R}^{(n)}$ and has the same distribution as $A_{R}^{(3n)}$.
By definition, we have
\[
\dim \mu_R = \lim_{n\to \infty} \frac{H(A_{R}^{(n)};n)}{n}.
\]
This means that for any $\e>0$,
\begin{align*}
H(A_{R}^{(4n)};4n|n)=&H(A_{R}^{(4n)};4n)-H(A_{R}^{(n)};n)\le(1+\e)\cdot 3n\dim \mu_R,\\
H(X^n\wt A_{R}^{(3n)};4n|n)=&H(A_{R}^{(3n)};3n)\ge(1-\e)\cdot 3n\dim\mu_R
\end{align*}
provided $n$ is sufficiently large.

In the proof of Proposition \ref{pr:dim=h}, we suppose to the contrary that $\dim \mu_R<h(R)$.
We reach a contradiction exploiting \eqref{eq:ssimilarity} and the following theorem yielding that the entropy
of the sum of two independent random variables is significantly larger than that of either summands under suitable conditions.

\begin{thm}\label{thm:ent growth}
For every $\e>0$, there exists $\delta>0$
such that the following holds.
Let $A$ and $B$ be independent random
elements of $\F[[X]]$. Let $n\ge1$ and suppose that $H(A;n| n-1)<\e^{-1}$
and $H(B;n| n-1)>\e$. Then,
\[
H(A+B;n| n-1)>H(A;n| n-1)+\delta.
\]
\end{thm}

This theorem is proved in Sections \ref{sc:entC} and \ref{sc:entCx}.
In Section \ref{sc:exp-sep}, we verify the conditions in Theorem \ref{thm:ent growth} in the setting of \eqref{eq:ssimilarity}
using the hypothesis $\dim \mu_R<h(R)$.
Finally, we complete the proof of Proposition \ref{pr:dim=h} in Section \ref{sc:proof-dim=h}.

\subsection{Entropy growth in torsion free Abelian groups}\label{sc:entC}

In this section, we study entropy growth of measures on $\F$ under convolution.
Our purpose is to prove the following result, which is the main ingredient
in our proof of Theorem \ref{thm:ent growth}.

\begin{prp}\label{prop:ent increase in C}
For every $\e>0$, there exists
$\delta>0$ such that the following holds.
Let $\mu$ and $\nu$ be
probability measures on $\F$.
Suppose that $H(\mu)<\e^{-1}$ and $H(\nu)>\e$.
Then $H(\nu*\mu)>H(\mu)+\delta$.
\end{prp}

This result depends only on the additive structure of $\F$, and it could be replaced by any countable Abelian torsion free group.

We begin by recalling the following useful fact that allows us to bound from below the growth
of entropy under convolution by a measure in terms of the growth under an iterated convolution.

\begin{lem}\label{lem:Kaimanovich-Vershik}
Let $\Gamma$ be a countable Abelian group.
Let $\mu$ and $\nu$ be probability
measures on $\Gamma$ with $H(\mu),H(\nu)<\infty$.
Then for every $n\ge1$,
\[
H(\mu*(\nu^{*n}))-H(\mu)\le n\cdot(H(\mu*\nu)-H(\mu)).
\]
\end{lem}

This lemma goes back to Kaimanovich and Vershik \cite{kaimanovich1983random} in some form.
See \cite{Hoc-self-similar}*{Lemma 4.7} for a proof in this formulation.

In the proof of Proposition \ref{prop:ent increase in C}, we decompose $\nu$ as a convex combination
of a Bernoulli measure and another measure, and exploit the next lemma,
which estimates entropy growth under convolution by a Bernoulli measure.

\begin{lem} \label{lem:ent increase by conv with measure supported on two points}
For every $M>1$, there exists $\eta>0$ such that the following holds.
Let $\mu$ be a probability measure on $\F$ with $H(\mu)<M$.
Let $a,b\in\F$ be with $a\ne b$, and write $\nu=\frac{\delta_{a}+\delta_{b}}{2}$.
Then $H(\nu*\mu)>H(\mu)+\eta$.
\end{lem}

\begin{proof}
Let $M>1$ and let $n\ge1$ be large with respect to $M$.
Let $\mu,a,b$
and $\nu$ be as in the statement of the lemma.

The maximal $\nu^{*n}$-measure of a point tends to $0$ as $n\to\infty$ uniformly in $a,b$. 
Hence for sufficiently large $n$ (in a manner not depending
on $a,b$), we have $H(\nu^{*n})>M+1$.
By this and Lemma \ref{lem:Kaimanovich-Vershik},
\[
H(\mu)+1<H(\nu^{*n})\le H(\mu*(\nu^{*n}))\le H(\mu)+n\cdot(H(\mu*\nu)-H(\mu)).
\]
This gives $H(\mu*\nu)>H(\mu)+\frac{1}{n}$, which proves the lemma
with $\eta=\frac{1}{n}$.
\end{proof}

\begin{proof}[Proof of Proposition \ref{prop:ent increase in C}]
Let $0<\e<1$, let $0<\rho<\frac{1}{2}$ be small with respect
to $\e$, and let $0<\eta<\frac{1}{2}$ be small with respect
to $\rho$.
Let $\mu$ and $\nu$ be probability measures on $\F$
with $H(\mu)<\e^{-1}$ and $H(\nu)>\e$.
We shall consider three cases,
\begin{enumerate}
\item there exist $a,b\in\F$ with $a\ne b$ and $\nu\{a\},\nu\{b\}>\eta$;
\item there exists $a\in\F$ such that $\nu\{a\}<1-\rho$ and $\nu\{b\}\le\eta$
for all $b\in\F\setminus\{a\}$;
\item there exists $a\in\F$ with $\nu\{a\}\ge1-\rho$.
\end{enumerate}
Clearly these cases cover all possibilities.

Suppose that the first alternative holds. Write
\[
\nu_{1}=\frac{\delta_{a}+\delta_{b}}{2}\quad\text{ and }\quad\nu_{2}=\frac{\nu-\eta\delta_{a}-\eta\delta_{b}}{1-2\eta}.
\]
Then $\nu_{1}$ and $\nu_{2}$ are probability measures on $\F$
and,
\[
\nu=2\eta\nu_{1}+(1-2\eta)\nu_{2}.
\]
By concavity, $H(\mu)<\e^{-1}$ and Lemma \ref{lem:ent increase by conv with measure supported on two points},
\begin{align*}
H(\nu*\mu)\ge&2\eta\cdot H(\nu_{1}*\mu)+(1-2\eta)H(\nu_{2}*\mu)\\
>&2\eta(H(\mu)+\eta)+(1-2\eta)H(\mu)\\
=&H(\mu)+2\eta^{2},
\end{align*}
which proves the proposition in this case.

Suppose now that the second alternative holds.
Write
\[
\nu_{1}=\frac{\nu-\nu\{a\}\delta_{a}}{1-\nu\{a\}},
\]
then $\nu_{1}$ is a probability measure,
\[
\nu_{1}\{b\}\le\frac{\eta}{1-\nu\{a\}}<\frac{\eta}{\rho}\text{ for all }b\in\F,
\]
and $\nu=\nu\{a\}\delta_{a}+(1-\nu\{a\})\nu_{1}$. From this and since
we may assume $\log\frac{\rho}{\eta}>\e^{-1}+1>H(\mu)+1$,
\begin{align*}
H(\nu*\mu)  \ge & \nu\{a\}H(\delta_{a}*\mu)+(1-\nu\{a\})H(\nu_{1}*\mu)\\
  \ge & \nu\{a\}H(\mu)+(1-\nu\{a\})H(\nu_{1})\\
  > & \nu\{a\}H(\mu)+(1-\nu\{a\})\log\frac{\rho}{\eta}\\
  > & \nu\{a\}H(\mu)+(1-\nu\{a\})(H(\mu)+1)\\
  > & H(\mu)+\rho,
\end{align*}
which proves the proposition also in this case.

Lastly, suppose that the third alternative holds.
Write
\begin{align*}
\nu_{1}=&\frac{\nu-\nu\{a\}\delta_{a}}{1-\nu\{a\}},\\
\mathcal{C}=&\{\{b\}:b\in\F\},\\
\mathcal{E}=&\{\{a\},\F\setminus\{a\}\}.
\end{align*}
Since $\rho$ is small with respect to $\e$, we may assume that
$H(\nu;\mathcal{E})<\e/2$, and so
\[
\e<H(\nu)=H(\nu;\mathcal{E})+H(\nu;\mathcal{C}|\mathcal{E})<\e/2+(1-\nu\{a\})H(\nu_{1}).
\]
We can also assume $\rho/\e<\e/4$, and
so
\[
(1-\nu\{a\})H(\mu)\le\frac{\rho}{\e}<\frac{\e}{4}.
\]
Hence,
\begin{align*}
H(\nu*\mu)  \ge & \nu\{a\}H(\delta_{a}*\mu)+(1-\nu\{a\})H(\nu_{1}*\mu)\\
  \ge & \nu\{a\}H(\mu)+(1-\nu\{a\})H(\nu_{1})\\
  > & \nu\{a\}H(\mu)+\e/2\\
  > & \nu\{a\}H(\mu)+(1-\nu\{a\})H(\mu)+\e/4\\
  = & H(\mu)+\e/4,
\end{align*}
which completes the proof of the proposition.
\end{proof}

\subsection{Entropy growth in function fields}\label{sc:entCx}

The purpose of this section is to prove Theorem \ref{thm:ent growth}, which we recall now.

\begin{thm*}
For every $\e>0$, there exists $\delta>0$
such that the following holds. Let $A$ and $B$ be independent random
elements of $\F[[X]]$.
Let $n\ge1$ and suppose that $H(A;n| n-1)<\e^{-1}$
and $H(B;n| n-1)>\e$. Then,
\[
H(A+B;n| n-1)>H(A;n| n-1)+\delta.
\]
\end{thm*}

\begin{proof}
For $k\ge1$ and $a_{0},\ldots,a_{k-1}\in\F$, write
\[
C(a_{0},\ldots,a_{k-1})=\Big\{\sum_{l=0}^{\infty}b_{l}x^{l}\in\F[[X]]:b_{l}=a_{l}\text{ for all }0\le l<k\Big\}.
\]
Set
\[
\mathcal{C}_{k}=\{C(a_{0},\ldots,a_{k-1}):a_{0},\ldots,a_{k-1}\in\F\}.
\]
Then $\mathcal{C}_{k}$ is a partition of $\F[[X]]$.
Also, write $\mathcal{C}_{0}$ for the trivial partition $\{\F[[X]]\}$.
The collection $\cup_{k\ge0}\mathcal{C}_{k}$ forms a basis for the
topology on $\F[[X]]$ induced by the metric $|\cdot|$.
Using this notation, we can write
\[
H(\mu;n|k)=H(\mu;\cC_n|\cC_k)
\]
for a probability measure $\mu$ and $n\ge k\in\Z_{\ge 0}$,
where the notation on the right hand side is the entropy of the measure with respect to the partition $\cC_n$
conditioned on the partition $\cC_k$.

Given a probability measure $\mu$ and $C\in\cC_{k}$ with $\mu(C)>0$, we write $\mu_C$ for
the restriction of $\mu$ to $C$ normalized to be a probability measure.
Given $P\in\F[[X]]$, let $\mathcal{C}_{k}(P)$ be the unique
$C\in\mathcal{C}_{k}$ with $P\in C$.
Using these notation, we can write
\[
H(\mu;\cC_n|\cC_k)=\int H(\mu_{\cC_{k}(P)};\cC_n)d\mu(P).
\]

Let $\e>0$, let $\eta>0$ be small with respect to $\e$,
and let $\delta>0$ be small with respect to $\eta$.
Let $\mu$ and $\nu$ be Borel probability measures on $\F[[X]]$.
(Here we consider the Borel $\s$-algebra generated by the topology induced by the $|\cdot|$ metric.)
Using the above notation, we can express the statement of the theorem as follows.
Let $n\ge1$ and suppose that $H(\mu;\mathcal{C}_{n}|\mathcal{C}_{n-1})<\e^{-1}$
and $H(\nu;\mathcal{C}_{n}|\mathcal{C}_{n-1})>\e$.
We need to show that
\[
H(\mu*\nu;\mathcal{C}_{n}|\mathcal{C}_{n-1})>H(\mu;\mathcal{C}_{n}|\mathcal{C}_{n-1})+\delta.
\]

Write
\[
E=\{P\in\F[[X]]:\nu(\mathcal{C}_{n-1}(P))>0\text{ and }H(\nu_{\mathcal{C}_{n-1}(P)};\mathcal{C}_{n})>\eta\}.
\]

Suppose first that $\nu(E)>\eta$.
Write
\[
F=\{P\in\F[[X]]:\mu(\mathcal{C}_{n-1}(P))>0\text{ and }H(\mu_{\mathcal{C}_{n-1}(P)};\mathcal{C}_{n})\le\eta^{-2}\}.
\]
We have
\[
\e^{-1}>\int H(\mu_{\mathcal{C}_{n-1}(P)};\mathcal{C}_{n})\:d\mu(P)\ge\eta^{-2}\cdot\mu(F^{c}),
\]
hence $\mu(F^{c})<\frac{\eta^{2}}{\e}$.

For $P\in F$ and $Q\in E$
\begin{equation}
H(\mu_{\mathcal{C}_{n-1}(P)}*\nu_{\mathcal{C}_{n-1}(Q)};\mathcal{C}_{n})>H(\mu_{\mathcal{C}_{n-1}(P)};\mathcal{C}_{n})+\delta,\label{eq:ent gain}
\end{equation}
which follows by applying Proposition \ref{prop:ent increase in C} to the pushforwards of $\mu_{\mathcal{C}_{n-1}(P)}$ and $\nu_{\mathcal{C}_{n-1}(Q)}$ via the map sending $\sum_{l=0}^{\infty}a_{l}x^{l}\in\F[[X]]$ to $a_n$. Also, by the concavity of entropy it follows that for every $P,Q\in\F[[X]]$
with $\mu(\mathcal{C}_{n-1}(P)),\nu(\mathcal{C}_{n-1}(Q))>0$,
\begin{equation}
H(\mu_{\mathcal{C}_{n-1}(P)}*\nu_{\mathcal{C}_{n-1}(Q)};\mathcal{C}_{n})\ge H(\mu_{\mathcal{C}_{n-1}(P)};\mathcal{C}_{n}).\label{eq:no ent loss}
\end{equation}

We may assume $\eta/\e<1/2$.
Using this together with $\nu(E)>\eta$, we get
\[
\mu\times\nu(F\times E)\ge\nu(E)-\mu(F^{c})>\eta-\frac{\eta^{2}}{\e}>\eta/2.
\]
From this, \eqref{eq:ent gain}, \eqref{eq:no ent loss}, and by concavity, we have
\begin{align*}
H(\mu*\nu;\mathcal{C}_{n}|\mathcal{C}_{n-1})
= & H\left(\iint\mu_{\mathcal{C}_{n-1}(P)}*\nu_{\mathcal{C}_{n-1}(Q)}\:d\mu(P)\:d\nu(Q);\mathcal{C}_{n}|\mathcal{C}_{n-1}\right)\\
\ge & \iint H\left(\mu_{\mathcal{C}_{n-1}(P)}*\nu_{\mathcal{C}_{n-1}(Q)};\mathcal{C}_{n}\right)\:d\mu(P)\:d\nu(Q)\\
\ge & \int H\left(\mu_{\mathcal{C}_{n-1}(P)};\mathcal{C}_{n}\right)\:d\mu(P)+\delta\cdot\mu\times\nu(F\times E)\\
> & H(\mu;\mathcal{C}_{n}|\mathcal{C}_{n-1})+\frac{\delta\eta}{2},
\end{align*}
which proves the theorem in the case $\nu(E)>\eta$.

Suppose now that $\nu(E)\le\eta$. We have,
\[
\e<H(\nu;\mathcal{C}_{n}|\mathcal{C}_{n-1})\le\int_{E}H(\nu_{\mathcal{C}_{n-1}(P)};\mathcal{C}_{n})\:d\nu(P)+\eta.
\]
From this, by concavity, and by $H(\mu;\mathcal{C}_{n}|\mathcal{C}_{n-1})<\e^{-1}$,
\begin{align*}
H(\mu*\nu;\mathcal{C}_{n}|\mathcal{C}_{n-1})
\ge& \iint H(\mu_{\mathcal{C}_{n-1}(P)}*\nu_{\cC_{n-1}(Q)};\mathcal{C}_{n})\:d\mu(P)\:d\nu(Q)\\
\ge & \int_{E^{c}}\int H(\mu_{\mathcal{C}_{n-1}(P)};\mathcal{C}_{n})\:d\mu(P)\:d\nu(Q)\\
&+ \int_{E}\int H(\nu_{\mathcal{C}_{n-1}(Q)};\mathcal{C}_{n})\:d\mu(P)\:d\nu(Q)\\
> & (1-\eta)H(\mu;\mathcal{C}_{n}|\mathcal{C}_{n-1})+\e-\eta\\
\ge & H(\mu;\mathcal{C}_{n}|\mathcal{C}_{n-1})-\eta/\e+\e-\eta.
\end{align*}
Since we can assume that $\e-\eta/\e-\eta>\delta$,
this completes the proof of the theorem.
\end{proof}

\subsection{Exponential separation}\label{sc:exp-sep}

The purpose of this section is to study the separation between points in the
support of $A_{R}^{(n)}$ and deduce bounds for its entropy on suitable scales.
We prove two lemmata that will be used in the proof of Proposition \ref{pr:dim=h}
to verify the conditions in Theorem \ref{thm:ent growth}.

\begin{lem}
\label{lem:full ent}Let $R\in\F[[X]]$ and suppose that $Q(X,1,R(X))\ne0$
for all $0\ne Q\in\cup_{n\ge1}\mathcal{Q}^{(n)}$. Then for every $N\ge1$,
there exists $n\ge N$ with $H(A_{R}^{(n)};4n)=nH(p)$.
\end{lem}

\begin{proof}
Suppose to the contrary that there exists $N\ge1$ with,
\[
H(A_{R}^{(n)};4n)<nH(p)\text{ for all }n\ge N.
\]
Since
\[
T_{i}(Y_{1},Y_{2})\ne T_{j}(Y_{1},Y_{2})\text{ for }1\le i<j\le m,
\]
it follows that for every $n\ge N$ there exists $0\ne Q_{n}\in\mathcal{Q}^{(n)}$
with $|Q_{n}(X,1,R(X))|\le2^{-4n}$.

For every $n\ge N$, there exist $P_{n,1},P_{n,2}\in\cP_{L}^{(n)}$ with
\[
Q_{n}(X,Y_{1},Y_{2})=P_{n,1}(X)Y_{1}+P_{n,2}(X)Y_{2}.
\]
From $|Q_{n}(X,1,R(X))|\le2^{-4n}$ and $Q_{n}\ne0$ it follows that
$P_{n,2}\ne0$. This together with $\deg(P_{n,2})<n$ implies $|P_{n,2}|\ge2^{-n}$. Thus
\[
|1/P_{n,2}(X)|=|P_{n,2}(X)|^{-1}\le2^n,
\]
and so
\[
\left|\frac{P_{n,1}(X)}{P_{n,2}(X)}+R(X)\right|=\left|\frac{Q_{n}(X,1,R(X))}{P_{n,2}(X)}\right|\le2^{-3n}.
\]

Therefore,
\begin{align*}
\left|\frac{P_{n,1}}{P_{n,2}}-\frac{P_{n+1,1}}{P_{n+1,2}}\right|
\le&\max\left\{ \left|\frac{P_{n,1}(X)}{P_{n,2}(X)}+R(X)\right|,\left|\frac{P_{n+1,1}(X)}{P_{n+1,2}(X)}+R(X)\right|\right\}\\ \le&2^{-3n},
\end{align*}
which gives,
\[
|P_{n,1}P_{n+1,2}-P_{n+1,1}P_{n,2}|\le2^{-3n}|P_{n,2}|\cdot|P_{n+1,2}|\le2^{-3n}.
\]
Note that,
\[
\deg(P_{n,1}P_{n+1,2}-P_{n+1,1}P_{n,2}) < 2n,
\]
hence
$P_{n,1}P_{n+1,2}=P_{n+1,1}P_{n,2}$.

For all $n\ge N$, we have
\[
\frac{P_{n,1}(X)}{P_{n,2}(X)}=\frac{P_{n+1,1}(X)}{P_{n+1,2}(X)}
\]
and
\[
\left|\frac{P_{n,1}(X)}{P_{n,2}(X)}+R(X)\right|\le2^{-3n},
\]
which clearly implies
\[
R(X)=-\frac{P_{N,1}(X)}{P_{N,2}(X)}.
\]
It follows that $Q_{N}(X,1,R(X))=0$, which contradicts our assumption
and completes the proof of the lemma.
\end{proof}

\begin{lem}\label{lem:full ent2}
Let $R\in\F[[X]]$ and suppose that $Q(X,1,R(X))=0$
for some $0\ne Q\in\mathcal{Q}^{(N)}$ for some $N\ge1$.
Then for every $n\ge N$, we have
$H(A_{R}^{(n)};4n)\ge n h(R)$.
\end{lem}
\begin{proof}
Fix $n\ge N$, and let $0 \ne \wt Q\in\cQ^{(n)}$ such that $|\wt Q(X,1,R(X))|\le 2^{-4n}$.
We set out to prove that $\wt Q(X,1,R(X))=0$ for all such $\wt Q$.
This implies
\[
H(A_{R}^{(n)};4n)=H(A_{R}^{(n)})\ge n h(R)
\]
completing the proof of the lemma.

We write
\begin{align*}
Q(X,Y_{1},Y_{2})=&P_{1}(X)Y_{1}+P_{2}(X)Y_{2}\\
\wt Q(X,Y_{1},Y_{2})=&\wt P_{1}(X)Y_{1}+\wt P_{2}(X)Y_{2},
\end{align*}
with some $P_{1},P_{2},\wt P_1,\wt P_2\in\cP_{L}^{(n)}$.
By $Q(X,1,R(X))=0$, we have
\[
\frac{P_1(X)}{P_2(X)}+R(X)=0.
\]
By $|\wt Q(X,1,R(X))|\le 2^{-4n}$ and $|\wt P_2|\ge 2^{-n}$, we have
\[
\left|\frac{\wt P_{1}(X)}{\wt P_{2}(X)}+R(X)\right|=\left|\frac{\wt Q(X,1,R(X))}{\wt P_{2}(X)}\right|\le2^{-3n}.
\]

Combining the last two formulae, we get
\[
\left|\frac{P_{1}}{P_{2}}-\frac{\wt P_{1}}{\wt P_{2}}\right|\le 2^{-3n}.
\]
As above, this yields
\[
|P_1\wt P_2-\wt P_1 P_2|\le 2^{-3n}.
\]
Since $\deg(P_1\wt P_2-\wt P_1 P_2)\le n+N$, we have
\[
P_1\wt P_2-\wt P_1 P_2=0.
\]
Hence,
\[
\frac{P_1}{P_2}=\frac{\wt P_1}{\wt P_2},
\]
and $\wt Q(X,1,R(X))=0$, as required.
\end{proof}

\subsection{Concluding the proof}\label{sc:proof-dim=h}

Before we complete the proof of Proposition \ref{pr:dim=h}, we need one last ingredient
that will be used to verify a condition of Theorem \ref{thm:ent growth}.

\begin{lem}
\label{lem:bounded ent}For every $R\in\F[[X]]$ and $n,k\ge1$,
we have
\[
H(A_{R}^{(n)};k| k-1)\le H(p).
\]
\end{lem}

\begin{proof}
Fix some $k$ and $n$.
For every
$l\ge\max\{n,k\}$, we can write
\begin{align*}
H(A_{R};l | l-1)=&H(A_{R}^{(l)};l | l-1)\\
=&H(A_{R}^{(l-k)}+X^{l-k}\wt A_{R}^{(k)};l | l-1)\\
\ge& H(X^{l-k}\wt A_{R}^{(k)};l | l-1)\\
=&H(A_{R}^{(k)};k | k-1)\\
=&H(A_{R}^{(l)};k | k-1)\\
=&H(A_{R}^{(n)}+X^{n}\wt A_{R}^{(l-n)};k | k-1)\\
\ge& H(A_{R}^{(n)};k|k-1).
\end{align*}
Therefore,
\begin{align*}
H(p)\ge& \dim\mu_R=\lim_N\frac{H(A_{R};N)}{N}\\
=&\lim_N\frac{1}{N}\sum_{l=1}^{N}H(A_{R};l| l-1)\ge H(A_{R}^{(n)};k|k-1),
\end{align*}
which completes the proof of the lemma.
\end{proof}

\begin{proof}[Proof of Proposition \ref{pr:dim=h}]
Suppose to the contrary that $\dim \mu_R<h(R)$, and write
\[
\e=h(R)-\dim\mu_R.
\]

By the definition of $\dim\mu_R$,
\begin{align}
\lim_n\frac{1}{3n}H(A_{R}^{(4n)};4n| n)=&\lim_n\frac{1}{3n}\left(H(A_{R}^{(4n)};4n)-H(A_{R}^{(n)};n)\right)\nonumber\\
=&\dim\mu_R.\label{eq:4n|n}
\end{align}

Let $n\ge1$ be large with respect to $H(p)$ and $\e$, and assume
\[
H(A_{R}^{(n)};4n)\ge nh(R).
\]
By Lemmata \ref{lem:full ent} and \ref{lem:full ent2}, there is always such a choice of $n$.
We may also assume that $n$ is sufficiently large so that
\[
H(A_{R}^{(n)};n) < n(\dim \mu_R +\e/2).
\]
Hence
\begin{equation}
\frac{1}{3n}\sum_{k=n}^{4n-1}H(A_{R}^{(n)};k+1| k)=\frac{1}{3n}H(A_{R}^{(n)};4n)-\frac{1}{3n}H(A_{R}^{(n)};n)>\e/6.\label{eq:nontrivial ent}
\end{equation}
Let $\mathcal{N}$ be the set of all integers $k\in[n,4n)$ with
\[
H(A_{R}^{(n)};k+1| k)>\frac{\e}{20}.
\]
From \eqref{eq:nontrivial ent} and Lemma \ref{lem:bounded ent} it
follows that we must have
\[
\frac{1}{3n}|\mathcal{N}|>\frac{\e}{10H(p)}.
\]

Another application of Lemma \ref{lem:bounded ent} shows that for
every $k\in[n,4n)$,
\[
H(X^{n}\wt A_{R}^{(3n)};k+1| k)=H(\wt A_{R}^{(3n)};k+1-n| k-n)\le H(p).
\]

By Theorem \ref{thm:ent growth}, we can write
\begin{align*}
\frac{1}{3n}H(A_{R}^{(4n)};4n| n)  \ge & \frac{1}{3n}\sum_{k=n}^{4n-1}1_{\{k\in\mathcal{N}\}}H(A_{R}^{(n)}+X^{n}\wt A_{R}^{(3n)};k+1| k)\\
  &+  \frac{1}{3n}\sum_{k=n}^{4n-1}1_{\{k\notin\mathcal{N}\}}H(X^{n}\wt A_{R}^{(3n)};k+1| k)\\
  \ge & \frac{1}{3n}\sum_{k=n}^{4n-1}H(X^{n}\wt A_{R}^{(3n)};k+1| k)+\frac{|\mathcal{N}|}{3n}\delta\\
  > & \frac{1}{3n}\sum_{k=n}^{4n-1}H(A_{R}^{(3n)};k+1-n| k-n)+\frac{\e\delta}{10H(p)}\\
  = & \frac{1}{3n}H(A_{R}^{(3n)};3n)+\frac{\e\delta}{10H(p)},
\end{align*}
where $\d$ is as in Theorem \ref{thm:ent growth} applied with $\min(\e/20, H(p)^{-1})$ in the role of $\e$.

From this and by assuming that $n$ is large enough,
\[
\frac{1}{3n}H(A_{R}^{(4n)};4n|n)\ge \dim\mu_R+\frac{\e\delta}{20H(p)}.
\]
Since $n$ may be chosen to be arbitrarily large this contradicts
\eqref{eq:4n|n} and completes the proof of the proposition.
\end{proof}

\section{Parameters without exact overlaps repel low entropy curves}\label{sc:repel}

Recall the notation $\cP_L^{(n)}$, $\cP_L$, $\cR_L$, $\cQ^{(n)}$ and $\Gamma$
from Section \ref{sc:general-setting}.
Recall also the random functions $A_U^{(n)}$ and the entropy rate $h(U)$ that are defined for
subsets $U\subset(0,1)\times\R$.
We write $B(x,r)$ for the open ball of radius $r$ around the point $x$ in a metric space.
We write $\overline B(x,r)$ for the closed ball.

The purpose of this section is to prove the following proposition and corollary.

\begin{prp}\label{pr:no-bad-curves}
Let $(\l_0,\tau_0)\in(0,1)\times(\R\setminus\{0\})$,
and let $h_0\in(0,\log\l_{0}^{-1})$.
Assume that for each $\e>0$, there is a curve $\gamma\in\Gamma$ that intersects
$B((\l_0,\tau_0),\e)$ and $h(\gamma)\le h_0$.
Then there is a curve $\gamma_0\in\Gamma$ that contains $(\l_0,\tau_0)$ such that $h(\gamma_0)\le h_0$.
\end{prp}

In what follows, with some abuse of notation, we often say that $(\l,\tau)\in(0,1)\times\R$ has exact overlaps if the corresponding IFS \eqref{eq:IFS} has exact overlaps.

\begin{cor}\label{cr:no-bad-curves}
Let $(\l_0,\tau_0)\in(0,1)\times\R$ without exact overlaps.
Then for all $h_0<\min\{\log \l_0^{-1},H(p)\}$, there is a neighborhood of $(\l_0,\tau_0)$ that
is not intersected by a curve $\gamma\in\Gamma$ with $h(\gamma)\le h_0$.
\end{cor}

The following simple facts are immediate consequences of the definitions and they will be used repeatedly.

\begin{itemize}
\item $h(U)<H(p)$ if and only if there is $n$ and $Q\in\cQ^{(n)}$ such that $Q(\l,1,\tau)=0$ for all $(\l,\tau)\in U$.
\item $h(U_1)\le h(U_2)$ if $U_1\subset U_2$.
\item If $h(U)<H(p)$, then all $(\l,\tau)\in U$ has exact overlaps.
\item If $\gamma\in\Gamma$ is a non-degenerate curve that corresponds to $R\in\cR_L\cap\Q[[x]]$, then $h(\gamma)=h(R)$. Here $h(R)$ is as defined in \eqref{def of h(R)}.
\item If $\gamma=\{\l_0\}\times\R$ is a degenerate curve, then $h(\gamma)=h(\l_0,\tau)$ for any
$\tau\in\R\backslash\Q(\l)$.
\end{itemize}

The next lemma will be used in the case when the curves are non-degenerate and they have no singularities near $\l_0$. Recall the absolute value $|\cdot|$ on $\Q[[x]]$, defined in Section \ref{sc:ssm-Cx}. 

\begin{lem}\label{lm:no-singularities}
Let $(\l_0,\tau_0)\in(0,1)\times\R$.
Let $R_1,R_2\ldots\in\cR_L\cap\Q[[x]]$, and let $\gamma_1,\gamma_2,\ldots\in\Gamma$ be the corresponding non-degenerate
curves.
Assume that
\[
\dist(\gamma_n,(\l_0,\tau_0))\to 0.
\]
Assume further that there is a neighborhood of $\l_0$ in $\C$ where none of the $R_n$ has a pole.
Then there is $R\in\cR_L\cap\Q[[x]]$ that is the limit of a subsequence $R_{n_j}$ both in the $|\cdot|$ metric and uniformly in
a neighborhood of $\l_0$.
In particular, $R(\l_0)=\tau_0$.
\end{lem}

\begin{proof}
For each $n$, fix $P_{1,n},P_{2,n}\in\cP_L$ such that $R_n=P_{1,n}/P_{2,n}$
and $|P_{2,n}|=1$.
This is possible, because $R_n\in\Q[[x]]$.
By passing to a subsequence without changing notation, we assume
\begin{align*}
\lim_{n\to\infty}|\wt P_1-P_{1,n}|= &  0,\\
\lim_{n\to\infty}|\wt P_2-P_{2,n}|= & 0
\end{align*}
for some $\wt P_1,\wt P_2\in\cP_L$.
By explicitly computing the coefficients of $P_{1,n}/P_{2,n}$ in terms of the coefficients of $P_{1,n}$ and $P_{2,n}$, it follows that for each $k\ge1$ the first $k$ coefficients of $P_{1,n}/P_{2,n}$ depend only on the first $k$ coefficients of $P_{1,n}$ and $P_{2,n}$. Thus,
\[
\lim_{n\to\infty}|\wt P_1/\wt P_2- R_n| =0.
\]

Let $\d\in(0,1-\l_0)$ be sufficiently small, such that none of the $R_n$ has a pole in $\overline B(\l_0,\d)$.
The underlying metric space for the $B(\cdot)$ notation is $\C$ here and later in this proof. 
We assume further, as we may, that $\wt P_2$ has no zeros on the circle $\partial B(\l_0,\d)$.
By uniform boundedness of the coefficients, we have
\begin{align*}
\lim_{n\to\infty} P_{1,n}(z)=&\wt P_1(z),\\
\lim_{n\to\infty} P_{2,n}(z)=&\wt P_2(z)
\end{align*}
uniformly on compact subsets of $B(0,1)$, in particular for  $z\in\overline B(\l_0,\d)$.

By discarding a finite number of elements of the sequences $P_{1,n}$, $P_{2,n}$ if necessary,
we may assume that there is a number $\d_2>0$ such that $|P_{2,n}(z)|>\d_2$ for
all $n$ and all $z\in\partial B(\l_0,\d)$.
Therefore, we have
\[
\lim_{n\to\infty} R_{n}(z)=\frac{\wt P_1(z)}{\wt P_2(z)}
\]
uniformly for $z\in\partial B(\l_0,\d)$.
By the maximum modulus principle, the convergence also holds uniformly on $\overline B(\l_0,\d)$,
and this completes the proof.
\end{proof}

When the curves have singularities near $\l_0$ we shall need to consider certain planar self-similar measures. Given $0\ne z\in B(0,1)(\subset\C)$ and $w\in\C$, we write $\mu_{z,w}$ for the distributions of $\sum_{k=0}^\infty T_{\xi_k}(1,w)z^k$ and $h(z,w)$ for
\[
\lim_n\frac{1}{n}H\Big(\sum_{k=0}^{n-1}T_{\xi_k}(1,w)z^k\Big).
\]

\begin{lem}\label{<= h(gamma)}
Let $R\in \cR_L$ and $0\ne z_0\in B(0,1)$ be given, and suppose that $z_0$ is not a pole of $R$. Then $h(z_0,R(z_0))\le h(\g)$, where $\g\in\Gamma$ is the non-degenerate curve corresponding to $R$.
\end{lem}

\begin{proof}
Let $n\ge1$ and $Q\in\cQ^{(n)}$ be such that $Q(\l,1,\tau)=0$ for all $(\l,\tau)\in\g$. Then the zero set of the meromorphic function $z\to Q(z,1,R(z))$ has an accumulation point, and so it is identically $0$. In particular $Q(z_0,1,R(z_0))=0$.
Since $Q$ is arbitrary subject to the condition that $Q(\l,1,\tau)=0$ for all $(\l,\tau)\in\g$, the inequality $h(z_0,R(z_0))\le h(\g)$ follows.
\end{proof}

The following semicontinuity result will be used in this section when dealing with the self-similar measures $\mu_{z,w}$. We state it more generally, since in Section \ref{lb ent ord M} it will be applied when dealing with certain higher dimensional self-affine measures. 

\begin{lem}\label{lm:sa-semi-cont}
Let $K\in\Z_{>0}$, and write $\cT$ for the set of $K\times K$ real matrices whose eigenvalues are all equal in modulus to the same value in $(0,1)$. Given $\Theta\in\cT$ and $(v_1,\ldots,v_m)=v\in(\R^K)^m$, denote by $\nu_{\Theta,v}$ the self-affine measure corresponding to the IFS
\[
\Psi_{\Theta,v}=\{\psi_{\Theta,v;j}(x)=\Theta x+v_j : j=1,\ldots,m\}
\]
and the probability vector $(p_1,\ldots,p_m)$. Then the map which takes $(\Theta,v)\in\cT\times(\R^K)^m$ to $\dim\nu_{\Theta,v}$ is lower semicontinuous.
\end{lem}

\begin{rmk*}
Note that the eigenvalues of elements in $\cT$ are allowed to be complex.
\end{rmk*}

\begin{proof}
Recall that we write $\Omega$ for $\{1,\ldots,m\}^{\Z_{\ge0}}$ and $\b$ for the Bernoulli measure $p^{\Z_{\ge0}}$.
Given $\Theta\in\cT$ write $r_\Theta\in(0,1)$ for the modulus of the eigenvalues of $\Theta$. For $v\in (\R^K)^m$ denote by $\Pi_{\Theta,v}$ the coding map corresponding to the IFS $\Psi_{\Theta,v}$. That is for every $\omega\in \Omega$,
\[
\Pi_{\Theta,v}\omega=\lim_n\psi_{\Theta,v;\omega_0}\circ\ldots\circ\psi_{\Theta,v;\omega_n}(0).
\]
Write $\cC$ for the partition of $\Omega$ according to the first coordinate, and $\cB$ for the Borel $\sigma$-algebra of $\R^K$.

For every unit vector $x\in\R^K$,
\[
|\Theta^nx|=r_\Theta^{n+o(n)}\quad \text{ as }n\rightarrow\infty.
\]
Thus from \cite{feng2019dimension}*{Theorem 1.3},
\begin{equation}\label{dim formula}
\dim\nu_{\Theta,v} = \frac{1}{\log r_\Theta^{-1}}(H(p) - H_\b(\cC|\Pi_{\Theta,v}^{-1}(\cB))),
\end{equation}
where $H_\b(\cC|\Pi_{\Theta,v}^{-1}(\cB))$ is the conditional entropy of $\cC$ given $\Pi_{\Theta,v}^{-1}(\cB)$ with respect to $\b$.
Let $F:\cT\times(\R^K)^m\rightarrow\R$ be with $F(\Theta,v) = H_\b(\cC|\Pi_{\Theta,v}^{-1}(\cB))$, and fix some $(\Theta_0,v_0)\in\cT\times(\R^K)^m$.
Since $(\Theta_0,v_0)$ is arbitrary, it suffices to show that $F$ is upper semicontinuous at $(\Theta_0,v_0)$, and the lemma will
follow by \eqref{dim formula}.

For $n\ge1$ let $\cD_n$ be the dyadic partition of $\R^K$ into cubes of side length $2^{-n}$, and write $\s(\cD_n)$ for the $\s$-algebra generated $\cD_n$. By translating all of these partitions by the same vector if necessary, we may assume that $\nu_{\Theta_0,v_0}(\partial D)=0$ for all $n\ge1$ and $D\in\cD_n$. It follows easily from this that the map which takes $(\Theta,v)\in\cT\times(\R^K)^m$ to $F_n(\Theta,v):=H_\b(\cC|\Pi_{\Theta,v}^{-1}(\s(\cD_n)))$ is continuous at $(\Theta_0,v_0)$ for each $n\ge1$. Moreover we clearly have $F_1\ge F_2 \ge\ldots$, and $F_n\overset{n}{\rightarrow}F$ pointwise  by the increasing martingale theorem. This shows that $F$ is upper semicontinuous at $(\Theta_0,v_0)$, which completes the proof.
\end{proof}

The next lemma will be used in the proof of Proposition \ref{pr:no-bad-curves} when the curves are non-degenerate but have singularities arbitrarily close to $\l_0$.

\begin{lem}\label{lm:singularities}
Let $(\l_0,\tau_0)\in(0,1)\times\R$.
Let $R_1,R_2\ldots\in\cR_L$, and let $\gamma_1,\gamma_2,\ldots\in\Gamma$ be the corresponding non-degenerate
curves.
Assume that
\[
\dist(\gamma_n,(\l_0,\tau_0))\to 0.
\]
Assume further that for every neighborhood $V$ of $\l_0$ in $\C$ there exists an arbitrarily large $n$ such that $V$ contains a pole of $R_n$.
Then there is an interval $[a,b]\subset\R$ of positive length such that
\[
\dim\mu_{\l_0,\tau}\le \limsup_{n\to \infty} \frac{h(\gamma_n)}{\log\l_0^{-1}}
\]
for all $\tau\in[a,b]$.
\end{lem}

\begin{proof}
By moving to a subsequence without changing the notation, we may assume
that there exist sequences $\{\l_k\}_{k\ge1}\subset(0,1)$ and $\{\d_k\}_{k\ge1}\subset(0,1)$, such that $\l_k\overset{k}{\to}\l_0$, $R_k(\l_k)\overset{k}{\to}\tau_0$, $\d_k\overset{k}{\to}0$ and,
\[
\max\{|z|:z\in E_k\}=|\tau_0|+2\text{ for each }k\ge1,
\]
where $E_k:=R_{k}(\overline{B}(\l_k,\d_{k}))$.
The underlying metric space for the $B(\cdot)$ notation is $\C$ here and later in this proof.

Denote by $\mathcal{K}$ the collection of all nonempty compact subsets
of $\overline{B}(0,|\tau_0|+2)$, and let $d_{H}$ be the Hausdorff metric on $\mathcal{K}$.
Then $(\mathcal{K},d_{H})$ is a compact metric space, and $E_{k}$ is a member of $\mathcal{K}$ for each $k\ge1$. Thus, by moving to a
subsequence without changing the notation, we may assume that there
exists $E\in\mathcal{K}$ with $d_{H}(E,E_{k})\overset{k}{\rightarrow}0$.

We show that
\begin{equation}
\dim\mu_{\lambda_0,w}\le\limsup_{n\to \infty} \frac{h(\gamma_n)}{\log\l_0^{-1}}\qquad\text{ for all }w\in E.\label{eq:ub on dim for w in E}
\end{equation}
Given $w\in E$, there exists a sequence $\{w_{k}\}_{k\ge1}\subset\C$
with $w_{k}\overset{k}{\rightarrow}w$ and $w_{k}\in E_{k}$ for all
$k\ge1$.
For each $k\ge1$, let $z_{k}\in\overline{B}(\lambda_k,\delta_{k})$
be with $R_{k}(z_{k})=w_{k}$.
Since $\l_k\overset{k}{\to}\l_0$, $\delta_{k}\overset{k}{\rightarrow}0$
and $0<\l_0<1$, we have $z_k\overset{k}{\to}\l_0$ and we may assume $0<|z_{k}|<1$.
Additionally, from Lemma \ref{<= h(gamma)} we get,
\[
\dim\mu_{z_k,w_k}\le\frac{h(z_k,w_k)}{\log|z_k|^{-1}}\le \frac{h(\gamma_k)}{\log|z_k|^{-1}}.
\]
Now from this, $(z_{k},w_{k})\overset{k}{\rightarrow}(\lambda_0,w)$
and Lemma \ref{lm:sa-semi-cont}, we obtain \eqref{eq:ub on dim for w in E}.

For every $k\ge1$, the set $E_{k}$ is connected and intersects the
circles $\partial B(0,|\tau_0|+1)$ and $\partial B(0,|\tau_0|+2)$.
It is easy to
see that these properties are preserved under convergence with respect
to $d_{H}$.
Thus $E$ is connected and there exist $u_{1},u_{2}\in E$
with $|u_{1}|=|\tau_0|+1$ and $|u_{2}|=|\tau_0|+2$.

Since $u_{1}\ne u_{2}$, there
exists an $\R$-linear map $S:\C\to\R$
with $Su_{1}\ne Su_{2}$ and $Sx=x$ for $x\in\R$.
Write
$[a,b]=S(E)$, then $[a,b]\subset\R$ is a nontrivial interval.
Since $S$ is a Lipschitz map it does not increase dimension.
Thus
by \eqref{eq:ub on dim for w in E}, it follows that for every $w\in E$,
\[
\dim S\mu_{\lambda_0,w}\le\dim	\mu_{\lambda_0,w}
\le\limsup_{n\to \infty} \frac{h(\gamma_n)}{\log\l_0^{-1}}.
\]
Recall the numbers $a_1,\ldots,a_m,b_1,\ldots,b_m$ from Section \ref{sc:general-setting}. Since $\l_0$ and these numbers are all real, since
$S$ is $\R$-linear and since $Sx=x$ for $x\in\R$, it is easy to verify that $S\mu_{\lambda_0,w}=\mu_{\lambda_0,Sw}$
for every $w\in\C$.
This completes the proof of the lemma.
\end{proof}

We will use the next lemma in combination with the previous one and also in the case when
the curves are degenerate.

\begin{lem}\label{lm:degenerate}
Let $\l_0\in(0,1)$, let $h_0\in(0,\log\l^{-1}_{0})$, and let $[a,b]\subset\R$ be an interval of positive length.
Assume
\[
\dim\mu_{\l_0,\tau}\le\frac{h_0}{\log\l_0^{-1}}
\]
for all $\tau\in[a,b]$.
Then
\[
h(\{\l_0\}\times\R)\le h_0.
\]
\end{lem}

\begin{proof}
Suppose first that $\l_0$ is algebraic.
Then let $\tau\in(a,b)\backslash \Q(\l_0)$ be an algebraic number.
It follows from \cite{Hoc-self-similar}*{Theorem 1.3} (see \cite{BV-entropy}*{Section 3.4} for more details)
that
\[
\dim\mu_{\l_0,\tau}=\min\Big\{1,\frac{h(\l_0,\tau)}{\log \l_0^{-1}}\Big\},
\]
hence $h(\l_0,\tau)\le h_0$.
Furthermore, $h(\{\l_0\}\times\R)=h(\l_0,\tau)$, since $\tau\notin \Q(\l)$.
This proves the lemma in the case when $\l_0$ is algebraic.

Now suppose that $\l_0$ is transcendental.
Then Theorem \ref{thm:ver conj for int trans} implies that
\[
\dim\mu_{\l_0,\tau}=\min\Big\{1,\frac{H(p)}{\log\l_0^{-1}}\Big\}
\]
provided $\tau$ is rational and is chosen in such a way that $T_i(1,\tau)\neq T_j(1,\tau)$ for $i\neq j$.
Since there are only finitely many choices of $\tau$ for which this fails, we conclude
$H(p)\le h_0$.
Now the lemma follows from $h(\{\l_0\}\times\R)\le H(p)$.
\end{proof}

\begin{proof}[Proof of Proposition \ref{pr:no-bad-curves}]
We first consider the case when for each $\e>0$, there is $\l_\e$ with $|\l_\e-\l_0|<\e$ such that
$h(\{\l_\e\}\times\R)\le h_0$.
Then $\dim\mu_{\l_\e,\tau}\le h_0/\log\l_\e^{-1}$ for each $\tau\in\R$. 
By lower semi-continuity of dimension, it follows
that $\dim\mu_{\l_0,\tau}\le h_0/\log\l_0^{-1}$ for each $\tau\in\R$.
Now we can conclude by Lemma \ref{lm:degenerate}.

Next we consider the case when for each $n$, there is $R_n\in\cR_L$ and a corresponding non-degenerate
curve $\gamma_n\in\Gamma$ that intersects the $1/n$-neighborhood of $(\l_0,\tau_0)$ such that $h(\gamma_n)\le h_0$.
If there are infinitely many values of $n$ such that $R_n\in\Q[[x]]$, then we pass to a subsequence without changing
notation containing curves
only with this property.

Note that from $|R_n^{-1}|=|R_n|^{-1}$ and since $\Q[[X]]$ is the closed unit ball of $\Q((X))$ with respect to $|\cdot|$, it follows that $R_n^{-1}\in\Q[[x]]$ whenever $R_n\notin\Q[[x]]$.
Thus, if there are only finitely many values of $n$ such that $R_n\in\Q[[x]]$, then we replace $\tau_0$ by $\tau_0^{-1}$ (recall that we assume $\tau_0\ne0$),
$R_n$ by $R_n^{-1}$ and we also exchange the integers $a_j$ and $b_j$ in the definition of the IFS.
It is not difficult to see that this operation preserves the entropy of curves.
After this we pass to a suitable subsequence as above.

Now the conditions of either Lemma \ref{lm:no-singularities} or Lemma \ref{lm:singularities} are satisfied.
In the former case, we get a non-degenerate limit curve $\gamma$ corresponding to some $R\in\cR_L$
that passes through $(\l_0,\tau_0)$.
Now it follows by Corollary \ref{cr:h-semi-cont} that $h(\g)=h(R)\le h_0$.

If we are in the case of Lemma \ref{lm:singularities}, we can conclude by Lemma \ref{lm:degenerate} as above.
\end{proof}

\begin{proof}[Proof of Corollary \ref{cr:no-bad-curves}]
When $\tau_0\ne0$ the corollary follows directly from Proposition \ref{pr:no-bad-curves}.

Suppose that $\tau_0=0$. Then since $(\l_0,\tau_0)$ has no exact overlaps, it follows by Theorem \ref{thm:ver conj for int trans} in case $\l_0$ is transcendental, or by Hochman's result on systems with algebraic parameters in case $\l_0$ is algebraic, that
\[
\dim\mu_{\l_0,\tau_0}=\min\{1,H(p)/\log\l^{-1}_0\}.
\]
Now the corollary follows by lower semi-continuity of dimension.
\end{proof}

\section{Proofs of Theorem \ref{th:BV2-general} and Corollary \ref{cr:Mike-general}}\label{sc:prf of apx thm}

Some of the arguments in this section are based on the paper \cite{BV-transcendent}.
These are discussed in Section \ref{sc:gen of BV2 int trans} of the appendix in the simpler setting of
homogeneous IFS's with rational translations.
The reader not familiar with \cite{BV-transcendent} may find it helpful to read that part of the appendix before
this section, but this section can also be read independently.

Recall the notation $\xi_0,\xi_1,\ldots$, $A_{\l,\tau}^{(n)}$ and $\mu_{\l,\tau}^{(n)}$ from Section \ref{sc:general-setting}.
We write $s(\l)$ for $\min\{1,H(p)/\log\lambda^{-1}\}$, and given $x,y\in\R$ write $\Vert(x,y)\Vert_\infty$ in place of $\max\{|x|,|y|\}$.

Given $r>0$ and a random variable $A$ with distribution $\nu$, recall the notation
\[
H(A;r)=\int_{0}^1H(\lfloor r^{-1}A+t\rfloor)dt,
\]
and that we write $H(\nu;r)$ in place of $H(A;r)$.
We also write
\[
H(\nu;r_1|r_2)=H(\nu;r_1)-H(\nu;r_2),
\]
and call this quantity the entropy of $\nu$ between the scales $r_1,r_2\in\R_{>0}$.
Given a set $I\subset\R_{>0}$, we write $\mu^I_{\l,\tau}$ for the law of the random variable
\[
\sum_{j:\l^j\in I}T_{\xi_j}(1,\tau)\l^j.
\]
With this notation, we have $\mu_{\l,\tau}=\mu^{(0,1]}_{\l,\tau}$ and $\mu^{(k)}_{\l,\tau}=\mu^{(\l^k,1]}_{\l,\tau}$.

The proof of Theorem \ref{th:BV2-general} is a proof by contradiction, and it is based
on the following strategy borrowed from \cite{BV-transcendent}.
Using the indirect hypothesis, we show that there is a sequence of integers $n_1,\ldots,n_N$ such that
$\mu^{(n_j)}_{\l,\tau}$ has some significant amount of entropy between certain suitably chosen scales for each $j$.
We then use the scaling identity
\begin{equation}\label{eq:scaling}
H(\mu^{I}_{\l,\tau};r_1|r_2)=H(\mu^{\l^kI}_{\l,\tau};\l^kr_1|\l^kr_2)
\end{equation}
to produce disjoint intervals $I_j$ such that the measures $\mu^{I_j}_{\l,\tau}$ have some significant amount of entropy
between a common scale range.
Then we use the identity
\begin{equation}\label{eq:disjoint-convolution}
\mu^{I_1\dot\cup\cdots\dot\cup I_N}_{\l,\tau}=\mu^{I_1}_{\l,\tau}*\ldots*\mu^{I_N}_{\l,\tau}
\end{equation}
and a general result about how the entropy of measures grows under convolution.
To verify the conditions of this result about entropy growth, we need to use the assumption $\dim\mu_{\l,\tau}<1$, and
the conclusion will be that \eqref{eq:disjoint-convolution} has more entropy between certain scales
than any probability measure can have. We reach the desired contradiction.

The scaling and entropy increase by convolution argument is encapsulated in the following result.

\begin{prp}
\label{prop:scale-and-convolve}
For all $\lambda\in(0,1)$ and $\a>0$, there exists $C>1$ such
that the following holds for all $\tau\in\R$.
Let $N\ge1$, $\{n_{j}\}_{j=1}^{N}\subset\Z_{>0}$
and $\{K_{j}\}_{j=1}^{N}\subset[10,\infty)$ be given.
Suppose that
$\lambda^{-n_{1}}\ge\max\{2,\lambda^{-2}\}$ and,
\begin{enumerate}
\item $n_{j+1}\ge K_{j}n_{j}$ for all $1\le j<N$;
\item $H(\mu_{\lambda,\tau};r|2r)\le1-\alpha$ for all $r>0$;
\item $H(\mu_{\lambda,\tau}^{(n_{j})};\lambda^{K_{j}n_{j}}|\lambda^{10n_{j}})\ge\alpha n_{j}$
for all $1\le j\le N$;
\item $n_{j}\ge C(\log K_{j})^{2}$ for all $1\le j\le N$.
\end{enumerate}
Then,
\[
\sum_{j=1}^{N}\frac{1}{\log K_{j}\log\log K_{j}}\le C\left(1+\frac{1}{n_{1}}\sum_{j=1}^{N}\log K_{j}\right).
\]
\end{prp}

This result is proved in \cite{BV-transcendent}*{Proposition 30} in the setting of Bernoulli convolutions based on the ideas
presented above.
The proof uses only the properties \eqref{eq:scaling} and \eqref{eq:disjoint-convolution} of the measures $\mu_{\l,\tau}^{(n)}$
along with general facts that are valid for all probability measures.
For this reason, we do not repeat the proof here.

To prove Theorem \ref{th:BV2-general}, we will show that under an indirect hypothesis, the parameters
can be chosen in Proposition \ref{prop:scale-and-convolve} in such a way that the hypotheses
of the proposition hold, and the conclusion leads to a contradiction.

We begin with condition $(2)$ of the proposition.
This will be satisfied using the assumption $\dim\mu_{\l,\tau}<1$ and the following result.

\begin{lem}
\label{lem:bd of ent of single dig2}
Suppose that $\dim\mu_{\lambda,\tau}<1$.
Then there exists
$\alpha>0$ (depending on $\l$, $\tau$, the parameters in the IFS \eqref{eq:IFS} and $p_1,\ldots,p_m$) such that,
\[
H(\mu_{\lambda,\tau};r|2r)<1-\alpha\text{ for all }r>0.
\]
\end{lem}

This is proved in  \cite{BV-transcendent}*{Lemma 13} in the setting of Bernoulli convolutions.
Again, the proof depends only on properties \eqref{eq:scaling} and \eqref{eq:disjoint-convolution}
and we do not repeat it. 
The dependence of $\a$ on $\tau$, the parameters in \eqref{eq:IFS} and $p_1,\ldots,p_m$ is only through
the difference $1-\dim\mu_{\l,\tau}$, otherwise these parameters play no role in the proof.

\subsection{A first approximation for parameters with small entropy}\label{sc:common-root-Q}

We move on to consider the conditions other than $(2)$ in Proposition \ref{prop:scale-and-convolve}.
We begin by studying the situation when $H(\mu_{\l,\tau}^{(n)};r)$ is significantly smaller than
$ns(\l)\log\l^{-1}$ for some $n$ and a suitable scale $r$.
In the analogous situation for Bernoulli convolutions, there exists an algebraic approximation of $\lambda$
with lots of exact overlaps.
As we discussed in Section \ref{sc:outline}, the geometry is more complicated in the setting of the IFS \eqref{eq:IFS}.
We describe it in detail in the next result.

Let $n\in\Z_{>0}$.
Write $\cX^{(n)}$ for the set of parameters $(\l,\tau)\in(0,1)\times \R$ such that 
there are
\[
Q=P_1(X)Y_1+P_2(X)Y_2,\;\wt Q=\wt P_1(X)Y_1+\wt P_2(X)Y_2\in\cQ^{(n)}
\]
with $Q(\l,1,\tau)=\wt Q(\l,1,\tau)=0$, $P_2(\l)\neq 0$ and $P_1\wt P_2-P_2\wt P_1\neq 0$.
In other words, $\cX^{(n)}$ is the set of parameters $(\l,\tau)$ with exact overlaps such that the family of polynomials in $\cQ^{(n)}$
that vanish on $(\l,\tau)$ does not vanish along a common curve containing $(\l,\tau)$.
Indeed, since $P_2(\l)\neq 0$, $Q$ does not vanish on the degenerate curve $\{\l\}\times \R$, and 
since $P_1\wt P_2-P_2\wt P_1\neq 0$, $\wt Q$ does not vanish along the non-degenerate curve determined by $Q$.

We note two immediate consequences of the definition of $\cX^{(n)}$.
If $(\l,\tau)\in\cX^{(n)}$, then $\l$ is a root of the nonzero polynomial $P_1\wt P_2-P_2\wt P_1\in\cP_{2L^{2}n}^{(2n)}$
appearing in the definition, where $L$ is as defined in  \eqref{eq:def of L}.
Furthermore, we have $\tau=-P_1(\l)/P_2(\l)$.
In particular, both $\l$ and $\tau$ are algebraic of degree at most $2n$, and we have control over their heights.

The purpose of this section is to prove the following result.

\begin{prp}
\label{prp:approx by alg ifs when ent is small}
For every $\e>0$
there exists $C=C(L,\e)>1$ such that for every $\lambda,\tau\in\R$
with $\e\le\lambda\le1-\e$ and $|\tau|\le\e^{-1}$,
there exists $N=N(L,\e,C,\lambda,\tau)\ge1$ such that the following
holds.
Let $n\ge N$ and $0<r\le n^{-Cn\log n}$, and suppose that
\[
H(\mu_{\lambda,\tau}^{(n)};r)< n H(p).
\]

Then at least one of the following two alternatives hold.
\begin{enumerate}
\item There is $(\eta,\s)\in\cX^{(n)}$ such that $|\lambda-\eta|,|\tau-\sigma|\le r^{1/(C\log n)}$
and $H(\mu_{\eta,\sigma}^{(n)})\le H(\mu_{\lambda,\tau}^{(n)};r)$.
\item There is $\gamma\in\Gamma$ such that $\dist(\gamma,(\l,\tau))\le r^{1/(C\log n)}$
and $H(A_\gamma^{(n)})\le H(\mu_{\lambda,\tau}^{(n)};r)$.
\end{enumerate}
\end{prp}

When we apply this result, we will assume that $(\l,\tau)$ has no exact overlaps and
\[
\frac{1}{n}H(\mu_{\l,\tau}^{(n)};r)<\min\{\log\l^{-1},H(p)\}-\e,
\]
hence the second alternative will be impossible by Corollary \ref{cr:no-bad-curves}.

For the proof of this proposition we first need a good understanding of the geometry of the set
where a collection of polynomials in $\cQ^{(n)}$ attain small values simultaneously.
We study this in the next proposition.

\begin{prp}
\label{prop:common zero}
For every $\e>0$ there
exists $C=C(L,\e)>1$, such that for all $n\ge N(L,\e,C)\ge1$
the following holds.
Let $q\ge1$,
\[
\{P_{i,1}(X)Y_{1}+P_{i,2}(X)Y_{2}=Q_{i}(X,Y_{1},Y_{2})\}_{i=1}^{q}=\mathcal{A}\subset\mathcal{Q}^{(n)}\setminus\{0\},
\]
$(\lambda,\tau)\in(\e,1-\e)\times(-\e^{-1},\e^{-1})$, and $0<r<n^{-Cn\log n}$ be given.
Suppose
\[
|Q_{i}(\lambda,1,\tau)|\le r\qquad\text{ for each }1\le i\le q.
\]

Then
at least one of the following three alternatives holds.
\begin{enumerate}
\item There exists $(\eta,\s)\in\cX^{(n)}$ such that,
\begin{itemize}
\item $|\lambda-\eta|,|\tau-\sigma|\le r^{1/(C\log n)}$ and
\item $Q_{i}(\eta,1,\sigma)=0$ for each $1\le i\le q$.
\end{itemize}
\item There exists $\eta\in(0,1)$ such that,
\begin{itemize}
\item $|\lambda-\eta|\le r^{1/(C\log n)}$ and
\item $P_{i,j}(\eta)=0$ for all $1\le i\le q$ and $j=1,2$.
\end{itemize}
\item We have,
\begin{itemize}
\item $P_{i,1}(X)P_{j,2}(X)=P_{i,2}(X)P_{j,1}(X)$ for each $1\le i,j\le q$;
\item $|P_{i,2}(\lambda)|\ge r^{1/2}$ for some $1\le i\le q$.
\end{itemize}
\end{enumerate}
\end{prp}

In the first alternative, all the polynomials vanish at a common point in $\cX^{(n)}$ near $(\l,\tau)$.
In the second, all of them vanish along a degenerate curve that passes near $(\l,\tau)$.
In the third, it will be shown that they all vanish along a non-degenerate curve that passes near $(\l,\tau)$.

\begin{proof}
Let $\e>0$, let $C>1$ be large with respect to
$L$ and $\e$, and let $n\ge1$ be large with respect to $C$.
Let $q$, $\mathcal{A}$, $\lambda$, $\tau$ and $r$ be as in the
statement of the proposition.

Let $\Q(X)$ be the field of
rational functions over $\Q$. Set $I=\{1,\ldots,q\}\times\{1,2\}$,
and denote by $d$ the rank of the matrix,
\[
(P_{i,j}(X))_{(i,j)\in I}\in\mathrm{Mat}_{q,2}(\Q(X)).
\]
Since $Q_{i}\ne0$ for each $1\le i\le q$, we have $d=1\text{ or }2$.

We consider the following three not mutually exclusive cases.
\begin{itemize}
\item[Case 1.]
$d=2$,
\item[Case 2.]
$|P_{i,2}(\l)|< r^{1/2}$ for all $1\le i\le q$,
\item[Case 3.]
$d=1$ and $|P_{i,2}(\l)|\ge r^{1/2}$ for some $1\le i\le q$.
\end{itemize}
In Case 3, the third alternative of the conclusion is immediate.
In Case 2, we show that the second alternative holds.
In Case 1, we show that at least one of the first two alternatives hold.

We write $\widetilde{\cP}$
for $\mathcal{P}_{2L^{2}n}^{(2n)}$ for the reminder of the proof.
It is easy to see that all minors of the matrix $(P_{i,j}(X))_{(i,j)\in I}$
belong to $\wt\cP$.

Our first objective is to show that in both Cases 1 and 2, there is a non-zero
$E\in\wt\cP$ such that $|E(\l)|\le r^{1/2}$.
In Case 1, we suppose without loss of generality that
\[
E(X):=\det\left((P_{i,j}(X))_{i,j=1}^{2}\right)\ne0.
\]
Write $M$ for $(P_{i,j}(\lambda))_{i,j=1}^{2}$ and $v$ for the
column vector $(1,\tau)^{\mathrm{T}}$. Since $\e<\lambda<1-\e$,
the largest singular value of $M$ is $O_{L,\e}(1)$. Since
$|v|\ge1$ and $|Mv|=O(r)$, the smallest singular value of $M$ is
$O(r)$.
It follows that
\begin{equation}\label{eq:det-bound}
|E(\l)|=|\det(M)|\le Cr\le r^{1/2}
\end{equation}
provided $C$ is sufficiently large in terms of $\e$ and $L$ and $r$ is sufficiently small in terms of $C$.
These follow from our assumptions, as $r<n^{-n}$.

In Case 2, we can simply take $E=P_{1,2}$ if $P_{1,2}\neq 0$.
If $P_{1,2}=0$, then $P_{1,1}\neq 0$ by $Q_1\neq 0$, and we take $E=P_{1,1}$.
Now we have
\[
|E(\l)|=|P_{1,1}(\l)|=|Q_1(\l,1,\tau)|\le r,
\]
as required.

Recall that $E\in\wt\cP$.
By Lemma \ref{lem:close root single poly} and by assuming that $C$
and $n$ are large enough with respect to $L$ and $\e$, it
follows that there exists $\eta\in\C$ such that $|\lambda-\eta|\le r^{C^{-1/2}/\log n}$
and $E(\eta)=0$.
We show that $\eta\in(0,1)$.
Since $\lambda\in\R$, we also have $|\lambda-\overline{\eta}|\le r^{C^{-1/2}/\log n}$,
and so 
\[
|\eta-\overline{\eta}|\le2r^{C^{-1/2}/\log n}\le2n^{-C^{1/2}n}.
\]
Clearly $E(\overline{\eta})=0$, hence by Lemma \ref{lem:dist between distinct roots}
and by assuming that $C$ is large enough, it follows that $\eta=\overline{\eta}$
or equivalently that $\eta\in\R$. Since $\e\le\lambda\le1-\e$,
we may assume that $\eta\in(0,1)$.

Next we show that for all $P\in\wt\cP$ with $|P(\l)|<Cr^{1/2}$, we have $P(\eta)=0$.
For every $P\in\widetilde{\cP}$ and $\xi\in\R$ with $|\xi|\le1-\e/2$
we have $|P'(\xi)|=O_{L,\e}(n^{2})$. From this, from $|\lambda-\eta|\le r^{C^{-1/2}/\log n}$,
and by assuming that $C$ and $n$ are large enough, it follows that
\begin{equation}
|P(\eta)-P(\lambda)|=O_{L,\e}(n^{2}\cdot r^{C^{-1/2}/\log n})\le r^{1/(2C^{1/2}\log n)}\le n^{-C^{1/2}n/2}.\label{eq:dist between P at two points}
\end{equation}
From this, from $|P(\l)|<Cr^{1/2}$, and by assuming
that $n$ is large enough,
\[
|P(\eta)|\le|P(\lambda)|+|P(\eta)-P(\lambda)|\le2n^{-C^{1/2}n/2}.
\]
Hence from Lemma \ref{lem:lb on value of poly at root}, and by assuming
that $C$ and $n$ are large enough, we get $P(\eta)=0$.

At this point, the argument separates in the two remaining cases.
In Case 2, we have $|P_{i,2}(\l)|<r^{1/2}$ by assumption for all $i$, and we get
\[
|P_{i,1}(\l)|\le |Q_i(\l,1,\tau)|+|\tau P_{i,2}(\l)|\le r+\e^{-1}r^{1/2}\le Cr^{1/2},
\]
provided $C$ is sufficiently large depending on $\e$, as we assumed.
Now $P_{i,j}(\eta)=0$ follows for all $i$ and $j$, and we see that
the second alternative in the conclusion holds.

Now we assume that we are in Case 1.
If $P_{i,j}(\eta)=0$ for all $i$ and $j$, then the second alternative holds, so we assume this is not the case.
Without loss of generality, we assume $(P_{1,1}(\eta),P_{1,2}(\eta))\neq0$.
We show that $|P_{1,2}(\eta)|>n^{-C^{1/2}n/8}$.
Suppose to the contrary that this is not the case.
Then by Lemma \ref{lem:lb on value of poly at root}, and by assuming $C$ is sufficiently large, we get $P_{1,2}(\eta)=0$.
By \eqref{eq:dist between P at two points}, we have
\begin{align*}
|P_{1,1}(\eta)|\le & |P_{1,1}(\eta)-P_{1,1}(\l)|+|Q_1(\l,1,\tau)|+\e^{-1}|P_{1,2}(\eta)-P_{1,2}(\l)| \\
\le & C n^{-C^{1/2}n/2}.
\end{align*}
Using Lemma \ref{lem:lb on value of poly at root} again, we get $P_{1,1}(\eta)=0$, a contradiction.

We set $\sigma=-P_{1,1}(\eta)/P_{1,2}(\eta)$.
Then $Q_{1}(\eta,1,\sigma)=0$, and we show that $Q_i(\eta,1,\sigma)=0$ for
all $1\le i\le q$.
Since $d=2$, this will also imply $(\eta,\s)\in\cX^{(n)}$.
By the argument leading up to \eqref{eq:det-bound}, we have
\[
|P_{1,1}(\l)P_{i,2}(\l)-P_{1,2}(\l)P_{i,1}(\l)|\le r^{1/2}.
\]
Since this polynomial is in $\wt\cP$, we have
\[
P_{1,1}(\eta)P_{i,2}(\eta)-P_{1,2}(\eta)P_{i,1}(\eta)=0.
\]
Thus,
\[
Q_{i}(\eta,1,\sigma)= P_{i,1}(\eta)-\frac{P_{1,1}(\eta)}{P_{1,2}(\eta)}P_{i,2}(\eta) =0.
\]

In order to show that the first alternative of the proposition holds
and to complete the proof, it remains to estimate $|\tau-\sigma|$.
Since $|P_{1,2}(\eta)|>n^{-C^{1/2}n/8}$, we have
$|P_{1,2}(\l)|\ge n^{-C^{1/2}n/8}/2 \ge r^{1/2}$ by \eqref{eq:dist between P at two points}. Hence,
\[
\left|\tau+\frac{P_{1,1}(\lambda)}{P_{1,2}(\lambda)}\right|=\left|\frac{Q_{1}(\lambda,1,\tau)}{P_{1,2}(\lambda)}\right|\le r^{1/2}.
\]
Additionally by \eqref{eq:dist between P at two points},
\begin{align*}
\left|\frac{P_{1,1}(\eta)}{P_{1,2}(\eta)}-\frac{P_{1,1}(\lambda)}{P_{1,2}(\lambda)}\right|  \le & \frac{O_{L,\e}\left(|P_{1,1}(\eta)-P_{1,1}(\lambda)|+|P_{1,2}(\eta)-P_{1,2}(\lambda)|\right)}{|P_{1,2}(\eta)P_{1,2}(\lambda)|}\\
  = & O_{L,\e}(n^{C^{1/2}n/4}\cdot r^{1/(2C^{1/2}\log n)}).
\end{align*}
Hence,
\begin{align*}
|\tau-\sigma|  \le & \left|\tau+\frac{P_{1,1}(\lambda)}{P_{1,2}(\lambda)}\right|+\left|\frac{P_{1,1}(\eta)}{P_{1,2}(\eta)}-\frac{P_{1,1}(\lambda)}{P_{1,2}(\lambda)}\right|\\
  = & O_{L,\e}(n^{C^{1/2}n/4}\cdot r^{1/(2C^{1/2}\log n)}).
\end{align*}
Thus by recalling that $r<n^{-Cn\log n}$, and by assuming that $C$ and $n$ are large enough, we get $|\tau-\sigma|\le r^{1/(C\log n)}$.
This completes the proof of the proposition. 
\end{proof}

\begin{proof}[Proof of Proposition \ref{prp:approx by alg ifs when ent is small}]
Let $\e>0$ and let $C>1$ be large with respect to $L$ and $\e$.
Let $(\lambda,\tau)\in(\e,1-\e)\times(-\e^{-1},\e^{-1})$.
Let $n\ge1$ be large with respect to $L,\e,C,\l$ and $\tau$, and let $0<r<n^{-Cn\log n}$ be such that
$H(\mu_{\lambda,\tau}^{(n)};r)<n H(p)$.

Let $0\le t\le1$ be with,
\begin{equation}
H\left(\left\lfloor r^{-1}A_{\l,\tau}^{(n)}+t\right\rfloor \right)\le H(\mu_{\lambda,\tau}^{(n)};r)< nH(p).\label{eq:small ent of translation}
\end{equation}
Denote by $\cA$ the collection of all nonzero polynomials $Q(X,Y_{1},Y_{2})\in\mathcal{Q}^{(n)}$
with $|Q(\lambda,1,\tau)|\le r$. From \eqref{eq:small ent of translation}
and
\[
T_{i}(Y_{1},Y_{2})\ne T_{j}(Y_{1},Y_{2})\text{ for }1\le i<j\le m,
\]
it follows that $\mathcal{A}\ne\varnothing$.

Let 
\[
\{Q_{i}(X,Y_{1},Y_{2})=P_{i,1}(X)Y_{1}+P_{i,2}(X)Y_{2}\}_{i=1}^{q}
\]
be an enumeration of $\mathcal{A}$. By assuming that $C$ is
large enough with respect to $L$ and $\e$, and that $n$ is
large enough with respect to $L$, $\e$ and $C$, it follows that at
least one of the three alternatives in Proposition $\ref{prop:common zero}$ holds for the collection $\mathcal{A}$.

In case the first alternative holds, we have $Q_i(\eta,1,\s)=0$ for all $Q_i\in \cA$ and hence the definition of $\cA$ yields
\[
H(\mu_{\eta,\s}^{(n)})\le H\left(\left\lfloor r^{-1}A_{\l,\tau}^{(n)}+t\right\rfloor \right)\le H(\mu_{\lambda,\tau}^{(n)};r).
\]
Thus we see that the first alternative of Proposition \ref{prp:approx by alg ifs when ent is small}
holds.

Now we consider the case, when the second alternative of Proposition $\ref{prop:common zero}$ holds.
Then we set $\gamma=\{\eta\}\times\R$.
We have $\dist(\gamma,(\l,\tau))\le r^{1/(C\log n)}$.
Furthermore, we have $Q_i(\eta,1,\s)=0$ for all $\s\in\R$ and $Q_i\in\cA$.
As above, this and the definition of $\cA$ yield
\[
H(A_\gamma^{(n)})\le H\left(\left\lfloor r^{-1}A_{\l,\tau}^{(n)}+t\right\rfloor \right)\le H(\mu_{\lambda,\tau}^{(n)};r).
\]
Thus we see that the second alternative of  Proposition \ref{prp:approx by alg ifs when ent is small}
holds.

Finally, we consider the third alternative of Proposition $\ref{prop:common zero}$.
In this case, we assume without loss of generality that $|P_{1,2}(\lambda)|\ge r^{1/2}$.
We set $\gamma$ to be the non-degenerate curve
\[
\{(\wt \l,\wt\tau)\in(0,1)\times \R: \wt \tau = -\frac{P_{1,1}}{P_{1,2}}(\wt\l)\}.
\]
We note that
\[
\Big|\frac{P_{1,1}}{P_{1,2}}(\l) +\tau\Big|=\frac{|Q_1(\l,1,\tau)|}{|P_{1,2}(\l)|}\le r^{1/2},
\]
hence $\dist(\gamma,(\l,\tau))\le r^{1/2}$.
Furthermore, it follows from $P_{1,1}P_{i,2}-P_{1,2}P_{i,1}=0$ that $Q_{i}(\cdot, 1,\cdot)$ vanishes
along $\gamma$ for all $Q_{i}\in \cA$.
Again, this and the definition of $\cA$ yields
\[
H(A_\gamma^{(n)})\le H\left(\left\lfloor r^{-1}A_{\l,\tau}^{(n)}+t\right\rfloor \right)\le H(\mu_{\lambda,\tau}^{(n)};r).
\]
Thus we see that the second alternative of  Proposition \ref{prp:approx by alg ifs when ent is small}
holds.
\end{proof}

\subsection{}

Next, we consider the situation when the approximating parameters $(\eta,\s)$ described in the
first alternative of Proposition \ref{prp:approx by alg ifs when ent is small} exist.
In this case, we show that we can find a larger value of $n$ for which $\mu_{\lambda,\tau}^{(n)}$
has significant entropy on a suitable scale.
This is achieved in the next result.
We will apply it choosing $n$ to be as large as possible subject to the constraint $|\lambda-\eta|,|\tau-\s|<n^{-Cn}$.
Therefore, how large this new $n$ will be, ultimately depends on the approximations $|\lambda-\eta|,|\tau-\s|$, which we
will control using the indirect hypothesis in our proof of Theorem \ref{th:BV2-general}.

The proof of this result uses in a crucial way a quantitative bound on the separation between points in $\cX^{(n)}$.
There is no analogous result corresponding to the second alternative  of Proposition \ref{prp:approx by alg ifs when ent is small},
because members of the family of potential approximating curves may intersect each other and there is no separation
between them.
This is why the results of Sections \ref{sc:ssm-Cx} and \ref{sc:repel} are indispensable.

\begin{prp}
\label{prp:cond for large ent}
For every $\e>0$ there exists
$C=C(L,\e)>1$ such that for every $\lambda,\tau\in\R$,
with $\e\le\lambda\le1-\e$, $|\tau|\le\e^{-1}$
and without exact overlaps,
there exists $N=N(L,\e,C,\lambda,\tau)\ge1$ such that the following
holds.
Let $n\ge N$ and suppose that there exists $(\eta,\sigma)\in\cX^{(n)}$ such that
\[
|\lambda-\eta|,|\tau-\sigma|\le n^{-Cn}.
\]
Then,
\[
\frac{1}{n\log\lambda^{-1}}H(\mu_{\lambda,\tau}^{(n)};r)>s(\l)-\e
\]
for every $0<r<\|(\lambda,\tau)-(\eta,\sigma)\|_{\infty}^{C\log n}$.
\end{prp}

\begin{proof}
Let $\e>0$ and let $\lambda,\tau\in\R$ be with $\e\le\lambda\le1-\e$,
$|\tau|\le\e^{-1}$
and without exact overlaps.
Let $C=C(L,\e)>1$ and $N=N(L,\e,C,\lambda,\tau)\ge1$
be as obtained in Proposition \ref{prp:approx by alg ifs when ent is small}.
We also assume that $N$ and $C$ are large with respect to $L$
and $\e$ in a manner to be described below.
Let $n\ge N$ and
suppose that there exists $(\eta,\sigma)\in\cX^{(n)}$ 
such that
$|\lambda-\eta|,|\tau-\sigma|\le n^{-Cn}$.

Suppose to the contrary that there exists $0<r<\Vert(\lambda,\tau)-(\eta,\sigma)\Vert_{\infty}^{C\log n}$
with
\[
\frac{1}{n\log\lambda^{-1}}H(\mu_{\lambda,\tau}^{(n)};r)\le s(\l)-\e.
\]
We apply  Proposition \ref{prp:approx by alg ifs when ent is small}.
As we have already remarked, the second alternative in the conclusion cannot hold
by Corollary \ref{cr:no-bad-curves}, provided $n$ is sufficiently large depending on $\l$, $\tau$ and $\e$,
which we assumed.
Therefore, we must have the first alternative, hence
there exists $(\wt\eta,\wt\sigma)\in\cX^{(n)}$
such that
\[
\Vert(\lambda,\tau)-(\wt\eta,\wt\sigma)\Vert_{\infty}\le r^{1/(C\log n)}<\Vert(\lambda,\tau)-(\eta,\sigma)\Vert_{\infty}.
\]
In particular $(\eta,\sigma)\ne(\wt\eta,\wt\sigma)$, and
\begin{equation}
\Vert(\eta,\sigma)-(\wt\eta,\wt\sigma)\Vert_{\infty}\le\Vert(\eta,\sigma)-(\lambda,\tau)\Vert_{\infty}+\Vert(\lambda,\tau)-(\wt\eta,\wt\sigma)\Vert_{\infty}<2n^{-Cn}.\label{eq:ub on dist of pairs}
\end{equation}

Note that $\eta$ and $\wt \eta$ are both roots of polynomials in $\cP_{2L^2n}^{(2n)}$.
Using $|\eta-\tilde{\eta}|<2n^{-Cn}$,
and that $C$ is large enough with respect to $L$,
Lemma \ref{lem:dist between distinct roots} gives $\eta=\wt{\eta}$.
This together with $(\eta,\sigma)\ne(\wt{\eta},\wt{\sigma})$ implies $\sigma\ne\wt{\sigma}$.

Since $(\eta,\sigma),(\eta,\wt\sigma)\in\cX^{(n)}$ (we just used $\eta=\wt \eta$),
it follows that there are $P_1,P_2,\wt P_1,\wt P_2\in\cP_L^{(n)}$
such that $P_2(\eta),\wt P_2(\eta)\neq 0$, $\s=P_1(\eta)/P_2(\eta)$ and $\wt\s=\wt P_1(\eta)/\wt P_2(\eta)$.
We set
\[
A(X):=P_{1}(X)\wt P_{2}(X)-P_{2}(X)\wt P_{1}(X).
\]
Then
we get
\[
A(\eta)=P_{2}(\eta)\wt P_{2} (\eta)(\sigma-\wt\sigma).
\]
In particular, $A(\eta)\ne0$.
Since $\e\le\lambda\le1-\e$
we may assume that $|\eta|\le1-\e/2$.
This together with $P_{2},\wt P_{2}\in\cP_{L}^{(n)}$
implies $|P_{2}(\eta)\wt P_{2}(\eta)|=O_{\e,L}(1)$. Thus
from $|\sigma-\wt\sigma|<2n^{-Cn}$ we get $|A(\eta)|=O_{\e,L}(n^{-Cn})$.
Since $A\in\cP_{2L^2n}^{(2n)}$ and $\eta$ is a root of a polynomial in $\cP_{2L^2n}^{(2n)}$,
it follows from Lemma \ref{lem:lb on value of poly at root} that $A(\eta)=0$, provided $C$ is sufficiently
large depending on $\e$ and $L$, which we assumed.
This contradiction completes the proof.
\end{proof}

\subsection{Proof of Theorem \ref{th:BV2-general}}\label{sc:BV2-general-proof}

For $(\lambda,\tau)\in(0,1)\times\R$, $\e>0$ and $n\ge1$ we write $F_{\l,\tau,\e}^{(n)}$
for the set of all pairs $(\eta,\sigma)\in\cX^{(n)}$
such that
\[
\frac{1}{n\log\eta^{-1}}H(\mu_{\eta,\sigma}^{(n)})\le\dim\mu_{\lambda,\tau}+\e.
\]

With this notation, we rephrase the statement of Theorem \ref{th:BV2-general} as follows.

\begin{thm*}
Let $(\lambda,\tau)\in(0,1)\times\R$ be with $\dim\mu_{\lambda,\tau}<s(\lambda)$,
and suppose that $(\lambda,\tau)$ has no exact overlaps.
Then for
every $\e>0$ and $N\ge1$, there exist $n\ge N$ and $(\eta,\sigma)\in F_{\l,\tau,\e}^{(n)}$
such that,
\[
|\lambda-\eta|,|\tau-\sigma|\le\exp(-n^{1/\e}).
\]
\end{thm*}

Note that we have omitted the claim that $h(\g)\ge\min\{\log \l^{-1},H(p)\}-\e$ for all $\gamma\in\Gamma$ with $(\eta,\s)\in\gamma$.
However, this follows immediately from the bound on the distance between $(\l,\tau)$ and $(\eta,\s)$
and Corollary \ref{cr:no-bad-curves}.
We have also omitted the claims related to the algebraicity of $\eta$ and $\s$, however these immediately follow
from $(\eta,\s)\in\cX^{(n)}$. 

We shall also need the following lemma which is stated in \cite{BV-transcendent}*{Lemma 12}.
\begin{lem}\label{diff of ent}
Let $\nu$ be a compactly supported Borel probability measure on $\R$. Then for any $r_2\ge r_1>0$ we have
\[
0\le H(\nu;r_1)-H(\nu;r_2)\le 2(\log r_2-\log r_1).
\]
\end{lem}

\begin{proof}[Proof of Theorem \ref{th:BV2-general}]
Let $(\lambda,\tau)\in(0,1)\times\R$ be with $\dim\mu_{\lambda,\tau}<s(\lambda)$
and without exact overlaps.
Suppose to the contrary that there exists
\begin{equation}
0<\e<\frac{1}{3}\min\{1,\log\l^{-1}\}(s(\lambda)-\dim\mu_{\lambda,\tau}),\label{eq:restriction on epsilon-1}
\end{equation}
such that,
\begin{equation}
\Vert(\lambda,\tau)-(\eta,\sigma)\Vert_{\infty}>\exp(-n^{1/\e})\label{eq:cont assumption-1}
\end{equation}
for all $n\ge\e^{-1}$ and $(\eta,\sigma)\in F_{\l,\tau,3\e}^{(n)}$.
Let $C>1$ be large with respect to $L$, $\lambda$, $\tau$ and
$\e$, and let $n_{0}\ge1$ be large with respect to $C$.

We
next define by induction a sequence $n_{0}<n_{1}<\ldots$ of positive integers.
Let $j\ge0$ and suppose that $n_{j}$ has been chosen. Write
\[
q=\left\lceil \frac{Cn_{j}(\log n_{j})^{2}}{\log\lambda^{-1}}\right\rceil ,
\]
and assume first that,
\[
H(\mu_{\lambda,\tau}^{(q)};q^{-Cq\log q})\ge q(\dim\mu_{\lambda,\tau}+2\e)\log\lambda^{-1}.
\]
In this case we set $n_{j+1}=q$. Note that from $\dim\mu_{\lambda,\tau}<1$
and \cite{Hoc-self-similar}*{Theorem 1.3},
\begin{equation}
\lim_n\frac{H(\mu_{\lambda,\tau}^{(n)};\lambda^{10n})}{n\log\lambda^{-1}}=\dim\mu_{\lambda,\tau}.\label{eq:follows from =00005BHo=00005D-1}
\end{equation}
Thus, by assuming that $n_{0}$ is large enough,
\begin{equation}
H(\mu_{\lambda,\tau}^{(n_{j+1})};n_{j+1}^{-Cn_{j+1}\log n_{j+1}}|\lambda^{10n_{j+1}})\ge\e n_{j+1}\log\lambda^{-1}.\label{eq:lb in first case-1}
\end{equation}

Next suppose that
\begin{equation}\label{eq:second-case}
H(\mu_{\lambda,\tau}^{(q)};q^{-Cq\log q})<q(\dim\mu_{\lambda,\tau}+2\e)\log\lambda^{-1}.
\end{equation}
By \eqref{eq:restriction on epsilon-1}, we have
\[
\dim\mu_{\lambda,\tau}+2\e<s(\lambda)-\e.
\]
We apply Proposition \ref{prp:approx by alg ifs when ent is small}.
The second alternative cannot hold by Corollary \ref{cr:no-bad-curves}.
Therefore, we must have the first alternative, which together with \eqref{eq:second-case}
imply that there exists $(\eta,\s)\in  F_{\l,\tau,3\e}^{(q)}$ with 
$|\lambda-\eta|,|\tau-\sigma|<q^{-C^{1/2}q}$ provided
$C$ and $n_{0}$ are large enough (with respect
to the specified parameters), which we assumed.
In this case, we take $n_{j+1}$ to be the largest integer $n$ with
$\Vert(\lambda,\tau)-(\eta,\sigma)\Vert_{\infty}<n^{-C^{1/2}n}$.
In particular we have $n_{j+1}\ge q$.

Since
\[
(n_{j+1}+1)^{-C^{1/2}(n_{j+1}+1)}\le\Vert(\lambda,\tau)-(\eta,\sigma)\Vert_{\infty}<n_{j+1}^{-C^{1/2}n_{j+1}},
\]
by Proposition \ref{prp:cond for large ent} and by assuming that $C$
and $n_{0}$ are large enough,
\[
H(\mu_{\lambda,\tau}^{(n_{j+1})};(n_{j+1}+1)^{-C(n_{j+1}+1)\log n_{j+1}})>n_{j+1}(s(\lambda)-\frac{\e}{2\log\lambda^{-1}})\log\lambda^{-1}.
\]
From this, Lemma \ref{diff of ent}, \eqref{eq:follows from =00005BHo=00005D-1}
and by assuming that $n_{0}$ is large enough,
\begin{align*}
H(\mu_{\lambda,\tau}^{(n_{j+1})};n_{j+1}^{-Cn_{j+1}\log n_{j+1}}|\lambda^{10n_{j+1}})  \ge & n_{j+1}(s(\lambda)-\frac{\e}{\log\lambda^{-1}})\log\lambda^{-1}\\
 & -  n_{j+1}(\dim\mu_{\lambda}+\frac{\e}{\log\lambda^{-1}})\log\lambda^{-1}.
\end{align*}
Thus from \eqref{eq:restriction on epsilon-1},
\begin{equation}\label{lb for second case-1}
H(\mu_{\lambda,\tau}^{(n_{j+1})};n_{j+1}^{-Cn_{j+1}\log n_{j+1}}|\lambda^{10n_{j+1}})\ge\e n_{j+1}.
\end{equation}

Fix a large integer $N=N(C,n_{0})\ge1$, to be determined later in
the proof. For $1\le j\le N$ set
\[
K_{j}=\frac{C(\log n_{j})^{2}}{\log\lambda^{-1}}.
\]
By \eqref{eq:lb in first case-1} and \eqref{lb for second case-1} it follows that in either case,
\[
H(\mu_{\lambda,\tau}^{(n_{j})};\lambda^{K_{j}n_{j}}|\lambda^{10n_{j}})\ge\e n_{j}\min\{1,\log\lambda^{-1}\}.
\]
We also have $\lambda^{-n_{1}}\ge\max\{2,\lambda^{-2}\}$, $n_{j}\ge C(\log K_{j})^{2}$
and
\[
n_{j+1}\ge\left\lceil \frac{Cn_{j}(\log n_{j})^{2}}{\log\lambda^{-1}}\right\rceil \ge K_{j}n_{j},
\]
if $n_{0}$ is assumed to be large enough with respect to $C$ and
$\lambda$. From all of this together with Lemma \ref{lem:bd of ent of single dig2}
it follows that the conditions of Proposition \ref{prop:scale-and-convolve}
are satisfied. Thus, by assuming that $C$ is large enough,
\begin{equation}
\sum_{j=1}^{N}\frac{1}{\log K_{j}\log\log K_{j}}\le C\left(1+\frac{1}{n_{1}}\sum_{j=1}^{N}\log K_{j}\right).\label{eq:by inequality involving =00007BK_j=00007D-1}
\end{equation}

Next we estimate how fast the sequence $\{n_{j}\}_{j\ge0}$ may grow.
Let $j\ge0$ and recall $q=\left\lceil \frac{Cn_{j}(\log n_{j})^{2}}{\log\lambda^{-1}}\right\rceil $.
Recall also that in the definition of $n_{j+1}$, if the first alternative
occurred then we took $n_{j+1}=q$. Suppose next that the second alternative
has occurred, and let $(\eta,\sigma)$ be as obtained during the definition
of $n_{j+1}$. Recall that we have selected $n_{j+1}$ so that $\Vert(\lambda,\tau)-(\eta,\sigma)\Vert_{\infty}<n_{j+1}^{-C^{1/2}n_{j+1}}$.
Additionally, from $(\eta,\sigma)\in F_{\l,\tau,3\e}^{(q)}$,
since we may assume $n_{0}\ge\e^{-1}$ and by \eqref{eq:cont assumption-1},
it follows that $\Vert(\lambda,\tau)-(\eta,\sigma)\Vert_{\infty}>\exp(-q^{\e^{-1}})$.
Thus, if $n_{0}$ is large enough,
\[
n_{j+1}<\log n_{j+1}^{C^{1/2}n_{j+1}}<-\log\Vert(\lambda,\tau)-(\eta,\sigma)\Vert_{\infty}<q^{\e^{-1}}\le n_{j}^{2\e^{-1}}.
\]
Hence by induction it follows that $n_{j}\le n_{0}^{2^{j}\e^{-j}}$
for all $j\ge0$.

Write $j_{0}=\left\lceil \log\log n_{0}\right\rceil $, then if $n_{0}$
is assumed to be large enough,
\[
\log K_{j}\le3\log\log n_{j}\le3\log\log(n_{0}^{2^{j}\e^{-j}})\le6\e^{-1}(j+j_{0}),
\]
and
\[
\log\log K_{j}\le\log(6\e^{-1}(j+j_{0}))\le2\log(j+j_{0}).
\]
It follows that,
\begin{align*}
\sum_{j=1}^{N}\frac{1}{\log K_{j}\log\log K_{j}}\ge&\sum_{j=1}^{N}\frac{\e}{12(j+j_{0})\log(j+j_{0})}\\
\ge&\frac{\e}{12}\int_{j_{0}+1}^{j_{0}+N+1}\frac{1}{x\log x}dx\\
\ge&\frac{\e}{12}(\log^{(2)}N-\log^{(2)}(j_{0}+1)),
\end{align*}
where $\log^{(2)}$ stands for the composition of the $\log$ function with itself.
Similarly, we write $\exp^{(2)}$ for the composition of $\exp$ with itself.
Set,
\[
N=\left\lceil \exp^{(2)}(\log^{(2)}(j_{0}+1)+C^{2})\right\rceil,
\]
then,
\begin{equation}
\sum_{j=1}^{N}\frac{1}{\log K_{j}\log\log K_{j}}\ge\frac{\e}{12}C^{2}.\label{eq:lb on log log term}
\end{equation}

On the other hand, for each $1\le j\le N$,
\[
\log K_{j}\le6\e^{-1}(j+j_{0})\le12\e^{-1}N.
\]
Additionally, if $n_{0}$ is assumed to be sufficiently large,
\[
N\le\exp^{(2)}(\log j_{0})\le2\log n_{0}.
\]
Hence, by assuming once more that $n_{0}$ is sufficiently large,
\[
\frac{1}{n_{1}}\sum_{j=1}^{N}\log K_{j}\le\frac{12N^{2}}{n_{1}\e}\le\frac{48(\log n_{0})^{2}}{n_{0}\e}\le1.
\]
If $C$ is assumed to be large enough with respect to $\e$
then this, together with \eqref{eq:by inequality involving =00007BK_j=00007D-1}
and \eqref{eq:lb on log log term}, yields the desired contradiction
and completes the proof of the theorem.
\end{proof}

\subsection{Proof of Corollary \ref{cr:Mike-general}}

We recall the statement of the Corollary.

\begin{cor*}
Let $E$ be the set of $(\l,\tau)\in(0,1)\times\R$ without exact overlaps and with $\dim\mu_{\l,\tau}<s(\l)$. Then $E$ is of Hausdorff dimension $0$.
\end{cor*}

\begin{proof}
For $n\ge1$ denote by $F^{(n)}$ the set of $(\eta,\s)\in(0,1)\times\R$ such that $\eta$ is a root of some nonzero polynomial in $\cP_{2L^{2}n}^{(2n)}$ and $\sigma=P_{1}(\eta)/P_{2}(\eta)$ for some $P_{1},P_{2}\in\cP_{L}^{(n)}$ with $P_{2}(\eta)\ne0$. We have
\begin{align}
|F^{(n)}|  \le & 2n |\cP_{2L^{2}n}^{(2n)}| \cdot |\cP_{L}^{(n)}|^2 \label{bd on card of F^n}\\
\le&2n (4L^2n+1)^{2n}(2L+1)^{2n} \le (5L^2n)^{5n}. \nonumber
\end{align}

Fix $\d>0$, and let $N\ge 1$ be large. For $x\in\R^2$ and $r>0$, write $\overline B(x,r)\subset\R^2$ for the closed ball with centre $x$ and radius $r$. From Theorem \ref{th:BV2-general},
\begin{equation}\label{cover by balls}
E \subset \cup_{n\ge N}\cup_{x\in F^{(n)}}\overline B(x,\exp(-n^2)).
\end{equation}

Write $\cH^\d_\infty$ for the $\d$-dimensional Hausdorff content. That is for $A\subset\R^2$,
\[
\cH^\d_\infty(A)=\inf\big\{\sum_{n\ge1}\mathrm{diam}(A_n)^\d:A_1,A_2,\ldots\subset\R^2\text{ and }A\subset\cup_{n\ge1}A_n\big\}.
\]
By \eqref{cover by balls} and \eqref{bd on card of F^n},
\[
\cH^\d_\infty(E)\le\sum_{n\ge N}|F^{(n)}|2^\d\exp(-\d n^2)\le \sum_{n\ge N}(5L^2n)^{5n}2^\d\exp(-\d n^2).
\]
The last expression tends to $0$ as $N$ tends to infinity, which gives $\cH^\d_\infty(E)=0$. It follows that $E$ is of Hausdorff dimension at most $\d$, which completes the proof of the corollary.
\end{proof}

\section{Entropy rates of derivatives I}\label{sc:lb on h(R,lam,M)}

In this section, we introduce certain self-affine measures whose coordinates will be defined as the derivatives of the random function $A_R$.
We will give lower bounds on the entropy rates of these measures.
We discussed informally the need for introducing these objects in Section \ref{sc:Ideas2}.
These measures are parametrized by a function $R\in\cR_L\cap\Q[[X]]$, a number $\l\in(0,1)$ and an integer $K$.
Our aim is to show that given $(\l_0,\tau_0)\in(0,1)\times \R$, the entropy rate of the self-affine measures will
be bounded below by $\min\{\log \l_0^{-1},h(\l_0,\tau_0)\}-\e$ for an arbitrarily small $\e$, provided $\l$ is
sufficiently close to $\l_0$, $R(\l)$ is sufficiently close to $\tau_0$, and $K$ is sufficiently large.
This result will play a role in the proof of Theorem \ref{th:EO-general} similar to the role of Corollary \ref{cr:no-bad-curves}
in the proof of Theorem \ref{th:BV2-general}.

We will achieve our above stated aim in the next section.
In this one, we give some estimates, which will be useful
when $R$ has no singularities near $\l_0$.
The complementary situation will be addressed in the next section.

We begin by discussing the definition of the self-affine measures in question.
Recall that given $R\in\cR_L\cap\Q[[X]]$ and $n\ge0$ we set
\[
A_{R}^{(n)}=\sum_{k=0}^{n-1}T_{\xi_{k}}(1,R(X))X^k \qquad\text{ and }\qquad A_{R}=\sum_{k=0}^{\infty}T_{\xi_{k}}(1,R(X))X^k,
\]
where $\xi_k$ are independent random variables with $\P\{\xi_k=j\}=p_j$ for $1\le j\le m$.
Furthermore, given $\l\in(0,1)$ that is not a pole of $R$ and $K\in\Z_{\ge 1}$, we define the $\R^K$-valued
random vector
\[
B_{R,\l,K}^{(n)}=\Big(A_{R}^{(n)}(\l),\frac{d}{dX}A_{R}^{(n)}(\l),\ldots,\frac{d^{K-1}}{dX^{K-1}}A_{R}^{(n)}(\l)\Big).
\]
We write $\nu_{R,\l,K}^{(n)}$ for the law of this random vector.
We note that these measures converge weakly to a self-affine measure $\nu_{R,\l,K}$ as $n\to\infty$.

Indeed, to see this, we first write
\[
A_{R}^{(n)}=X\wt A_{R}^{(n-1)}+T_{\xi_{0}}(1,R(X)),
\]
where
\[
\wt A_{R}^{(n-1)} = \sum_{k=1}^{n-1}T_{\xi_{k}}(1,R(X))X^{k-1}.
\]
Note that $\wt A_{R}^{(n-1)}$ has the same distribution as $A_{R}^{(n-1)}$, and it is independent of $\xi_0$.
Taking derivatives of this identity, we get
\begin{align*}
\frac{d}{dX}A_{R}^{(n)}=&X\frac{d}{dX}\wt A_{R}^{(n-1)} + \wt A_{R}^{(n-1)} + T_{\xi_{0}}(0,\frac{d}{dX} R),\\
\frac{d^2}{dX^2}A_{R}^{(n)}=&X\frac{d^2}{dX^2}\wt A_{R}^{(n-1)} + 2\frac{d}{dX}\wt A_{R}^{(n-1)} + T_{\xi_{0}}(0,\frac{d^2}{dX^2}R),\\
\vdots&\\
\frac{d^{K-1}A_{R}^{(n)}}{dX^{K-1}}=&X\frac{d^{K-1}\wt A_{R}^{(n-1)}}{dX^{K-1}} + (K-1)\frac{d^{K-2}\wt A_{R}^{(n-1)}}{dX^{K-2}} + T_{\xi_{0}}(0,\frac{d^{K-1}R}{dX^{K-1}}).
\end{align*}
Using these identities we can write
\begin{equation}\label{rec iden B}
B_{R,\l,K}^{(n)}=\Theta(\l,K) \wt B_{R,\l,K}^{(n-1)}+v_{\xi_0}(R,\l,K),
\end{equation}
where
\[
\wt B_{R,\l,K}^{(n-1)} = (\wt A_{R}^{(n-1)}(\l),\frac{d}{dX} \wt A_{R}^{(n-1)}(\l),\ldots,\frac{d^{K-1}}{dX^{K-1}}\wt A_{R}^{(n-1)}(\l)),
\]
\[
\Theta(\l,K)
=\left(
\begin{array}{ccccc}
\l &0&0&\cdots&0\\
1&\l&0&\cdots&0\\
0&2&\l&\cdots&0\\
\vdots&\vdots&\vdots&\ddots&\vdots\\
0&0&\cdots&K-1&\l
\end{array}
\right),
\]
and
\[
v_j(R,\l,K)=(T_{j}(1,R(\l)),T_{j}(0,\frac{d}{dX}R(\l)),\cdots,T_{j}(0,\frac{d^{K-1}}{dX^{K-1}}R(\l)))^T.
\]
Note that $\wt B_{R,\l,K}^{(n-1)}$ has the same distribution as $B_{R,\l,K}^{(n-1)}$, and it is independent of $\xi_0$.
We see therefore that $\nu_{R,\l,K}$ is the self-affine measure associated to the IFS
\begin{equation}\label{SA IFS}
\{x\mapsto \Theta(\l,K)x+v_j(R,\l,K):j=1,\ldots,m\}
\end{equation}
and the probability vector $(p_1,\ldots,p_m)$.

We define
\[
h(R,\l,K)=\lim_{n\to\infty}\frac{H(B_{R,\l,K}^{(n)})}{n},
\]
where the limit exists by subadditivity, which can be seen by iterating \eqref{rec iden B}.

Given $R\in\Q[[X]]$ recall the notation $h(R)$ from Section \ref{sc:ssm-Cx}. The purpose of this section is to prove the following result.

\begin{prp}\label{pr:ent-bouquet}
Let $R_0\in\cR_L\cap\Q[[X]]$.
Then for all $\e>0$, there is a number $K=K(R_0,\e)>0$ such that the following holds.
Let $R\in\cR_L\cap\Q[[x]]$ such that $|R-R_0|\le 2^{-K}$, and let $\l\in(\e,1-\e)$ which is not a pole of $R$.
Then
\[
h(R,\l,K)\ge h(R_0)-\e.
\]
\end{prp}

This result is based on the following intuitive idea.
The dimension of the self-affine measure $\nu_{R_n,\l_n,K_n}$ is bounded above by
\[
h(R_n,\l_n,K_n)/\log \l_n^{-1},
\]
(see Lemma \ref{lm:sa-dim}). 
If we let $R_n\to R$ in the $|\cdot|$ metric and $K_n\to\infty$, then it is reasonable to expect that
$\nu_{R_n,\l_n,K_n}$ converges to $\mu_R$ in some sense, where $\mu_R$ is as defined in Section \ref{sc:ssm-Cx}.
Then we may expect that dimension is lower semi-continuous with respect to this convergence in a way that takes into
account the factor $1/\log\l_n^{-1}$.
This would yield $\liminf h(R_n,\l_n,K_n)\ge \dim\mu_R$, and we could conclude by Proposition \ref{pr:dim=h}.
It is not obvious how to define this convergence in a rigorous way, or indeed if it can be done at all.
Nevertheless, this intuition motivates the proof of the proposition that follows.

The proof requires some preparation.
We fix some $R\in\cR_L\cap\Q[[x]]$, $\l\in(\e,1-\e)$, which is not a pole of $R$, and $K\in\Z_{>0}$.
Recall that $\Omega=\{1,\ldots,m\}^{\Z_{\ge0}}$.
For $n\in\Z_{\ge 0}$ we denote by $\cD_n$ the partition of $\O$ into cylinder sets of generation $n$. That is,
\[
\cD_n:=\{D(j_0,\ldots,j_{n-1}):j_0,\ldots, j_{n-1}\in\{1,\ldots,m\}\},
\]
where
\[
D(j_0,\ldots,j_{n-1})=\{(\o_0,\o_1,\ldots)\in\Omega:\o_l=j_l\text{ for all $0\le l<n$}\}.
\]
We also introduce the partition $\cE_n$ that is defined such that two elements
$(\o_0,\o_1,\ldots)$ and $(\o'_0,\o'_1,\ldots)$ are in the same atom if and only if
\[
\frac{d^a}{dX^a}\sum_{k=0}^{n-1} T_{\o_k}(1,R(X))X^k\Big|_{X=\l}
=\frac{d^a}{dX^a}\sum_{k=0}^{n-1} T_{\o'_k}(1,R(X))X^k\Big|_{X=\l}
\]
for all $0\le a< K$.
Recall that $\beta$ is the Bernoulli measure on $\Omega$ corresponding to $p$.
Unwinding the definitions, we see that
\[
H(B_{R,\l,K}^{(n)})=H(\beta;\cE_n).
\]

For each $k\in\Z_{\ge 0}$, we also introduce the relation $\sim_k$ on $\Omega$.
For
\[
(\o_0,\o_1,\ldots),(\o'_0,\o'_1,\ldots)\in\Omega
\]
we write
\[
(\o_0,\o_1,\ldots)\sim_k(\o'_0,\o'_1,\ldots)
\]
if and only if there is some $n\ge k$ and
\[
(\o_0,\ldots, \o_{k-1},\wt\o_{k},\wt \o_{k+1},\ldots),(\o'_0,\ldots, \o'_{k-1},\wt\o_{k}',\wt\o_{k+1}',\ldots)\in\Omega
\]
that are in the same atom of $\cE_n$.
Note that the tails $(\wt\o_{k},\wt \o_{k+1},\ldots)$ and $(\wt\o_{k}',\wt\o_{k+1}',\ldots)$ are not required to be equal.
This relation is symmetric, but we do not claim that it is transitive, and indeed it is most likely not.
Therefore, $\sim_k$ does not necessarily induce a partition of $\Omega$.
Nevertheless, for $\omega\in\Omega$, we write
\[
[\omega]_{\sim_n}=\{\omega'\in\Omega: \omega\sim_n\omega'\}.
\]
It is clear that
\[
H(\beta;\cE_n)\ge \int -\log\beta([\omega]_{\sim_n}) d\beta(\omega).
\]
Given a $\Q[[X]]$-valued random element $A$ and $n\ge1$, recall the notation $H(A;n)$ from Section \ref{sc:ssm-Cx}.

The next lemma is a key step towards the proof of Proposition \ref{pr:ent-bouquet}.

\begin{lem}\label{lm:sim-conditional}
Let $\e>0$, $R\in\cR_L\cap\Q[[X]]$, $\l\in(\e,1-\e)$ and $K,N\in\Z_{>0}$.
Assume $\l$ is not a pole of $R$, and that $K$ is sufficiently large in a manner depending only on $\e$, $N$ and $L$.
Let $n\ge 0$ and let $D$ be an atom of $\cD_n$.
Then
\[
\int_{D} -\log\frac{\beta([\omega]_{\sim_{n+N}})}{\beta([\omega]_{\sim_{n}})}d\beta(\omega)\ge\beta(D) H(A_{R};N).
\]
\end{lem}

The proof of this lemma depends on a series of lemmata, which we give now.
Given a Borel subset $E\subset\O$ with $\b(E)>0$, we write $\b_E$ for the conditioning of $\b$ on $E$. That is, $\b_E(F)=\b(E\cap F)/\b(E)$ for every $F\subset\O$ Borel.

\begin{lem}\label{lm:atom-partitions}
Let $\e>0$, $R\in\cR_L\cap\Q[[X]]$, $\l\in(\e,1-\e)$ and $K,N\in\Z_{>0}$.
Assume $\l$ is not a pole of $R$, and that $K$ is sufficiently large in a manner depending only on $\e$, $N$ and $L$.
Let $n\ge 0$,
then each atom $D$ of $\cD_n$ has a finite partition $\cF_D$ that depends only on $n$, $R$ and $N$, and the following hold.
\begin{enumerate}
\item For each $D$, we have $H(\beta_{D};\cF_D)=H(A_{R};N)$.
\item For each pair $D_1$, $D_2$, there is bijection $\Phi:\cF_{D_1}\to\cF_{D_2}$ such that
$\omega\in D_1$, $\omega'\in D_2$ and $\omega\sim_{n+N}\omega'$ implies
$\Phi([\omega]_{\cF_{D_1}})=[\omega']_{\cF_{D_2}}$.
\end{enumerate}
\end{lem}

In the next lemma, recall the absolute value $|\cdot|$ on $\Q[[X]]$ defined in Section \ref{sc:ssm-Cx}.

\begin{lem}\label{lm:multiplicity}
Let $K,N,l\in\Z_{>0}$ and $\e>0$.
Assume $K$ is sufficiently large in a manner depending only on $\e$, $N$ and $l$.
Let $\l\in(\e,1-\e)$, $R\in\cR_l\cap\Q[[X]]$ and $P_1,P_2\in\cP_l$, be such that
$\l$ is not a pole of $R$, and
\[
\frac{d^a}{dX^a}(P_1(X)+P_2(X)R(X))\Big|_{X=\l}=0
\]
for $0\le a<K-1$.
Then $|P_1+P_2R|\le 2^{-N}$.
\end{lem}

\begin{proof}
Let $K$ be large in a manner depending on $\e$, $N$ and $l$.
Suppose to the contrary that $|P_1+P_2R|= 2^{-n}$ for some $n<N$.
Let $R=\wt P_1/\wt P_2$ for some $\wt P_1,\wt P_2\in\cP_l$ with $|\wt P_2|=1$,
and consider
\[
F=f_0+f_1X+f_2X^2+\ldots=X^{-n}(P_1\wt P_2+P_2\wt P_1)\in\Z[[X]].
\]
A simple calculation yields
\[
|f_j|\le 2l^2(j+n+1)\le 2Nl^2(j+1)\text{ for }j\ge0.
\]
If $K$ is sufficiently large with respect to the specified parameters then by Jensen's formula, it follows that $F$ cannot have a zero of multiplicity $K$ at $\l$.
This can be seen, for example, by following the proof of Lemma \ref{lem:bound on num of roots} applied to
the function $F((1-\e/2)X)$ in the role of $P$, and noting that the proof requires the integrality of the coefficients only to show that the first
nonzero coefficient is at least $1$ in absolute value.
Then
\[
P_1+P_2R=X^nF/\wt P_2
\]
cannot have a zero of multiplicity $K$ at $\l$ either, a contradiction.
\end{proof}

\begin{proof}[Proof of Lemma \ref{lm:atom-partitions}]
For
\[
P=\a_0+\a_1X+\a_2X^2+\ldots\in\Q[[X]]
\]
and $0\le n\le k\le\infty$, we write
$P^{[n,k)}$ for the polynomial
\[
\sum_{n\le j<k}\a_jX^j.
\]

Let $D\in\cD_n$.
The atoms of the partition $\cF_D$ are defined as the fibres of the map
\[
\omega\mapsto \Big(\sum_{j=n}^{n+N-1} T_{\omega_j}(1,R(X))X^j\Big)^{[n,n+N)}.
\]
The first claim follows immediately from the definitions.

Let $D_1,D_2\in\cD_n$.
Let $\omega_1\in D_1$ and $\omega_2\in D_2$ be
such that $\omega_1\sim_{n+N}\omega_2$ if such elements exist.
Now we take $\omega_1'\in[\omega_1]_{\cF_{D_1}}$ and $\omega_2'\in D_2$ such that $\omega_1'\sim_{n+N}\omega_2'$.
We show below that $\omega_2'\in[\omega_2]_{\cF_{D_2}}$, and this implies that the bijection
satisfying the second claim can be defined.

Without loss of generality, we assume that $\omega_1$ and $\omega_2$ are in the same $\cE_M$ atom for some $M\ge n+N$
and $\omega_1'$ and $\omega_2'$ are in the same $\cE_{M'}$ atom for some $M'\ge n+N$.
Indeed, to achieve this, we only need to modify the $j$'th coordinates for $j\ge n+N$ in $\omega_1$, $\omega_2$, $\omega_1'$
and  $\omega_2'$, which does not affect $\sim_{n+N}$, $\cF_{D_1}$ and $\cF_{D_2}$.
Therefore, we have
\begin{align*}
\frac{d^a}{dX^a}\sum_{j=0}^{M-1}T_{\omega_{1,j}}(1,R(X))X^j\Big|_{X=\l}
=&\frac{d^a}{dX^a}\sum_{j=0}^{M-1}T_{\omega_{2,j}}(1,R(X))X^j\Big|_{X=\l}\\
\frac{d^a}{dX^a}\sum_{j=0}^{M'-1}T_{\omega'_{1,j}}(1,R(X))X^j\Big|_{X=\l}
=&\frac{d^a}{dX^a}\sum_{j=0}^{M'-1}T_{\omega'_{2,j}}(1,R(X))X^j\Big|_{X=\l}
\end{align*}
for $0\le a<K$.
We take the difference of these equations.
Observe that the first $n$ terms cancel on both sides, because each of the pairs $\omega_1,\omega_1'$
and $\omega_2,\omega_2'$ is contained in an atom of $\cD_n$.
We write
\[
F(X)={\sum_{j\ge n}}^* (T_{\omega_{1,j}}-T_{\omega'_{1,j}}-
T_{\omega_{2,j}}+T_{\omega'_{2,j}})(1,R(X))X^j,
\]
where $\sum^*$ means that the terms $T_{\omega_{\cdot,j}}$ are present only for $j<M$ and $T_{\omega'_{\cdot,j}}$
are present only for $j < M'$. We have
\[
\frac{d^a}{dX^a}F(X)\Big|_{X=\l}=0
\]
for $0\le a<K$.

We apply Lemma \ref{lm:multiplicity} to the function $X^{-n}F(X)$ and conclude that
$F^{[n,n+N)}=0$.
Since $\omega_1$ and $\omega_1'$ are in the same $\cF_{D_1}$ atom, we have
\[
\Big({\sum_{j\ge n}}^*T_{\omega_{1,j}}(1,R(X))X^j\Big)^{[n,n+N)}
=\Big({\sum_{j\ge n}}^*T_{\omega'_{1,j}}(1,R(X))X^j\Big)^{[n,n+N)}.
\] 
Therefore,
\[
\Big({\sum_{j\ge n}}^*T_{\omega_{2,j}}(1,R(X))X^j\Big)^{[n,n+N)}=
\Big({\sum_{j\ge n}}^*T_{\omega'_{2,j}}(1,R(X))X^j\Big)^{[n,n+N)},
\] 
which shows that $\omega_2$ and $\omega_2'$ are in the same $\cF_{D_2}$ atom, as required.
\end{proof}

\begin{lem}\label{lm:min-entropy}
Let $y_1,\ldots,y_n,z_1,\ldots,z_n\in\R_{> 0}$.
Set $Y=y_1+\ldots+y_n$ and $Z=z_1+\ldots+z_n$.
Then
\[
\sum_{j=1}^n y_j\log z_j^{-1}\ge \sum_{j=1}^{n} y_j\log (y_jZ/Y)^{-1}.
\]
\end{lem}

\begin{proof}
Set $v_j=y_j/Y$ and $u_j=z_j/y_j$.
Then $(v_1,\ldots,v_n)$ is a probability vector and $\sum v_ju_j=Z/Y$.
By Jensen's inequality, we can write
\[
\sum -\frac{y_j}{Y}\log(z_j/y_j)=\sum -v_j \log u_j\ge -\log(Z/Y).
\]
This in turn yields
\[
\sum y_j\log z_j^{-1}\ge -Y\log(Z/Y)+\sum y_j \log y_j^{-1}
=\sum y_j \log(y_jZ/Y)^{-1},
\]
as required.
\end{proof}

\begin{proof}[Proof of Lemma \ref{lm:sim-conditional}]
The relation $\sim_n$ is defined in terms of the first $n$ coordinates, hence $[\omega]_{\sim_n}$ is the same for all
$\omega \in D$ and it is the union of some $\cD_n$-atoms.
Let $\cF_{D'}$ for $D'\in\cD_n$ be as in Lemma \ref{lm:atom-partitions}.
For $\omega\in D$, we let $E_{D'}(\omega)$ be the atom of $\cF_{D'}$ that is paired with $[\omega]_{\cF_D}$
by the bijection in Lemma \ref{lm:atom-partitions}.
Furthermore, let $E(\omega)$ be the union of $E_{D'}(\omega)$ with $D'$ running over the atoms of $\cD_n$
contained in $[\omega]_{\sim_n}$.
Then
\[
[\omega]_{\sim_{n+N}}\subset E(\omega).
\]

By abuse of notation, for $F\in\cF_D$, we write $E(F)$ for $E(\omega)$ with an arbitrary $\omega\in F$
and we write $[D]_{\sim_n}$ for $[\omega]_{\sim_n}$ with an arbitrary $\omega\in D$.
We write
\begin{align*}
\int_{D}-\log\beta([\omega]_{\sim_{n+N}})d\beta(\omega)
\ge& \int_D-\log\beta(E(\omega))d\beta(\omega)\\
=&\sum_{F\in\cF_D} \beta(F)\log\beta(E(F))^{-1}.
\end{align*}
Using $\sum_{F\in\cF_D}\beta(F)=\beta(D)$, $\sum_{F\in\cF_D}\beta(E(F))=\beta([D]_{\sim_n})$
and Lemma \ref{lm:min-entropy}, we can write
\begin{align*}
\sum_{F\in\cF_D} \beta(F)\log\beta(E(F))^{-1}
\ge& \sum_{F\in\cF_D} \beta(F)\log(\beta(F)\beta([D]_{\sim_n})/\beta(D))^{-1}\\
=&\beta(D)\sum_{F\in\cF_D} \frac{\beta(F)}{\beta(D)}\log(\beta(F)/\beta(D))^{-1}\\
&-\beta(D)\log\beta([D]_{\sim_n})\\
=&\beta(D)H(\beta_D;\cF_D)-\int_D\log\beta([\omega]_{\sim_n})d\beta(\omega).
\end{align*}
Since $H(\beta_D;\cF_D)=H(A_R;N)$, this proves the lemma.
\end{proof}

\begin{proof}[Proof of Proposition \ref{pr:ent-bouquet}]
Recall the definition of $\dim\mu_{R_0}$ from Section \ref{sc:ssm-Cx}. We let $N$ be sufficiently large, depending on $\e$ and $R_0$, so that
\[
H(A_{R_0};N)\ge N(\dim\mu_{R_0}-\e)=N (h(R_0)-\e).
\]
Here we used Proposition \ref{pr:dim=h}.
We let $K$ be sufficiently large depending on $\e$, $N$ and $L$, as required by Lemma \ref{lm:sim-conditional}.
Applying that lemma for any $R\in\cR_L\cap\Q[[X]]$ with $|R-R_0|\le 2^{-N}$ and $n=kN$, we get
\[
\int_{\Omega}-\log\frac{\beta([\omega]_{\sim_{(k+1)N}})}{\beta([\omega]_{\sim_{kN}})}d\beta(\omega)\ge H(A_R;N)\ge N(h(R_0)-\e).
\]

Summing this inequality for $k=0,\ldots,M-1$, we get
\[
H(B_{R,\l,K}^{(MN)})=H(\beta;\cE_{MN})
\ge\int_{\Omega}-\log\beta([\omega]_{\sim_{MN}})d\beta(\omega)
\ge MN(h(R_0)-\e).
\]
Dividing both sides by $MN$ and taking the limit $M\to\infty$, we conclude the proof.
\end{proof}

\section{Entropy rates of derivatives II}\label{lb ent ord M}

In this section, we continue our study of the entropy rate of the self-affine measures $\nu_{R,\l,K}$, which we
introduced in the previous section.
We will prove the following result, which achieves the goal we set out there.

\begin{prp}\label{pr:no-bad-bouquet}
Let $(\l_0,\tau_0)\in(0,1)\times(\R\setminus \{0\})$.
Then for every $h_0<\min\{\log\l_0^{-1},h(\l_0,\tau_0)\}$, there is $K\in\Z_{>0}$ such that
$h(R,\l,K)\ge h_0$ for all $R\in\cR_L$ and $\l\in(0,1)$ such that $|R(\l)-\tau_0|,|\l-\l_0|\le K^{-1}$.
\end{prp}

The proof relies on the following two propositions, which will be proven in Sections \ref{proof of first prop} and \ref{proof of second prop} below.

\begin{prp}\label{pr:complex}
Let $(\l_0,\tau_0)\in(0,1)\times(\R\setminus\{0\})$.
Then there is some $K\in\Z_{>0}$, depending only on $\l_0$ and $L$, such that the following holds.
Let $R_1,R_2,\ldots\in\cR_L\cap\Q[[X]]$ and let $\l_1,\l_2,\ldots\in(0,1)$ be such that $\lim \l_n=\l_0$ and $\lim R_n(\l_n)=\tau_0$.
Then at least one of the following two statements holds.
\begin{itemize}
\item There is a subsequence of $R_n$ that converges to some $R_0\in\cR_L\cap\Q[[X]]$ in the $|\cdot|$
metric and $R_0(\l_0)=\tau_0$, or
\item we have $\limsup_{n\to\infty} |d^j/dX^jR_n(\l_n)|=\infty$ for some $j\le K$.
\end{itemize}
\end{prp}

\begin{prp}\label{pr:singular-bouquet}
Let $(\l_0,\tau_0)\in(0,1)\times\R$.
Let $R_1,R_2,\ldots\in\cR_L\cap\Q[[X]]$ and let $\l_1,\l_2,\ldots\in(0,1)$ be such that
$\lim \l_n=\l_0$ and $\lim R_n(\l_n)=\tau_0$.
Assume that there is some $K\in\Z_{>0}$ such that
\[
|d^{K-1}/dX^{K-1}{R_n}(\l_n)|\to\infty.
\]
Then
\[
\dim\mu_{\l_0,\tau}\le\limsup_{n\to\infty}\frac{h(R_n,\l_n,K)}{\log\l_0^{-1}}
\]
for all $\tau\in\R$.
\end{prp}

\begin{proof}[Proof of Proposition \ref{pr:no-bad-bouquet}]
Suppose to the contrary that there are
$K_1,K_2,\ldots$, $R_1,R_2,\ldots{\in\cR_L}$ and $\l_1,\l_2,\ldots\in(0,1)$, such that
$\l_n\to\l_0$, $R_n(\l_n)\to\tau_0$, $K_n\to\infty$
and $h(R_n,\l_n,K_n)< h_0$ for all $n\ge1$.
As in the proof of Proposition \ref{pr:no-bad-curves}, since $\tau_0\ne0$ we may assume that $R_n\in\Q[[X]]$ for all $n\ge1$.

We apply Proposition \ref{pr:complex}.
It follows from Proposition \ref{pr:ent-bouquet} that the first alternative cannot hold, since in that case, we would have
\[
h(\l_0,\tau_0)\le h(R_0)\le h_0.
\]

We can, therefore, apply Proposition \ref{pr:singular-bouquet} and conclude that
\[
\dim\mu_{\l_0,\tau}\le\frac{h_0}{\log\l_0^{-1}}
\]
for all $\tau\in\R$. 
By Lemma \ref{lm:degenerate}, we now have $h(\{\l_0\}\times \R)\le h_0$, which contradicts $h_0<h(\l_0,\tau_0)$.
This completes the proof.
\end{proof}

\subsection{}\label{proof of first prop}

The purpose of this section is to prove Proposition \ref{pr:complex}, which relies on the following lemma.

\begin{lem}\label{lm:power-sums}
For every $\e>0$ and $M\in\Z_{>0}$, there is $\d>0$ such that the following holds.
Let $U,W\subset\C$ be finite multisets such that $|U|,|W|\le M$ and $|u|,|w|\ge 1$ for all $u\in U$ and $w\in W$.
Assume
\[
\Big|\sum_{u\in U} u^j-\sum_{w\in W} w^j\Big|\le\d
\]
for all $j=1,\ldots,2M$.
Then $|U|=|W|$ and
\[
\Big|1-\frac{\prod_{u\in U} u}{\prod_{w\in W}w}\Big|\le\e.
\]
\end{lem}

We use the following theorem of Tur\'an.

\begin{thm}[\cite{Tur-PWsums}*{Theorem 10.2}]\label{th:Turan}
Let $n\in\Z_{>0}$, $m\in\Z_{\ge 0}$ and $z_1,\ldots,z_n,b_1,\ldots,b_n\in\C$ be such that $z_1\ne0$ and
\[
|z_1-z_2|\le|z_1-z_3|\le\ldots\le|z_1-z_n|.
\]
Let
\[
0<\d_2<\d_1<\frac{n}{m+n+1}.
\]
Let $h$ be the largest index such that $|z_1-z_h|<|z_1|\d_2$.
Assume that $|z_1-z_{h+1}|>|z_1|\d_1$ or $h=n$.
Then there is an integer $j\in\{m+1,\ldots,m+n\}$ such that
\[
|b_1z_1^{j}+\ldots+b_nz_n^j|\ge 2 \Big(\frac{\d_1-\d_2}{12e}\Big)^n|b_1+\ldots+b_h||z_1|^j
\]
\end{thm}

\begin{rmk}
This theorem is stated and proved in \cite{Tur-PWsums} under the additional assumptions that $z_1=1$ and $z_1\neq z_2$.
The first one is simply a choice of normalization, which does not restrict generality.
The second one is not necessary, for we could satisfy it by an arbitrarily small perturbation of $z_2$.

We also note that \cite{Tur-PWsums}*{Theorem 10.2} is slightly more general than how we stated it here in that
the condition $|z_1-z_{h+1}|>|z_1|\d_1$ can be relaxed if the conclusion is adjusted accordingly.
However, this is not required by our application, and we refer to \cite{Tur-PWsums} for the details.
\end{rmk}

\begin{proof}[Proof of Lemma \ref{lm:power-sums}]
Set $\e_0=\min\{1/3,\e/2^{M+2}\}$.
For each $w\in W$, let $\d_w\in \{\e_0/2^j:1\le j\le2M+1\}$ satisfy
\begin{equation}\label{choice of d_w}
(B(w,2\d_w|w|)\backslash B(w,\d_w|w|))\cap (W\cup U)=\varnothing.
\end{equation} 
Such a $\d_w$ exists since $|W\cup U|\le 2M$.
This choice ensures that whenever $\d_{w_1}|w_1|\ge\d_{w_2}|w_2|$, we have
either $w_2\in B(w_1,\d_{w_1}|w_1|)$ or
\[
B(w_1,\d_{w_1}|w_1|)\cap B(w_2,\d_{w_2}|w_2|)=\varnothing.
\]
So there is a subset $W'\subset W$ such that $B(w,\d_w|w|)$ are pairwise disjoint for $w\in W'$
and they cover $W$.
Indeed, one can select the elements of $W'$ by going through the elements of $W$ in decreasing order of $\d_w|w|$,
taking an element in $W'$ if it is not yet covered.

We apply Theorem \ref{th:Turan} for each $w\in W'$ with $m=0$ and the following data.
We take $z_1=w$ and take $z_2,\ldots,z_n$ to be an appropriate ordering of the rest of
$U\cup W$. We take $b_j=1$ for indices corresponding to elements of $U$ and $b_j=-1$
for indices corresponding to elements of $W$.
We take $\d_2=\d_w$ and $\d_1=3\d_w/2$. We have $\d_1<3\e_0/2\le1/2\le n/(n+1)$. Let $h$ be as in the statement of the theorem. If $h<n$ then $|w - z_{h+1}|\ge|w|\d_w$, and so by \eqref{choice of d_w}
\[
|w - z_{h+1}|\ge2|w|\d_w>|w|\d_1.
\]
Thus by the theorem, there exists $1\le j\le2M$ so that
\begin{align*}
|b_1+\ldots+b_{h}|\le & \frac{1}{2}\Big(\frac{\d_1-\d_2}{12e}\Big)^{-n}|b_1z_1^{j}+\ldots+b_nz_n^j| \\
\le & \frac{1}{2}\Big(\frac{\e_0/2^{2M+2}}{12e}\Big)^{-2M}\d<1,
\end{align*}
provided $\d$ is taken sufficiently small in the statement of the lemma.
Hence, we must have $b_1+\ldots+b_h=0$, that is $|B(w,\d_w|w|)\cap U|=|B(w,\d_w|w|)\cap W|$.

Since $B(w,\d_w|w|)$ are pairwise disjoint for $w\in W'$ and they cover $W$, we can conclude
$|W|\le |U|$ and the opposite inequality can be proved by reversing the roles of $U$ and $W$.
Moreover, we can define a bijection $\f:W\to U$, such that for every $w\in W$ there exists $w'\in W'$ with $w,\f(w)\in B(w',\d_{w'}|w'|)$. For such $w,w'$,
\[
|w'|\le|w|+|w-w'|\le|w|+|w'|/3,
\]
and so 
\[
|w - \f(w)|<2\d_{w'}|w'|<4\e_0|w|.
\]
It follows that,
\[
\Big|1-\frac{\prod_{u\in U} u}{\prod_{w\in W}w}\Big|=\Big|1-\prod_{w\in W}\frac{\f(w)}{w}\Big|<2^{M}\cdot 4\e_0\le\e,
\]
which completes the proof of the lemma.
\end{proof}

\begin{proof}[Proof of Proposition \ref{pr:complex}]
We write $R_n=P^{(1)}_n/P^{(2)}_n$ for some $P^{(1)}_n,P_n^{(2)}\in\cP_L$ with $|P_n^{(2)}|=1$.
By passing to a subsequence if necessary, but without changing notation, we assume
$\lim_n P^{(1)}_n=P_0^{(1)}$ and $\lim_n P^{(2)}_n=P_0^{(2)}$ for some $P^{(1)}_0,P^{(2)}_0\in\cP_L$.
This convergence holds both in the $|\cdot|$ metric and uniformly on $B(0,r)$ for all $r<1$. Write $R_0$ for $P_0^{(1)}/P_0^{(2)}\in\cR_L\cap\Q[[X]]$, then $\lim_n|R_0-R_n|=0$.

Suppose first that $P_0^{(1)}\ne0$. We handle the case $P_0^{(1)}=0$ later on. We choose $\kappa\in(0,1-\l_0)$ small enough so that all zeros of $P_0^{(1)}$ and $P_0^{(2)}$ in the disk $B(\l_0,\kappa)$,
if they exist, are equal to $\l_0$.
We write $U_n$ and $W_n$ for the multisets of the zeros of $P_n^{(1)}$ and $P_n^{(2)}$ in $B(\l_0,\kappa)$.
By perturbing the numbers $\l_n$ if necessary, but without changing the notation, we may assume $\l_n\notin U_n\cup W_n$.
By the choice of $\kappa$, all elements of $U_n$ and $W_n$ are in arbitrarily small neighborhoods of $\l_0$ if $n$
is sufficiently large.
Furthermore, writing $k_1,k_2\in\Z_{\ge 0}$ for the multiplicities of $\l_0$ as zeros of $P_0^{(1)}$ and $P_0^{(2)}$,
respectively, we have $|U_n|=k_1$, $|W_n|=k_2$ if $n$ is sufficiently large. Note that by Jensen's formula, (see Lemma \ref{lem:bound on num of roots}), we have $k_1,k_2=O_{L,\l_0}(1)$.

We write
\[
F_n^{(1)}(z)=\frac{P_n^{(1)}(z)}{\prod_{u\in U_n}(z-u)},\quad
F_n^{(2)}(z)=\frac{P_n^{(2)}(z)}{\prod_{w\in W_n}(z-w)}.
\]
We observe that the limits
\[
\lim_{n\to\infty} F_n^{(1)}(z)=\frac{P_0^{(1)}(z)}{(z-\l_0)^{k_1}},\quad
\lim_{n\to\infty} F_n^{(2)}(z)=\frac{P_0^{(2)}(z)}{(z-\l_0)^{k_2}}
\]
are uniform on the boundary of $B(\l_0,\kappa)$ and therefore on all of $B(\l_0,\kappa)$
by the maximum modulus principle.
Since $F_n^{(2)}(z)$ is uniformly bounded away from $0$ on $B(\l_0,\kappa)$, we have
\[
\lim_{n\to\infty}\frac{F_n^{(1)}(z)}{F_n^{(2)}(z)}=R_0(z)(z-\l_0)^{k_2-k_1}
\]
uniformly on $B(\l_0,\kappa)$.

Write $K=2\max\{k_1,k_2\}$ and note that $K=O_{L,\l_0}(1)$.
We assume that the second alternative in the conclusion of the proposition does not hold, that is,
\begin{equation}\label{eq:assumpt-derivative}
\limsup_{n\to\infty} |\frac{d^j}{dz^j}R_n(\l_n)|<\infty 
\end{equation}
for all $j\le K$.
We show below that this implies $k_1=k_2$ and
\begin{equation}\label{eq:claim2}
\lim_{n\to\infty}\frac{\prod_{w\in W_n}(\l_n-w)}{\prod_{u\in U_n}(\l_n-u)}=1.
\end{equation}
This implies
\begin{align*}
R_0(\l_0)
=\lim_{n\to\infty}\frac{F_n^{(1)}(\l_n)}{F_n^{(2)}(\l_n)}
=\lim_{n\to\infty}R_n(\l_n)\frac{\prod_{w\in W_n}(\l_n-w)}{\prod_{u\in U_n}(\l_n-u)}
=\tau_0, 
\end{align*}
hence the first alternative of the conclusion holds.
Subject to the above two claims, this proves the proposition in the case $P_0^{(1)}\ne0$.

We turn to the proofs of the claims. Since $\tau_0\ne0$, the numbers $|R_n(\l_n)|$ are bounded away from $0$ and $\infty$. This together with \eqref{eq:assumpt-derivative} imply
\[
\limsup_{n\to\infty} |\frac{d^j}{dz^j}\log R_n(\l_n)|<\infty
\]
for all $j\le K$.
Furthermore, since $F_n^{(1)}$ and $F_n^{(2)}$ are uniformly convergent on $B(\l_0,\kappa)$,
we have
\[
\limsup_{n\to\infty} \Big|\frac{d^j}{dz^j}\log\Big(\frac{\prod_{w\in W_n}(z-w)}{\prod_{u\in U_n}(z-u)}\Big)\Big|_{z=\l_n}\Big|<\infty,
\]
hence
\[
\limsup_{n\to\infty} \Big|\sum_{w\in W_n}(\l_n-w)^{-j}-\sum_{u\in U_n}(\l_n-u)^{-j}\Big|<\infty
\]
for all $j\le K$ (recall that $\l_n\notin U_n\cup W_n$).

We define the multisets
\[
\wt U_n = \{\kappa_n (\l_n-u)^{-1}:u\in U_n\},\quad
\wt W_n = \{\kappa_n (\l_n-w)^{-1}:w\in W_n\},
\]
where $\kappa_n$ is a sequence converging to $0$ slowly enough so that $|u|,|w|\ge 1$ for all $u\in \wt U_n$
and $w\in \wt W_n$.
Therefore,
\[
\lim_{n\to\infty} \Big(\sum_{w\in \wt W_n}w^{j}-\sum_{u\in \wt U_n}u^{j}\Big)=0
\]
for all $j\le K$.

We can apply Lemma \ref{lm:power-sums} with $\wt U_n$ and $\wt W_n$ in the role of $U$ and $W$
if $n$ is sufficiently large.
For such an $n$, the lemma gives $|\wt U_n|=|\wt W_n|$, which implies $k_1=k_2$.
Furthermore, the lemma also gives
\[
\lim_{n\to\infty}\frac{\prod_{u\in \wt U_n} u}{\prod_{w\in \wt W_n}w}=1,
\]
which in turn yields \eqref{eq:claim2}.

It remains to handle the case $P_0^{(1)}=0$. Assume this is so, and fix $q\in \{\pm1\}\setminus\{-\tau_0\}$. For $n\ge1$ set $\wt R_n=R_n+q$, $\wt P_n^{(1)}=P_n^{(1)}+qP_n^{(2)}$ and $\wt P_n^{(2)}=P_n^{(2)}$. Also set $\wt \tau_0=\tau_0+q$, $\wt P_0^{(1)}=qP_0^{(2)}$ and $\wt P_0^{(2)}=P_0^{(2)}$. Note that all of the assumptions of the proposition are satisfied for the modified input $(\l_0,\wt \tau_0)$, $\wt R_1,\wt R_2,\ldots$ and $\l_1,\l_2,\ldots$, with $2L$ in place of $L$. Additionally we have $\wt R_n=\wt P_n^{(1)}/\wt P_n^{(2)}$ for all $n\ge1$, $\lim_n \wt P^{(1)}_n=\wt P_0^{(1)}$, $\lim_n \wt P^{(2)}_n=\wt P_0^{(2)}$ and $\wt P_0^{(1)}\ne0$. Thus, by the argument for the case $P_0^{(1)}\ne0$, it follows that at least one of the statements in the proposition holds for the modified input. Clearly this remains true for the original input, which completes the proof.
\end{proof}

\subsection{}\label{proof of second prop}

The purpose of this section is to prove Proposition \ref{pr:singular-bouquet}. The following simple lemma will be needed. Recall the self-affine measures $\nu_{R,\l,K}$ from Section \ref{sc:lb on h(R,lam,M)}.

\begin{lem}\label{lm:sa-dim}
Let $\l\in(0,1)$, $R\in\cR_L$ and $K\in\Z_{>0}$.
Assume $\l$ is not a pole of $R$.
Then
\[
\dim\nu_{R,\l,K}\le\frac{h(R,\l,K)}{\log \l^{-1}}.
\]
\end{lem}

\begin{proof}
Recall the affine IFS \eqref{SA IFS}, and note that for every unit vector $x\in\R^K$
\begin{equation}\label{conformality of IFS}
|\Theta(\l,K)^nx|=\l^{n+o(n)} \text{ as }n\rightarrow\infty.
\end{equation}
Given $r>0$, denote by $\cC_r$ the partition of $\R^K$ into $K$-cubes of side length $r$. 
By \eqref{conformality of IFS} and \cite{FH-dimension}*{Theorem 2.8} it follows that $\nu_{R,\l,K}$ is exact dimensional.
Thus,
\[
\dim\nu_{R,\l,K} = \lim_n\frac{H(\nu_{R,\l,K};\cC_{\l^n})}{n\log\l^{-1}}.
\]
It also follows easily from \eqref{conformality of IFS} that
\[
\lim_n\frac{H(\nu_{R,\l,K};\cC_{\l^n})}{n\log\l^{-1}} = \lim_n\frac{H(\nu_{R,\l,K}^{(n)};\cC_{\l^n})}{n\log\l^{-1}}.
\]
Additionally, we have
\[
\lim_n\frac{H(\nu_{R,\l,K}^{(n)};\cC_{\l^n})}{n\log\l^{-1}} \le \lim_n\frac{H(\nu_{R,\l,K}^{(n)})}{n\log\l^{-1}} = \frac{h(R,\l,K)}{\log \l^{-1}},
\]
which completes the proof of the lemma.
\end{proof}

\begin{proof}[Proof of Proposition \ref{pr:singular-bouquet}]
We assume as we may that
\[
\limsup_{n\to\infty}|d^j/dX^j R_n(\l_n)|<\infty
\]
for all $j<K-1$.
For each $n\in\Z_{>0}$, we let $\Delta^{(n)}\in \R^{K\times K}$ be a diagonal matrix such that
$\Delta^{(n)}_{1,1}=1$ and $\Delta_{K,K}^{(n)}=d^{K-1}/dX^{K-1}R_n(\l_n)$ for each $n$ and the
ratio of consecutive entries on the diagonal converge to $\infty$.

By Lemma \ref{lm:sa-dim}, we have
\[
\dim\nu_{R_n,\l_n,K}\le\frac{h(R_n,\l_n,K)}{\log \l_n^{-1}}.
\]
We note that
$(\Delta^{(n)})^{-1}\nu_{R_n,\l_n,K}$
is a self-affine measure associated to the probability vector $p$ and the IFS
\begin{equation}\label{eq:saIFS}
\{x\mapsto(\Delta^{(n)})^{-1}\Theta(\l_n,K)\Delta^{(n)}x+(\Delta^{(n)})^{-1}v_j(R_n,\l_n,K):j=1,\ldots,m\}.
\end{equation}
By the choice of $\Delta^{(n)}$ and the definitions of $\Theta(\l_n,K)$, $v_j(R_n,\l_n,K)$, we have
\begin{align*}
\lim_{n\to \infty} (\Delta^{(n)})^{-1}\Theta(\l_n,K)\Delta^{(n)}=&\l_0\cdot\Id,\\
\lim_{n\to\infty} (\Delta^{(n)})^{-1} v_j(R_n,\l_n,K)=&(T_j(1,\tau_0),0,0,\ldots,0,T_j(0,1)).
\end{align*}

Since $x\mapsto\Delta^{(n)}x$ is bi-Lipschitz,
\[
\dim(\Delta^{(n)})^{-1}\nu_{R_n,\l_n,K}=\dim\nu_{R_n,\l_n,K}.
\]
We write $\nu$ for the self-similar measure associated to $p$ and the IFS
\[
\{x\mapsto(\l_0 x+(T_j(1,\tau_0),0,\ldots,0,T_j(0,1)):j=1,\ldots,m\}.
\]
Applying Lemma \ref{lm:sa-semi-cont} to the IFSs \eqref{eq:saIFS}, we get
\[
\dim\nu\le\liminf_{n\to\infty}\frac{h(R_n,\l_n,K)}{\log \l_0^{-1}}.
\]

For $\tau\in\R$, we write $\pi:\R^K\to\R$ for the projection $\pi(x)=x_1+(\tau-\tau_0)x_K$.
We observe $\mu_{\l_0,\tau}=\pi\nu$, and hence $\dim\mu_{\l_0,\tau}\le\dim\nu$, which completes the proof.
\end{proof}

\section{A bound for the Mahler measure}\label{sc:BV1}

The purpose of this section is the proof of Theorem \ref{th:BV1}.
Throughout this section $T_1(Y_1,Y_2)$,  $T_2(Y_1,Y_2)$ and $T_3(Y_1,Y_2)$ are the linear forms $0$, $Y_1$ and $Y_2$ respectively.
The following result can be deduced from the arguments in \cite{BV-entropy}, see \cite{Var-Bernoulli}*{Theorem 9}.

\begin{thm}\label{th:entropy-Bernoulli}
Let $\wt \xi_0,\wt\xi_1,\ldots$ be a sequence of independent random variables taking the values
$1$ and $2$ with equal probabilities.
Then for any $h\in(0,\log 2)$ there is a number $C(h)$, such that
\[
\lim_{n\to\infty} \frac{H(\sum_{j=0}^{n-1}T_{\wt\xi_j}(1,\tau)\l^j)}{n}
=\inf_{n\ge1}\frac{H(\sum_{j=0}^{n-1}T_{\wt\xi_j}(1,\tau)\l^j)}{n}\ge h
\]
for any $(\l,\tau)\in(0,1)\times\R$ with $M(\l)>C(h)$.
\end{thm}

\begin{rmk}
Two remarks are in order.
First, in the quoted references, the coefficients that are denoted by $T_{\wt\xi_j}(1,\tau)$ above,
take the values $\pm1$ instead of $0$ and $1$.
However, the two settings can be transformed into each other by appropriate scaling and translation,
which does not change entropy.
Second, in \cite{Var-Bernoulli}*{Theorem 9}, it is assumed that $\l$ is algebraic.
However, with our convention that $M(\l)=\infty$ for transcendental numbers, this is not needed.
Indeed, the random variable $\sum_{j=0}^{n-1}T_{\wt\xi_j}(1,\tau)\l^j$ takes $2^n$ distinct values
with equal probability if $\l$ is transcendental.
\end{rmk}

We use this to deduce the following.

\begin{lem}\label{lm:entropy-Bernoulli}
With the notation of Theorem \ref{th:entropy-Bernoulli}, let $h\in(0,\log 2)$ and let $(\l,\tau)\in(0,1)\times\R$ be with $M(\l)>C(h)$.
Let $A\subset\{0,\ldots,n-1\}$.
Then
\[
H\Big(\sum_{j\in A}T_{\wt\xi_j}(1,\tau)\l^j\Big)\ge hn-\log(2)(n- |A|).
\]
\end{lem}

\begin{proof}
We observe that $H(\sum_{j\in A}T_{\wt\xi_j}(1,\tau)\l^j+t)$ is independent of the value of $t$.
Therefore,
\begin{align*}
H\Big(\sum_{j\in A}T_{\wt\xi_j}(1,\tau)\l^j\Big)
=&H\Big(\sum_{j=0}^{n-1}T_{\wt\xi_j}(1,\tau)\l^j\Big|\wt\xi_j:j\notin A\Big)\\
\ge& H\Big(\sum_{j=0}^{n-1}T_{\wt\xi_j}(1,\tau)\l^j\Big)-\log(2)(n- |A|).
\end{align*}
The claim now follows from Theorem \ref{th:entropy-Bernoulli}.
\end{proof}

\begin{proof}[Proof of Theorem \ref{th:BV1}]
Let $h=\log 2-\e$.
Let $M=C(h)$, where $C(h)$ is as in Lemma \ref{lm:entropy-Bernoulli}.
Let $(\l,\tau)\in(0,1)\times\R$, and suppose $M(\l)>M$.
We show that $h(\l,\tau)\ge h-(\log 2)/3=(2\log 2)/3-\e$, which proves the theorem.

Fix some $n\in\Z_{>0}$.
We write $\nu$ for the distribution of the random variable
\[
\sum_{j=0}^{n-1}T_{\xi_j}(1,\tau)\l^j,
\]
where $\xi_0,\xi_1,\ldots$ are i.i.d. random variables taking the values $1$, $2$ and $3$ with equal probabilities. For $A\subset\{0,\ldots,n-1\}$, we write $\nu_A$ for the distribution of the random variable
\[
\sum_{j\in A} T_{\wt \xi_j}(1,\tau)\l^j+\sum_{j\in\{0,\ldots,n-1\}\backslash A} T_3(1,\tau)\l^j.
\]

We note that
\[
H(\nu_A)=H(\sum_{j\in A}T_{\wt\xi_j}(1,\tau)\l^j)\ge hn-\log(2)(n- |A|),
\]
and
\[
\nu=\sum_{A\subset \{0,\ldots,n-1\}} (2/3)^{|A|}(1/3)^{n-|A|}\nu_A.
\]

By concavity of entropy, we get
\[
H(\nu)\ge \sum_{A\subset \{0,\ldots,n-1\}} (2/3)^{|A|}(1/3)^{n-|A|} (hn-\log(2)(n- |A|))
=(h-\frac{\log 2}{3})n,
\]
as required.
\end{proof}

\section{Proof of Theorems \ref{th:EO-general} and \ref{th:large-lambda}}\label{sc:EO-proof}

Some of the arguments in this section are based on the paper \cite{Var-Bernoulli}.
These are discussed in Section \ref{sc:EO int trans} of the appendix in the simpler setting of
homogeneous IFS's with rational translations.
The reader not familiar with \cite{Var-Bernoulli} may find it helpful to read that part of the appendix before
this section, but this section can also be read independently.

We give the proofs of Theorems \ref{th:EO-general} and \ref{th:large-lambda} simultaneously for they
are almost identical.

In the setting of Theorem \ref{th:EO-general}, we assume that $(\l_0,\tau_0)\in(0,1)\times\R$
is such that the IFS \eqref{eq:IFS} contains no exact overlaps,
and we also assume that the answer to Question \ref{q:BV1-general} is affirmative.
In the setting of Theorem \ref{th:large-lambda}, we assume that $(\l_0,\tau_0)\in(2^{-2/3},1)\times\R$
is such that the IFS \eqref{eq:3maps} contains no exact overlaps.

In both settings, we suppose to the contrary that
\begin{equation}\label{suppose dim <}
\dim \mu_{\l_0,\tau_0}<\min\{1,H(p)/\log\l_0^{-1}\},
\end{equation}
where $p=(1/3,1/3,1/3)$ in the setting of Theorem  \ref{th:large-lambda}.
Conjecture \ref{cn:EO} was proved in \cite{Rap-EO} in the case when $\l_0$ is algebraic.
Thus our assumption that there are no exact overlaps and  \eqref{suppose dim <}
imply that $\l_0$ is transcendental. Moreover, since there are no exact overlaps, by  \eqref{suppose dim <} and by Theorem \ref{th:integral-translations}, it follows that $\tau_0\ne0$.

The proof will use the parameters $K,M,N\in\Z_{>0}$ and $c,h_0\in\R_{>0}$.
We set $h_0$ and $M$ depending only on $\l_0$, $\tau_0$ and the IFS.
We set $K$ sufficiently large depending only on $\l_0$, $\tau_0$, $h_0$ and the IFS.
We set $c$ sufficiently small depending only on $\l_0$, $\tau_0$, $h_0$, $K$, $M$ and the IFS.
Finally, we set $N$ sufficiently large depending only on $\l_0$, $\tau_0$, $h_0$, $K$, $M$, $c$ and the IFS.

We first apply Theorem \ref{th:BV2-general}, which yields that there is $n\ge N$ and $(\eta,\s)\in(0,1)\times\R$
such that the conclusion of Theorem \ref{th:BV2-general} holds.
We have, in particular, that $|\l_0-\eta|<\exp(-n^2)$ and $\eta$ is an algebraic number of degree at most $2n$. Since $\l_0$ is transcendental we have $\l_0\ne\eta$.
Furthermore, in the setting of Theorem \ref{th:EO-general}, we use that the answer to Question \ref{q:BV1-general} is
affirmative, which we assumed, and we conclude that there is a number $M$ depending only on $\l_0$, $\tau_0$
and importantly not on $N$ or $n$ such that $M(\eta)\le M$.
In the setting of Theorem \ref{th:large-lambda}, we use that $\l_0>2^{-2/3}$, hence $h(\eta,\s)<\log\l_0^{-1}\le (2\log 2)/3-\e$,
where $\e>0$ can be set depending only on $\l_0$.
Now we can invoke Theorem \ref{th:BV1}, which yields that $M(\eta)\le M$ for some $M$ depending only on $\l_0$
and not on $N$ or $n$ also in this setting.
From this point on, the proofs of Theorems \ref{th:EO-general} and \ref{th:large-lambda} will be identical.

The next step will be the application of the following result of Hochman \cite{Hoc-self-similar}. Given $m\in\Z_{>0}$ recall the notation $\mu_{\l_0,\tau_0}^{(m)}$ from Section \ref{sc:general-setting}.

\begin{thm}\label{th:Hoc}
Let notation and assumptions be as above.
In particular, we assume $\dim\mu_{\l_0,\tau_0}<1$.
For $r>0$, let $\cD_r$ be a partition of $\R$ to intervals of length $r$.
Then for all $c<\l_0$, we have
\[
\lim_{m\to\infty} \frac{1}{m}H(\mu_{\l_0,\tau_0}^{(m)};\cD_{c^m})=\dim\mu_{\l_0,\tau_0}\log\l_0^{-1}.
\]
\end{thm} 

\begin{rmk}
We note that in \cite{Hoc-self-similar}, the notation $\cD_m$ stands for what we denote by $\cD_{2^{-m}}$.
The conclusion in \cite{Hoc-self-similar} using our notation is formulated as
\[
\lim_{m\to\infty} \frac{1}{m}H(\mu_{\l_0,\tau_0}^{(m)};
\cD_{2^{-q\lfloor m\log\l_0^{-1}\rfloor}}|\cD_{2^{-\lfloor m\log\l_0^{-1}\rfloor}})=0
\]
for an arbitrary $q>1$.
Since
\[
\lim_{r\to0}\frac{1}{\log r^{-1}} H(\mu;\cD_{r})=\dim\mu
\]
for all exact dimensional measures $\mu$, the two forms are clearly equivalent.
\end{rmk}

We fix a number $h_0$ satisfying
\[
\log\l_0^{-1}\dim\mu_{\l_0,\tau_0}<h_0<\min\{\log\l_0^{-1},H(p)\}.
\]
This is possible, thanks to our assumption \eqref{suppose dim <}, and can be done
depending only on $\l_0$, $\tau_0$ and the IFS.
We fix an integer $n'$ in such a way that $M^{-4n'}<|\l_0-\eta|<M^{-3n'}$.
Since $|\l_0-\eta|\le \exp(-n^2)$, this implies $n'> c_Mn^2$ for some constant $c_M$ depending only on $M$.
We apply Theorem \ref{th:Hoc} with the number $c$ we mentioned at the beginning of the proof, and since
$N$ is suitably large, the theorem yields
\[
H(\mu_{\l_0,\tau_0}^{(n')};\cD_{c^{n'}}) \le h_0 n'.
\]

We write $\wt\cD$ for the partition of $\Omega$ that is the pullback
of $\cD_{c^{n'}}$ under the map
\[
\omega\to\sum_{j=0}^{n'-1} T_{\omega_j}(1,\tau_0)\l_0^{j}.
\]
We therefore, have
\[
H(\b;\wt\cD)\le h_0 n'.
\]

We write $\wt \cQ\subset\Z[X,Y_1,Y_2]$ for the collection of polynomials
\begin{equation}\label{eq:omega-poly}
\sum_{j=0}^{n'-1} T_{\omega_j}(Y_1,Y_2)X^j-\sum_{j=0}^{n'-1} T_{\omega_j'}(Y_1,Y_2)X^j,
\end{equation}
where $\omega,\omega'$ ranges over pairs of elements of $\Omega$ that are in the same
$\wt \cD$ atom.
We note that by our assumption on the IFS, the polynomial \eqref{eq:omega-poly} is non-zero for any choice of $\omega,\omega'$ with distinct $n'$-prefix.
In particular, $\wt\cQ\neq\{0\}$, as $h_0<H(p)$.

We show that there is $Q=P_1(X)Y_1+P_2(X)Y_2\in\wt \cQ$ such that
$P_2(\eta)\neq 0$.
Indeed, suppose to the contrary that such a choice is not possible.
Let $Q=P_1(X)Y_1+P_2(X)Y_2\in\wt \cQ$ be arbitrary.
Then $|P_1(\l_0)|\le c^{n'}+|\tau_0P_2(\l_0)|$ by definition.
This, $P_2(\eta)=0$ and the mean value theorem yield
\[
|P_1(\eta)|\le c^{n'}+(1+|\tau_0|)L(n')^2|\eta-\l_0|.
\]
Assuming $c<M^{-3}$, as we may, and using $|\eta-\l_0|<M^{-3n'}$,
we have $|P_1(\eta)|<M^{-2n'}$ since $N$ is sufficiently large in terms of $L$ and $\tau_0$.
Since $n'>c_Mn^2$, this implies $|P_1(\eta)|<(Ln')^{-2n}M(\eta)^{-n'}$ using again that $N$ is sufficiently large in terms of $L$
and $M$.
By Lemma \ref{lem:lb on val of poly}, this implies $P_1(\eta)=0$.
We can now conclude that $Q(\eta,1,\tau)=0$ for all $\tau$, hence
$h(\{\eta\}\times\R)\le h_0$, which contradicts Corollary \ref{cr:no-bad-curves}
since $N$ is sufficiently large and hence $|\l_0-\eta|$
is sufficiently small in a manner depending only on $\l_0,\tau_0$ and the IFS.
 
We fix $Q=P_1(X)Y_1+ P_2(X)Y_2\in\wt\cQ$ such that
$P_2(\eta)\neq 0$, and set $R=-P_1/P_2$.
We estimate $|R(\eta)-\tau_0|$.
We can write
\begin{align*}
|R(\eta)-\tau_0|=&\frac{|P_1(\eta)+\tau_0P_2(\eta)|}{|P_2(\eta)|}\\
\le&\frac{c^{n'}+|P_1(\eta)-P_1(\l_0)|+|\tau_0||P_2(\eta)-P_2(\l_0)|}{|P_2(\eta)|}\\
\le&\frac{c^{n'}+(1+|\tau_0|)L(n')^2|\eta-\l_0|}{|P_2(\eta)|}.
\end{align*}
By Lemma \ref{lem:lb on val of poly} and $P_2(\eta)\neq 0$, we have $|P_2(\eta)|\ge (Ln')^{-2n} M^{-n'}$.
Using the assumptions $c<M^{-3}$ and $|\eta-\l_0|<M^{-3n'}$, we get
\[
|R(\eta)-\tau_0|\le (Ln')^{2n} (2+|\tau_0|)L(n')^2 M^{-2n'}.
\]
This means that $|R(\eta)-\tau_0|<K^{-1}$ since $N$ is sufficiently large in a manner depending only on $\l_0$,
$\tau_0$, $M$, $L$ and $K$.
Here $K$ is the parameter we mentioned at the beginning of the proof.

We show that $\eta$ is a root of order at least $K$ of
$P_1\wt P_2-P_2\wt P_1$ for any $\wt P_1,\wt P_2$ such that $\wt P_1(X)Y_1+\wt P_2 Y_2\in\wt\cQ$.
To this end, we first write
\begin{align*}
|P_1(\l_0)\wt P_2(\l_0)-P_2(\l_0)\wt P_1(\l_0)|
\le&|(P_1(\l_0)+\tau_0P_2(\l_0))\wt P_2(\l_0)|\\
&+|(\wt P_1(\l_0)+\tau_0\wt P_2(\l_0)) P_2(\l_0)|\\
\le& L n' c^{n'}.
\end{align*}
As $N$ is sufficiently large depending on $M$ and $L$, Lemma \ref{lm:Dimitrov-precise} is applicable for the polynomial
$P_1\wt P_2-P_2\wt P_1$ with $\a=2$.
Since $c$ is sufficiently small depending on $L$, $M$ and $K$, we can conclude that $\eta$ is indeed a root
of multiplicity at least $K$ of $P_1\wt P_2-P_2\wt P_1$.

Dividing by $P_2$, we see that $\eta$ is also a zero of $\wt P_1+R\wt P_2$ of order at least $K$.
By the definition of $\wt \cQ$, this means that
\[
\frac{d^a}{dX^a}\sum_{j=0}^{n'-1}T_{\o_j}(1,R(X))X^j\Big|_{X=\eta}
=\frac{d^a}{dX^a}\sum_{j=0}^{n'-1}T_{\o'_j}(1,R(X))X^j\Big|_{X=\eta}
\]
for $a=0,\ldots,K-1$ for all $\o,\o'\in\Omega$ that are in the same atom of $\wt\cD$.
From this, we conclude
\[
H(B_{R,\eta,K}^{(n')})\le H(\b;\wt\cD)\le h_0 n',
\]
hence $h(R,\eta,K)\le h_0$.
If we choose $K$ sufficiently large in a manner depending on $\l_0$, $\tau_0$, the IFS and $h_0$,
then we reach a contradiction with
Proposition \ref{pr:no-bad-bouquet} (recall that $\tau_0\ne 0$).
This completes the proofs of Theorems \ref{th:EO-general} and \ref{th:large-lambda}.

\appendix

\section{\label{sec:int trans case}The case of integral translations}

The purpose of this appendix is to prove Conjecture $\ref{cn:EO}$ for homogeneous systems with rational translations.
Rescaling the IFS, we may assume that the translations are integral.

Let $\l\in(0,1)$, let $a_1,\ldots,a_m\in\Z$ be distinct integers, and let $(p_1,\ldots,p_m)$ be a positive probability vector.
We write $\mu_\l$ for the self-similar measure associated to the IFS
\begin{equation}\label{eq:IFS-integral}
\{x\mapsto \l x+a_j:j=1,\ldots,m\}
\end{equation}
with probability weights $p_1,\ldots,p_m$.
Note that \eqref{eq:IFS-integral} is the special case of \eqref{eq:IFS} with $\tau=0$,
except that the non-degeneracy assumption that $a_1,\ldots,a_m$ are distinct is slightly stronger than what we assumed for
\eqref{eq:IFS}.

As before, let $\xi_0,\xi_1,\ldots$ be i.i.d. random variables with $\P\{\xi_0=j\}=p_j$ for $1\le j \le m$.
Then $\mu_\l$ is the law of the random variable
\[
\sum_{j=0}^{\infty}a_{\xi_j}\l^j.
\]
We denote by $\mu_\l^{(n)}$ the law of the truncated sum
\[
\sum_{j=0}^{n-1}a_{\xi_j}\l^j.
\]
We set,
\[
s(\lambda)=\min\Big\{1,\frac{H(p)}{\log\lambda^{-1}}\Big\},\qquad
h(\lambda)=\lim_{n\to\infty}\frac{1}{n}H(\mu_{\lambda}^{(n)}).
\]
Note that the IFS \eqref{eq:IFS-integral} contains exact overlaps if and only if $h_\l<H(p)$.

We formulate the main result of the appendix as follows.

\begin{thm}\label{th:integral-translations}
With the above notation, we have
\[
\dim\mu_\l=\min\Big\{1,\frac{h(\l)}{\log\l^{-1}}\Big\}.
\]
In particular, if the IFS \eqref{eq:IFS-integral} contains no exact overlaps, we have
$\dim\mu_\l=s(\l)$.
\end{thm}

For algebraic $\l$, this is a result of Hochman \cite{Hoc-self-similar}.
For transcendental $\l$, the proof requires only some minor modification of the arguments in \cite{Var-Bernoulli} and its references,
which we will discuss.

\subsection{\label{sc:gen of BV2 int trans}Algebraic approximation of parameters with dimension drop}

We introduce some notation.
We write
\[
D=\{a_{i}-a_{j}:1\le i,j\le m\},\qquad\text{ and }\qquad L_{0}=\max(D).
\]
For $n\ge1$ and $\alpha>0$, we write
\[
E_{\alpha}^{(n)}=\{\eta\in(0,1):\dim\mu_{\eta}<\alpha\text{ and }P(\eta)=0\text{ for some }0\ne P\in\mathcal{P}_{L_0}^{(n)}\}.
\]
The purpose of this section is to prove the following theorem, which generalizes the main result of \cite{BV-transcendent}.

\begin{thm}
\label{thm:gen of BV2 main thm}
Let $\lambda\in(0,1)$ be with $\dim\mu_{\lambda}<s(\l)$.
Then for every $\e>0$ and $N\ge1$, there exists $n\ge N$ and
$\eta\in E_{\dim\mu_{\lambda}+\e}^{(n)}$ such that $|\lambda-\eta|\le\exp(-n^{1/\e})$.
\end{thm}

The proof of this result is adapted from \cite{BV-transcendent}.
Given $r>0$ and a random variable $A$ with distribution $\nu$, recall the notation $H(A;r)$ from \eqref{avg ent scl r} and that we write $H(\nu;r)$ in place of $H(A;r)$.
We also write
$H(\nu;r_1|r_2)=H(\nu;r_1)-H(\nu;r_2)$
for $r_1,r_2\in\R_{>0}$.

A significant proportion of the proof of Theorem \ref{thm:gen of BV2 main thm} is encapsulated in the
following result, which we quote from \cite{BV-transcendent}.

\begin{prp}
\label{prop:inequality involving =00007BK_j=00007D}
For all $\lambda\in(0,1)$ and $\a>0$, there exists $C>1$ such
that the following holds.
Let $N\ge1$, $\{n_{j}\}_{j=1}^{N}\subset\Z_{>0}$
and $\{K_{j}\}_{j=1}^{N}\subset[10,\infty)$ be given.
Suppose that
$\lambda^{-n_{1}}\ge\max\{2,\lambda^{-2}\}$ and,
\begin{enumerate}
\item $n_{j+1}\ge K_{j}n_{j}$ for all $1\le j<N$;
\item $H(\mu_{\lambda};r|2r)\le1-\alpha$ for all $r>0$;
\item $H(\mu_{\lambda}^{(n_{j})};\lambda^{K_{j}n_{j}}|\lambda^{10n_{j}})\ge\alpha n_{j}$
for all $1\le j\le N$;
\item $n_{j}\ge C(\log K_{j})^{2}$ for all $1\le j\le N$.
\end{enumerate}
Then,
\[
\sum_{j=1}^{N}\frac{1}{\log K_{j}\log\log K_{j}}\le C\left(1+\frac{1}{n_{1}}\sum_{j=1}^{N}\log K_{j}\right).
\]
\end{prp}

The main idea of the proof of this result is the following.
We rescale the inequalities in condition (3) in a suitable way.
The rescaled measures are all convolution factors of $\mu_\l$.
We use a result on how the entropy of measures grow under convolution (see \cite{BV-transcendent}*{Theorem 8} originally proved in \cite{Var-ac}) along with the
other conditions in the proposition to derive the inequality in the conclusion, whose two sides are
essentially lower and upper bounds for the entropy of $\mu_\l$ between some scales.

Proposition \ref{prop:inequality involving =00007BK_j=00007D} is proved in \cite{BV-transcendent}*{Proposition 30} in the setting of Bernoulli convolutions, that is, in the case $m=2$, $a_1=-1$, $a_2=1$ and $p_1=p_2=1/2$.
The proof uses only the scaling properties of the measures $\mu_\l^{(n)}$ and how they decompose as convolution products, which
hold also in the general case.
For this reason, we do not repeat the proof here.
The reader may find more details about this in Section \ref{sc:prf of apx thm}, and we refer to \cite{BV-transcendent} for the full details.

The proof of Theorem \ref{thm:gen of BV2 main thm} is a proof by contradiction, and it is based
on the following strategy.
We will show that under an indirect hypothesis, the parameters
can be chosen in Proposition \ref{prop:inequality involving =00007BK_j=00007D} in such a way that the hypotheses
of the proposition hold, and the conclusion leads to a contradiction.

We begin with condition $(2)$ of the proposition.
This will be satisfied using the assumption $\dim\mu_\l<1$ and the following result.

\begin{lem}
\label{lem:bd of ent of single dig}
Suppose that $\dim\mu_{\lambda}<1$.
Then there exists
$\alpha>0$ (depending on $\l$, $a_1,\ldots,a_m$ and $p_1,\ldots,p_m$) such that,
\[
H(\mu_{\lambda};r|2r)<1-\alpha\text{ for all }r>0.
\]
\end{lem}

This is proved in  \cite{BV-transcendent}*{Lemma 13} in the setting of Bernoulli convolutions.
Again, the proof depends only on the above mentioned properties of $\mu_\l$, which hold generally. 
The dependence of $\a$ on the parameters $a_1,\ldots,a_m$ and $p_1,\ldots,p_m$ is only through
the difference $1-\dim\mu_\l$, otherwise these parameters play no role in the proof.

We now move on to consider the other conditions in Proposition \ref{prop:inequality involving =00007BK_j=00007D}.
These will be satisfied using the indirect hypothesis.
The first step is the following result which shows that the only way the entropy of $\mu_\l^{(n)}$ can be small
on a suitably chosen scale is if $\l$ is approximated very closely by an algebraic number $\eta$, and moreover
we can control $H(\mu_\eta^{(n)})$.

\begin{prp}
\label{prop:initial approx integer trans}For every $\e>0$
there exists $C=C(L_0,\e)>1$ such that the following holds for
all $n\ge N(L_0,\e,C)\ge1$. Let $0<r<n^{-Cn}$ and $\e\le\lambda\le1-\e$
be given, and suppose that $\frac{1}{n}H(\mu_{\lambda}^{(n)};r)<H(p)$.
Then there exists $0<\eta<1$, which is a root of a nonzero polynomial
in $\mathcal{P}_{L_0}^{(n)}$, such that $|\lambda-\eta|<r^{1/C}$ and
\[
H(\mu_{\eta}^{(n)})\le H(\mu_{\lambda}^{(n)};r).
\]
\end{prp}

\begin{proof}
Let $\e>0$, let $C>1$ be large with respect to $L_0$ and $\e$,
let $n\ge1$ be large with respect to $C$, and let $r$ and $\lambda$
be as in the statement of the proposition.

Let $0\le s\le1$ be with,
\begin{equation}
H\left(\left\lfloor r^{-1}\sum_{k=0}^{n-1}a_{\xi_k}\lambda^{k}+s\right\rfloor \right)\le H(\mu_{\lambda}^{(n)};r)<nH(p).
\label{eq:trans with small ent}
\end{equation}
Let $\mathcal{A}$ be the set of all nonzero $\sum_{k=0}^{n-1}d_{k}X^{k}=P(X)\in\Z[X]$
with $|P(\lambda)|\le r$ and $d_{k}\in D$ for each $0\le k<n$.
Then $\mathcal{A}\subset\mathcal{P}_{L_0}^{(n)}\setminus\{0\}$, and
from \eqref{eq:trans with small ent} it follows that $\mathcal{A}$
is nonempty.

Given $P\in\mathcal{A}$ it follows from Lemma \ref{lem:close root single poly}
that there exists $\eta_{P}\in\C$ with $P(\eta_{P})=0$ and,
\[
|\eta_{P}-\lambda|\le(2^{n}\e^{-n}r)^{C^{-1/4}/\max\{\log L_{0},3\}}.
\]
From $r<n^{-Cn}$, since $C$ is large with respect to $L_{0}$, and
since $n$ is large with respect to $C$, it follows that we may assume
$|\lambda-\eta_{P}|<r^{C^{-1/2}}$.

For $Q,P\in\mathcal{A}$,
\[
|\eta_{P}-\eta_{Q}|\le|\eta_{P}-\lambda|+|\lambda-\eta_{Q}|\le2r^{C^{-1/2}}<2n^{-C^{1/2}n}.
\]
Thus, by Lemma \ref{lem:dist between distinct roots} and by assuming
that $C$ is large enough, it follows that $\eta_{P}=\eta_{Q}$. Write
$\eta$ for this common value, then $P(\eta)=0$ for all $P\in\mathcal{A}$.
From this, \eqref{eq:trans with small ent} and by the definition
of $\mathcal{A}$,
\[
H(\mu_{\eta}^{(n)})\le H(\mu_{\lambda}^{(n)};r).
\]

Since $\lambda\in\R$ we have $|\overline{\eta}-\lambda|=|\eta-\lambda|$,
and so $|\eta-\overline{\eta}|\le2n^{-C^{1/2}n}$. For $P\in\mathcal{A}$
we clearly have $P(\overline{\eta})=0$. Thus, another application
of Lemma \ref{lem:dist between distinct roots} gives $\eta=\overline{\eta}$.
Since $\lambda\in(\e,1-\e)$ we may assume $\eta\in(0,1)$,
which completes the proof of the proposition.
\end{proof}

In the course of the proof of Theorem \ref{thm:gen of BV2 main thm}, we will need to show that $\eta$
obtained above belongs to $E_{\dim\mu_\l+\e}^{(n)}$.
To this end, we need to convert our bound on $H(\mu_\eta^{(n)})$ to a bound on
the dimension of $\mu_\eta$.
This is achieved in the next lemma.

\begin{lem}
\label{lem:ub on dim}
We have
\[
\dim\mu_{\lambda}\le\frac{H(\mu_{\lambda}^{(n)})}{n\log\l^{-1}}\text{ for all }n\ge1.
\]
\end{lem}

\begin{proof}
By sub-additivity, 
\begin{equation}
\lim_n\frac{1}{n}H(\mu_{\lambda}^{(n)})=\inf_n\frac{1}{n}H(\mu_{\lambda}^{(n)}).\label{eq:lim=00003Dinf}
\end{equation}
It is easy to see that for every $n\ge1$,
\[
H(\mu_{\lambda};\l^n)=H(\mu_{\lambda}^{(n)};\l^n)+O(1).
\]
Thus,
\[
\dim\mu_{\lambda}=
\lim_n\frac{H(\mu_{\lambda};\l^n)}{n\log\l^{-1}}=
\lim_n\frac{H(\mu_{\lambda}^{(n)};\l^n)}{n\log\l^{-1}}
\le\lim_n\frac{H(\mu_{\lambda}^{(n)})}{n\log\l^{-1}},
\]
and the lemma follows by \eqref{eq:lim=00003Dinf}.
\end{proof}

Next, we consider the situation when the approximating parameter $\eta$ described in the
previous proposition exists.
In this case, we show that we can find a larger value of $n$ for which $\mu_{\lambda}^{(n)}$
has significant entropy on a suitable scale.
This is achieved in the next result.
We will apply it choosing $n$ to be as large as possible subject to the constraint $|\lambda-\eta|<n^{-Cn}$.
Therefore, how large $n$ will be, ultimately depends on the approximation $|\lambda-\eta|$, which we
will control using the indirect hypothesis in the proof of Theorem \ref{thm:gen of BV2 main thm}.

\begin{prp}
\label{prop:cond for large ent int trans}For every $\e>0$
there exists $C=C(L_0,\e)>1$ such that the following holds for
all $n\ge N(L_0,\e,C)\ge1$. Let $\e\le\lambda\le1-\e$
and suppose that there exists $\eta\in\C$, which is a root
of a nonzero polynomial in $\mathcal{P}_{L_0}^{(n)}$, such that $|\lambda-\eta|<n^{-Cn}$.
Then $\frac{1}{n}H(\mu_{\lambda}^{(n)};r)=H(p)$ for all $r\le|\lambda-\eta|^{C}$.
\end{prp}

\begin{proof}
Let $\e>0$, let $C>1$ be large with respect to $L_0$ and $\e$,
let $n\ge1$ be large with respect to $C$, and let $\lambda$ and
$\eta$ be as in the statement of the proposition.

Suppose to the contrary
that there exists $0<r\le|\lambda-\eta|^{C}$ with $\frac{1}{n}H(\mu_{\lambda}^{(n)};r)<H(p)$.
By Proposition \ref{prop:initial approx integer trans}, there exists
$\eta'\in(0,1)$, which is a root of a nonzero polynomial in $\mathcal{P}_{L_0}^{(n)}$,
such that $|\lambda-\eta'|<r^{1/C}\le|\lambda-\eta|$. In particular
$\eta\ne\eta'$ and
\[
|\eta-\eta'|\le|\eta-\lambda|+|\lambda-\eta'|\le2n^{-Cn}.
\]
However, as $C$ is assumed to be large enough,
this contradicts Lemma \ref{lem:dist between distinct roots}, which
completes the proof of the proposition.
\end{proof}

\begin{proof}[Proof of Theorem \ref{thm:gen of BV2 main thm}]
Let $\lambda\in(0,1)$ be with $\dim\mu_{\lambda}<s(\l)$.
Suppose to the contrary that there exists
\begin{equation}
0<\e<\frac{1}{3}\min\{1,\log\lambda^{-1}\}(s(\l)-\dim\mu_{\lambda}),\label{eq:restriction on epsilon}
\end{equation}
such that
\begin{equation}
|\lambda-\eta|>\exp(-n^{\e^{-1}})\text{ for all \ensuremath{n\ge\e^{-1}} and }\eta\in E_{\dim\mu_{\lambda}+3\e}^{(n)}.\label{eq:cont assumption}
\end{equation}

Let $C>1$ be large with respect to $a_1,\ldots,a_m$, $\lambda$ and $\e$,
and let $n_{0}\ge1$ be large with respect to $C$. We shall next
define by induction a sequence $n_{0}<n_{1}<\ldots$ of positive integers.

Let $j\ge0$ and suppose that $n_{j}$ has been chosen. Write $q=\left\lceil \frac{Cn_{j}\log n_{j}}{\log\lambda^{-1}}\right\rceil $
and assume first that
\[
H(\mu_{\lambda}^{(q)};q^{-Cq})\geq q\log\lambda^{-1}(\dim\mu_{\lambda}+2\e).
\]
In this case we, set $n_{j+1}=q$. Note that from $\dim\mu_{\lambda}<1$
and \cite{Hoc-self-similar}*{Theorem 1.3},
\begin{equation}
\lim_n\frac{H(\mu_{\lambda}^{(n)};\lambda^{10n})}{n\log\lambda^{-1}}=\dim\mu_{\lambda}.\label{eq:follows from =00005BHo=00005D}
\end{equation}
Thus, by assuming that $n_{0}$ is large enough,
\begin{equation}
H(\mu_{\lambda}^{(n_{j+1})};n_{j+1}^{-Cn_{j+1}}|\lambda^{10n_{j+1}})\ge\e n_{j+1}\log\lambda^{-1}.\label{eq:lb in first case}
\end{equation}

Next suppose that
\[
H(\mu_{\lambda}^{(q)};q^{-Cq})<q\log\lambda^{-1}(\dim\mu_{\lambda}+2\e).
\]
By \eqref{eq:restriction on epsilon},
\[
\log\lambda^{-1}(\dim\mu_{\lambda}+2\e)\le\log\lambda^{-1}(s(\l)-\e)\le H(p)-\e\log\lambda^{-1}.
\]
From this, by Proposition \ref{prop:initial approx integer trans},
and by assuming that $C$ and $n_{0}$ are large enough (with respect
to the specified parameters), it follows that there exists $0<\eta<1$,
which is a root of a nonzero polynomial in $\mathcal{P}_{L_0}^{(q)}$,
such that $|\lambda-\eta|<q^{-C^{1/2}q}$ and
\[
H(\mu_{\eta}^{(q)})\le H(\mu_{\lambda}^{(q)};q^{-Cq})<q\log\lambda^{-1}(\dim\mu_{\lambda}+2\e).
\]

Note that we may assume that $n_{0}$ is sufficiently large so that
\[
H(\mu_{\eta}^{(q)})<q\log\eta^{-1}(\dim\mu_{\lambda}+3\e).
\]
Hence by Lemma \ref{lem:ub on dim},
\[
\dim\mu_{\eta}\le\frac{H(\mu_{\eta}^{(q)})}{q\log\eta^{-1}}<\dim\mu_{\lambda}+3\e,
\]
which implies $\eta\in E_{\dim\mu_{\lambda}+3\e}^{(q)}$.

We take $n_{j+1}$ to be the largest integer $n$ with
$|\lambda-\eta|<n^{-C^{1/2}n}$.
In particular, we have $n_{j+1}\ge q$.
Since
\[
(n_{j+1}+1)^{-C^{1/2}(n_{j+1}+1)}\le|\lambda-\eta|<n_{j+1}^{-C^{1/2}n_{j+1}},
\]
by Proposition \ref{prop:cond for large ent int trans} and by assuming
that $C$ and $n_{0}$ are large enough,
\[
H(\mu_{\lambda}^{(n_{j+1})};(n_{j+1}+1)^{-C(n_{j+1}+1)})=n_{j+1}H(p).
\]
From this, Lemma \ref{diff of ent} and \eqref{eq:follows from =00005BHo=00005D},
it follows that we may assume
\[
H(\mu_{\lambda}^{(n_{j+1})};n_{j+1}^{-Cn_{j+1}}|\lambda^{10n_{j+1}})\ge n_{j+1}(H(p)-\e)-n_{j+1}\log\lambda^{-1}(\dim\mu_{\lambda}+\frac{\e}{\log\lambda^{-1}}).
\]
Thus from \eqref{eq:restriction on epsilon},
\begin{equation}\label{eq:nj+1-second-case}
H(\mu_{\lambda}^{(n_{j+1})};n_{j+1}^{-Cn_{j+1}}|\lambda^{10n_{j+1}})\ge\e n_{j+1}.
\end{equation}

Fix a large integer $N=N(C,n_{0})\ge1$ to be determined later in
the proof. For $1\le j\le N$ set $K_{j}=\frac{C\log n_{j}}{\log\lambda^{-1}}$.
By \eqref{eq:lb in first case} and \eqref{eq:nj+1-second-case}, it follows that
\[
H(\mu_{\lambda}^{(n_{j})};\lambda^{K_{j}n_{j}}|\lambda^{10n_{j}})\ge\e n_{j}\min\{1,\log\lambda^{-1}\}.
\]
We also have $\lambda^{-n_{1}}\ge\max\{2,\lambda^{-2}\}$, $n_{j}\ge C(\log K_{j})^{2}$
and
\[
n_{j+1}\ge\left\lceil \frac{Cn_{j}\log n_{j}}{-\log\lambda}\right\rceil \ge K_{j}n_{j},
\]
if $n_{0}$ is assumed to be large enough with respect to $C$ and
$\lambda$. From all of this together with Lemma \ref{lem:bd of ent of single dig},
it follows that the conditions of Proposition \ref{prop:inequality involving =00007BK_j=00007D}
are satisfied. Thus, by assuming that $C$ is large enough, we get
\begin{equation}
\sum_{j=1}^{N}\frac{1}{\log K_{j}\log\log K_{j}}\le C\left(1+\frac{1}{n_{1}}\sum_{j=1}^{N}\log K_{j}\right).\label{eq:by inequality involving =00007BK_j=00007D}
\end{equation}

Next we estimate how fast the sequence $\{n_{j}\}_{j\ge0}$ may grow.
Let $j\ge0$ and set $q=\left\lceil \frac{Cn_{j}\log n_{j}}{\log\lambda^{-1}}\right\rceil$ as before.
Recall that in the definition of $n_{j+1}$, if the first alternative
occurred then we took $n_{j+1}=q$. Suppose next that the second alternative
has occurred, and let $\eta$ be as obtained during the definition
of $n_{j+1}$. Recall that we have selected $n_{j+1}$ so that $|\lambda-\eta|<n_{j+1}^{-C^{1/2}n_{j+1}}$.
Additionally, from $\eta\in E_{\dim\mu_{\lambda}+3\e}^{(q)}$,
since we may assume $n_{0}\ge\e^{-1}$ and by \eqref{eq:cont assumption},
it follows that $|\lambda-\eta|>\exp(-q^{\e^{-1}})$. Thus,
if $n_{0}$ is large enough, then
\[
n_{j+1}<\log\left(n_{j+1}^{C^{1/2}n_{j+1}}\right)<-\log|\lambda-\eta|<q^{\e^{-1}}<n_{j}^{2\e^{-1}}.
\]
Now by induction it follows that $n_{j}\le n_{0}^{2^{j}\e^{-j}}$
for all $j\ge0$.

The rest of the proof is a simple but long calculation estimating both sides of \eqref{eq:by inequality involving =00007BK_j=00007D}
using primarily the above bounds on the growth of $n_j$, which in the end leads to a contradiction.
This calculation is identical to the end of the proof of Theorem \ref{th:BV2-general} given in Section \ref{sc:BV2-general-proof},
and we do not repeat it here.
\end{proof}

\subsection{Proof of Conjecture \ref{cn:EO} in the case of integral translations}\label{sc:EO int trans}

The purpose of this section is to prove Theorem \ref{th:integral-translations} for transcendental $\l$
generalizing the main result of \cite{Var-Bernoulli}.
We restate it as follows.

\begin{thm}
\label{thm:ver conj for int trans}We have $\dim\mu_{\lambda}=s(\lambda)$
for all transcendental $\lambda\in(0,1)$.
\end{thm}

The following theorem follows directly from
results found in \cite{Hoc-self-similar}, as deduced in \cite{BV-entropy}*{Section 3.4}.
\begin{thm}
\label{thm:formula for dim alg param}Let $\eta\in(0,1)$ be algebraic,
then
\[
\dim\mu_{\eta}=\min\{1,\frac{h(\eta)}{\log\eta^{-1}}\}.
\]
\end{thm}

The following theorem follows directly from \cite{BV-entropy}*{Proposition 13}.
For the details we refer to the proof of \cite{Var-Bernoulli}*{Theorem 9}.
\begin{thm}
\label{thm:large mahler --> large ent}Let $h\in(0,H(p))$.
There
exists $M>1$ depending only on $h$, $a_1,\ldots,a_m$ and $p_1,\ldots,p_m$,
such that $h(\lambda)>h$ for all algebraic $\lambda\in(0,1)$
with $M(\lambda)\ge M$.
\end{thm}

Recall the definition of the number $L_{0}$ from the beginning of Section \ref{sc:gen of BV2 int trans}.
The following theorem follows directly from \cite{Hoc-self-similar}*{Theorem 1.1}.
\begin{thm}
\label{thm:follows from =00005BHo=00005D}Let $\lambda\in(0,1)$ be
with $\dim\mu_{\lambda}<s(\lambda)$. Then for every $\theta>0$
there exists $N\ge1$, such that for every $n\ge N$ there exists
$0\ne P\in\mathcal{P}_{L_0}^{(n)}$ with $|P(\lambda)|<\theta^{n}$.
\end{thm}

\begin{proof}[Proof of Theorem \ref{thm:ver conj for int trans}]
Let $\lambda\in(0,1)$ be transcendental, and suppose to the contrary
that $\dim\mu_{\lambda}<s(\lambda)$. Let 
\[
0<\e<\frac{1}{3}(s(\lambda)-\dim\mu_{\lambda}),
\]
let $M\in\Z_{>0}$ be large with respect to $a_1,\ldots,a_m$, $p_1,\ldots,p_m$, $\lambda$ and $\e$,
and let $q_{0}\ge1$ be large with respect to $M$.

By Theorem \ref{thm:gen of BV2 main thm}, there
exist an integer $q\ge q_{0}$ and an algebraic $\eta\in(0,1)$, such
that $\deg\eta<q$,
\[
\dim\mu_{\eta}<\dim\mu_{\lambda}+\e<s(\lambda)-2\e,
\]
and $|\lambda-\eta|<2^{-q^{2}}$. We may assume that $q$ is
sufficiently large so that $|s(\lambda)-s(\eta)|<\e$,
and so $\dim\mu_{\eta}<s(\eta)-\e$.

By Theorem \ref{thm:formula for dim alg param},
\[
\dim\mu_{\eta}=\min\{1,\frac{h(\eta)}{\log\eta^{-1}}\}.
\]
From this, $\dim\mu_{\eta}<s(\eta)-\e$ and Theorem
\ref{thm:large mahler --> large ent}, it follows that we may assume
that $M(\eta)<M$.

Now let $n\ge1$ be with $(2M)^{-n-1}\le|\l-\eta|\le(2M)^{-n}$. We have,
\[
(n+1)\log(2M)\ge-\log|\l-\eta|>q^{2}\log2.
\]
Thus we may assume,
\begin{equation}\label{cond for Ves lem}
\frac{1}{n}\Big(\big(q(M+1)+(M+2)\big)\log n+(q+1)\log L_0+\log2\Big)<\log2,
\end{equation}
so Lemma \ref{lm:Dimitrov-precise} can be applied with $\a=2$.
Additionally, by Theorem \ref{thm:follows from =00005BHo=00005D} and since $n$ is arbitrarily large with respect to $M$, we may assume that there exists $0\ne P\in\cP_{L_0}^{(n)}$ such that $|P(\l)|\le (2M)^{-3Mn}$, which gives
\[
|\l-\eta|\ge(2M)^{-n-1}\ge (2M)^n|P(\l)|^{1/M}.
\]
From this, $|\l-\eta|\le(2M)^{-n}$, $\deg\eta<q$, \eqref{cond for Ves lem} and Lemma \ref{lm:Dimitrov-precise}, it follows that $\eta$ is a zero of $P$ of order at least $M$.

Now, by assuming $\eta<(1+\l)/2$ and that $M$ is sufficiently large with respect to $L_0$ and $\l$, we get a contradiction with Lemma  \ref{lem:bound on num of roots}. This completes the proof of the theorem.
\end{proof}

\bibliography{bibfile}

\end{document}